\documentclass{article}%
\usepackage{amsmath}
\usepackage{amsfonts}
\usepackage{amssymb}
\usepackage{graphicx}
\usepackage{color}%
\providecommand{\U}[1]{\protect\rule{.1in}{.1in}}
\newtheorem{theorem}{Theorem}

\newtheorem{definition}[theorem]{Definition}

\newtheorem{lemma}[theorem]{Lemma}

\newtheorem{proposition}[theorem]{Proposition}
\newtheorem{remark}[theorem]{Remark}

\newenvironment{proof}[1][Proof]{\noindent\textbf{#1.} }{\ \rule{0.5em}{0.5em}}
\def \RN {\mathbb{R}^N}
\def \N {\mathbb{N}}
\def \R {\mathbb{R}}

\def \e {\varepsilon}

\def \de {\partial}
\def \LL {\mathcal{L}}

\def \bd {\boldsymbol}
\def \G {\mathbb{G}}
\numberwithin{equation}{section}
\numberwithin{theorem}{section}
\begin{document}

\title{KFP operators with coefficients measurable in time and Dini continuous in space}
\author{S. Biagi, M. Bramanti, B. Stroffolini}
\maketitle

\begin{abstract}
We consider degenerate Kolmogorov-Fokker-Planck operators%
\begin{align*}
\mathcal{L}u  &  =\sum_{i,j=1}^{m_{0}}a_{ij}(x,t)\partial_{x_{i}x_{j}}%
^{2}u+\sum_{k,j=1}^{N}b_{jk}x_{k}\partial_{x_{j}}u-\partial_{t}u\\
&  \equiv\sum_{i,j=1}^{m_{0}}a_{ij}(x,t)\partial_{x_{i}x_{j}}^{2}u+Yu
\end{align*}
(with $(x,t)\in\mathbb{R}^{N+1}$ and $1\leq m_{0}\leq N$) such that the
corresponding mo\-del operator having constant $a_{ij}$ is hypoelliptic,
translation invariant w.r.t.\thinspace a Lie group operation in $\mathbb{R}%
^{N+1}$ and $2$-homogeneous w.r.t.\thinspace a family of non\-i\-so\-tro\-pic
dilations. The matrix $(a_{ij})_{i,j=1}^{m_{0}}$ is symmetric and uniformly
positive on $\mathbb{R}^{m_{0}}$. The coefficients $a_{ij}$ are bounded and
\emph{Dini continuous in space}, and only bounded measurable in time. This
means that, setting
\begin{align*}
\mathrm{i)} & \,\,S_{T}=\mathbb{R}^{N}\times\left(  -\infty,T\right) ,\\
\mathrm{ii)} & \,\,\omega_{f,S_{T}}(r) = \sup_{\substack{(x,t),(y,t)\in
S_{T}\\\| x-y\| \leq r}}\vert f(x,t) -f(y,t)\vert\\
\mathrm{iii)} & \,\,\Vert f\Vert_{\mathcal{D}( S_{T}) } =\int_{0}^{1}%
\frac{\omega_{f,S_{T}}(r) }{r}dr+\Vert f\Vert_{L^{\infty}\left(  S_{T}\right)
}%
\end{align*}
we require the finiteness of $\Vert a_{ij}\Vert_{\mathcal{D}(S_{T})}$. We
bound $\omega_{u_{x_{i}x_{j}},S_{T}}$, $\Vert u_{x_{i}x_{j}}\Vert_{L^{\infty}(
S_{T}) }$ ($i,j=1,2,...,m_{0}$), $\omega_{Yu,S_{T}}$, $\Vert Yu\Vert
_{L^{\infty}( S_{T}) }$ in terms of $\omega_{\mathcal{L}u,S_{T}}$,
$\Vert\mathcal{L}u\Vert_{L^{\infty}( S_{T}) }$ and $\Vert u\Vert_{L^{\infty
}\left(  S_{T}\right)  }$, getting a control on the uniform continuity in
space of $u_{x_{i}x_{j}},Yu$ if $\mathcal{L}u$ is bounded and Dini-continuous
in space. Under the additional assumption that both the coefficients $a_{ij}$
and $\mathcal{L}u$ are log-Dini continuous, meaning the finiteness of the
quantity%
\[
\int_{0}^{1}\frac{\omega_{f,S_{T}}\left(  r\right)  }{r}\left\vert \log
r\right\vert dr,
\]
we prove that $u_{x_{i}x_{j}}$ and $Yu$ are Dini continuous; moreover, in this
case, the derivatives $u_{x_{i}x_{j}}$ are locally uniformly continuous in
space \emph{and time}.

\end{abstract}

\section{Introduction and statement of the main result}

In this paper, we will be concerned with \emph{Kol\-mo\-go\-rov-Fokker-Planck}
(KFP, in short) operators of the form
\begin{equation}
\mathcal{L}u=\sum_{i,j=1}^{m_{0}}a_{ij}(x,t)\partial_{x_{i}x_{j}}^{2}%
u+\sum_{k,j=1}^{N}b_{jk}x_{k}\partial_{x_{j}}u-\partial_{t}u,\qquad
(x,t)\in\mathbb{R}^{N+1}, \label{L}%
\end{equation}
where $1\leq m_{0}\leq N$. The first-order part of the operator $\mathcal{L}$,
also called \emph{the drift term}, is a smooth vector field which will be
denoted by $Y$; more explicitly,
\begin{equation}
Yu=\sum_{k,j=1}^{N}b_{jk}x_{k}\partial_{x_{j}}u-\partial_{t}u.
\label{eq:driftY}%
\end{equation}
Points of $\mathbb{R}^{N+1}$ will be sometimes denoted by the compact
notation
\[
\xi=(x,t),\,\eta=(y,s).
\]
Given $T\in\mathbb{R}$, we set
\[
S_{T}:=\mathbb{R}^{N}\times(-\infty,T).
\]
We will make the following assumptions on $\mathcal{L}$:

\begin{itemize}
\item[\textbf{(H1)}] $A_{0}(x,t)=(a_{ij}(x,t))_{i,j=1}^{m_{0}}$ is a
symmetric, uniformly positive matrix on $\mathbb{R}^{m_{0}}$ of bounded
measurable coefficients, defined in $\mathbb{R}^{N+1}$; more precisely, there
exists a constant $\nu>0$ such that
\begin{equation}
\begin{gathered} \nu|v|^{2}\leq\sum_{i,j=1}^{m_0}a_{ij}(x,t)v_{i}v_{j}\leq\nu^{-1}|v|^{2} \\ \text{for every $v\in\R^{m_0},\,x\in\R^N$ and a.e.\,$t\in\R$}. \end{gathered} \label{nu}%
\end{equation}
The co\-ef\-fi\-cients will be also assumed to be Dini continuous
w.r.t.\thinspace$x$, uniformly w.r.t.\thinspace$t$. This assumption will be
specified later (see Definition \ref{Def Dini} and assumption (H3)), since it
requires some preliminaries.

\item[\textbf{(H2)}] The matrix $B=(b_{ij})_{i,j=1}^{N}$ satisfies the
following condition: for $m_{0}$ and suitable positive integers $m_{1}%
,\dots,m_{k}$ such that
\begin{equation}
m_{0}\geq m_{1}\geq\ldots\geq m_{k}\geq1\quad\mathrm{and}\quad m_{0}%
+m_{1}+\ldots+m_{k}=N, \label{m-cond}%
\end{equation}
we have
\begin{equation}
B=%
\begin{pmatrix}
\mathbb{O} & \mathbb{O} & \ldots & \mathbb{O} & \mathbb{O}\\
B_{1} & \mathbb{O} & \ldots & \ldots & \ldots\\
\mathbb{O} & B_{2} & \ldots & \mathbb{O} & \mathbb{O}\\
\vdots & \vdots & \ddots & \vdots & \vdots\\
\mathbb{O} & \mathbb{O} & \ldots & B_{k} & \mathbb{O}%
\end{pmatrix}
\label{B}%
\end{equation}
where $B_{j}$ is an $m_{j}\times m_{j-1}$ matrix of rank $m_{j}$ (for
$j=1,2,\ldots,k$).
\end{itemize}

To the best of our knowledge, the study of the KFP operators has a long
history which dates back to the 1934 paper by Kolmogorov \cite{Kolmo} on the
Theory of Gases. In this paper, Kolmogorov introduced the operator
\[
\mathcal{K}=\Delta_{u}+\langle u,\nabla_{v}\rangle-\partial_{t},\quad
\text{with $u,v\in\mathbb{R}^{n}$ and $t\in\mathbb{R}$},
\]
which can be obtained from \eqref{L} by choosing
\[
N=2n,\quad m_{0}=m_{1}=n,\quad A_{0}=\mathrm{Id}_{n},\quad,\quad B=%
\begin{pmatrix}
\mathbb{O}_{n} & \mathbb{O}_{n}\\
\mathrm{Id}_{n} & \mathbb{O}_{n}%
\end{pmatrix}
.
\]
It should be noticed that, since $m_{0}<N$, the operator $\mathcal{K}$ is
\emph{not parabolic}; however, Kolmogorov proved in \cite{Kolmo} that
$\mathcal{K}$ is $C^{\infty}$-hypoelliptic in $\mathbb{R}^{2n}$ by
constructing an explicit smooth fundamental solution. The (global) $C^{\infty
}$-hypoellipticity of the operator $\mathcal{K}$ is cited by H\"{o}rmander as
one of the main `inspiration' for his celebrated work \cite{Horm} on the
hypoellipticity of the \emph{sums of squares of vector fields} (plus a drift),
of which the KFP operators with constant coefficients $a_{i,j}$'s are a
particular case.

After the seminal paper by H\"{o}rmander, the KFP operators \emph{with
constant coefficients} have been studied by many authors, and from several
point of views; in particular, at the beginning of the '90s' Lanconelli and
Polidoro \cite{LP} started the study of constant coefficients KFP operators
from a \emph{geometrical viewpoint}, showing that these operators possess a
rich underlying \emph{subelliptic geometric structure}. More precisely, they
proved that the $m_{0}+1$ vector fields
\[
X_{1}=\partial_{x_{1}},\ldots,X_{m_{0}}=\partial_{x_{m_{0}}},\,Y=\sum
_{k,j=1}^{N}b_{jk}x_{k}\partial_{x_{j}}-\partial_{t}%
\]
(on which the KFP operators \eqref{L} are modeled) satisfy the following properties:

\begin{itemize}
\item[(a)] $X_{1},\ldots,X_{m_{0}},\,Y$ are left-invariant on the \emph{Lie
group} $\G=(\mathbb{R}^{N+1},\circ)$, where the (non-commutative) composition
law $\circ$ is defined as follows
\begin{align*}
(y,s)\circ(x,t)  &  =(x+E(t)y,t+s)\\
(y,s)^{-1}  &  =(-E(-s)y,-s),
\end{align*}
and $E(t)=\exp(-tB)$ (which is defined for every $t\in\mathbb{R}$ since the
matrix $B$ is nilpotent). For a future reference, we explicitly notice that
\begin{equation}
(y,s)^{-1}\circ(x,t)=(x-E(t-s)y,t-s), \label{eq:convolutionG}%
\end{equation}
and that the Lebesgue measure is the Haar measure, which is also invariant
with respect to the inversion. \vspace*{0.1cm}

\item[(b)] $X_{1},\ldots,X_{m_{0}}$ are homogeneous of degree $1$ and $Y$ is
homogeneous of degree $2$ with respect to a nonisotropic family of
\emph{dilations} in $\mathbb{R}^{N+1}$, which are automorphisms of $\G$ and
are defined by
\begin{equation}
D(\lambda)(x,t)\equiv(D_{0}(\lambda)(x),\lambda^{2}t)=(\lambda^{q_{1}}%
x_{1},\ldots,\lambda^{q_{N}}x_{N},\lambda^{2}t), \label{dilations}%
\end{equation}
where the $N$-tuple $(q_{1},\ldots,q_{N})$ is given by
\[
(q_{1},\ldots,q_{N})=(\underbrace{1,\ldots,1}_{m_{0}},\,\underbrace{3,\ldots
,3}_{m_{1}},\ldots,\underbrace{2k+1,\ldots,2k+1}_{m_{k}}).
\]
The integer%
\begin{equation}
Q=\sum_{i=1}^{N}q_{i}>N \label{eq:defQhomdim}%
\end{equation}
is called the (spatial)\emph{ homogeneous dimension} of $\mathbb{R}^{N}$,
while $Q+2$ is the homogeneous dimension of $\mathbb{R}^{N+1}$. We explicitly
point out that the exponential matrix $E(t)$ satisfies the following
homogeneity property
\begin{equation}
E(\lambda^{2}t)=D_{0}(\lambda)E(t)D_{0}\Big(\frac{1}{\lambda}\Big),
\label{LP 2.20}%
\end{equation}
for every $\lambda>0$ and every $t\in\mathbb{R}$ (see \cite[Rem.\,2.1.]{LP}).

\item[(c)] $X_{1},\ldots,X_{m_{0}},\,Y$ satisfy the \emph{H\"{o}r\-man\-der
Rank Condition} in $\mathbb{R}^{N+1}$.
\end{itemize}

Through the years, many Authors have studied KFP operators with
\emph{variable} coefficients $a_{ij}\left(  x,t\right)  $, modeled on the
above class of left invariant hypoelliptic operators. For instance, Schauder
estimates on bounded domains have been investigated first by Manfredini,
\cite{Ma}, and later by Di Francesco-Polidoro in \cite{DP} under more general
assumptions, assuming the coefficients $a_{ij}$ H\"{o}lder continuous with
respect to the intrinsic distance induced in $\mathbb{R}^{N+1}$ by the vector
fields $\partial_{x_{1}},...\partial_{x_{m_{0}}},Y$. With regards to Schauder
estimates for KFP operators, the reader is referred also to the papers by
Lunardi \cite{Lu}, Priola \cite{Pr}, Imbert-Mouhot \cite{IM}, Wang-Zhang
\cite{WZ}, and the references therein. Also, continuity estimates on
$u_{x_{i}x_{j}}$ under a Dini continuity assumption on $a_{ij}$ and
$\mathcal{L}u$ have been proved by Polidoro, Rebucci, Stroffolini in
\cite{PRS}.

Recent contributions from the field of stochastic differential equations (see
e.g. \cite{PP}) suggest the importance of developing a theory allowing the
coefficients $a_{ij}$ to be rough in $t$ (say, $L^{\infty}$), and uniformly
continuous (for instance, H\"{o}lder continuous) only w.r.t. the space
variables. The Schauder estimates that one can reasonably expect under this
mild assumption consist in controlling the H\"{o}lder seminorms w.r.t. $x$ of
the derivatives involved in the equations, uniformly in time. These estimates
are sometimes called \textquotedblleft partial Schauder
estimates\textquotedblright. Similar results can be expected when H\"{o}lder
continuity is replaced by Dini continuity. Results of this kind (in the
H\"{o}lder case) are well-known for uniformly parabolic operators (see
\cite{Br69}, \cite{K80}, and more recent papers quoted in the references in
\cite{BB}). Also, in the parabolic case, it is known that $u_{x_{i}x_{j}}$
satisfy a continuity estimate \emph{in time}, under the same assumptions of
continuity in space of $a_{ij}$ and $\mathcal{L}u$. Partial Schauder estimates
for $u_{x_{i}x_{j}},Yu$, together with local H\"{o}lder continuity in the
joint variables, have been recently proved by the first two of us in
\cite{BB}. Partial Schauder estimates for degenerate KFP operators have been
proved also in the recent paper \cite{CRHM} by Chaudru de Raynal, Honor\'{e},
Menozzi, with different techniques and without getting the H\"{o}lder control
in time of second order derivatives. We also quote the preprint \cite{LPP}, by
Lucertini, Pagliarani, Pascucci, dealing with the construction of a
fundamental solution for KFP operators with coefficients H\"{o}lder continuous
in space and $L^{\infty}$ in time.

In this paper, we address the problem of proving uniform continuity estimates
w.r.t. the space variables on $u_{x_{i}x_{j}},Yu$, assuming $a_{ij}$ and
$\mathcal{L}u$ to be Dini continuous w.r.t. the space variables, uniformly in
$t$. We prove an estimate of this kind, which, in turn, implies the (partial)
Dini continuity of $u_{x_{i}x_{j}},Yu$ under the stronger assumption that
$a_{ij}$ and $\mathcal{L}u$ are log-Dini continuous w.r.t. the space
variables, uniformly in $t$ (for the precise statement, see Theorem
\ref{Thm main space}). These results are consistent with those proved in
\cite{PRS} when $a_{ij}$ and $\mathcal{L}u$ are Dini-continuous in the joint
variables. Moreover, under the same stronger assumption of log-Dini continuity
of $a_{ij}$ and $\mathcal{L}u$, we prove a bound on the modulus of continuity
in the joint variables for $u_{x_{i}x_{j}}$, analogously to what happens in
the H\"{o}lder case. (For the exact statement, see Theorem \ref{Thm main time}).

\bigskip

\noindent\emph{Statement of the main result.} In order to introduce the
function spaces and the quantities which will be involved in the statements of
our results, we need to introduce some metric notions. First of all, the
system
\[
\mathbf{X}=\{X_{1},\ldots,X_{m_{0}},Y\}
\]
induces in a standard way a (weighted) control distance $d_{\mathbf{X}}$ in
$\mathbb{R}^{N+1}$, which is left invariant w.r.t. group operation
\thinspace$\circ$ and jointly $1$-homogeneous with respect to $D(\lambda)$. As
a consequence, the function $\rho_{\mathbf{X}}(\xi):=d_{\mathbf{X}}(\xi,0)$
satisfies \medskip

(1)\,\,$\rho_{\mathbf{X}}(\xi^{-1}) = \rho_{\mathbf{X}}(\xi)$; \vspace
*{0.05cm}

(2)\,\,$\rho_{\mathbf{X}}(\xi\circ\eta) \leq\rho_{\mathbf{X}}(\xi)
+\rho_{\mathbf{X}}(\eta)$. \medskip

\noindent(For these and related basic notions on H\"{o}rmander vector fields,
we refer to \cite[Chaps. 1-3]{BBbook}). In addition, since $d_{\mathbf{X}}$ is
a distance, we also have \medskip

(1)'\,\,$\rho_{\mathbf{X}}(\xi)\geq0$ and $\rho_{\mathbf{X}}(\xi
)=0\,\Leftrightarrow\,\xi=0$; \vspace*{0.05cm}

(2)'\,\,$\rho_{\mathbf{X}}(D(\lambda)\xi) = \lambda\rho_{\mathbf{X}}(\xi)$,
\medskip

\noindent and this means that $\rho_{\mathbf{X}}$ is a \emph{homogeneous norm}
in $\mathbb{R}^{N+1}$. We then notice that, owing to the explicit expression
of $D(\lambda)$ in \eqref{dilations}, the function
\begin{equation}
\rho(\xi)=\rho(x,t):=\Vert x\Vert+\sqrt{|t|}=\sum_{i=1}^{N}|x_{i}|^{1/q_{i}%
}+\sqrt{|t|} \label{eq:defrhonorm}%
\end{equation}
is also a homogeneous norm in $\mathbb{R}^{N+1}$, and therefore, it is
\emph{globally equivalent} to the norm $\rho_{\mathbf{X}}$. As a consequence,
the map
\begin{equation}
d(\xi,\eta):=\rho(\eta^{-1}\circ\xi)=\Vert x-E(t-s)y\Vert+\sqrt{|t-s|}
\label{d}%
\end{equation}
is a left-invariant, $1$-homogeneous \emph{quasi-distance} on $\mathbb{R}%
^{N+1}$. More precisely, there exists $\bd{\kappa}\geq1$ such that
\begin{align}
d(\xi,\eta)  &  \leq\bd{\kappa}\big(d(\xi,\zeta)+d(\eta,\zeta)\big)\qquad
\forall\,\,\xi,\eta,\zeta\in\mathbb{R}^{N+1};\label{eq:quasitriangled}\\
d(\xi,\eta)  &  \leq\bd{\kappa}\,d(\eta,\xi)\qquad\forall\,\,\xi,\eta
\in\mathbb{R}^{N+1}. \label{eq:quasisymd}%
\end{align}

The quasi-distance $d$ is \emph{globally equivalent} to the control distance
$d_{\mathbf{X}}$; hence, we will systematically use this quasi-distance $d$ in
place of $d_{\mathbf{X}}$. We refer the reader to Section
\ref{sec:preliminaries} for several properties of $d$ which shall be used in
the paper. \vspace*{0.1cm}

Using the quasi-distance $d$, we now introduce the relevant spa\-ces of
functions to which our main result applies.

\begin{definition}
[H\"{o}lder continuous functions]\label{def:Holderspacesd} Let $\Omega
\subseteq\mathbb{R}^{N+1}$ be an open set, and let $\alpha\in(0,1)$. Given a
fun\-ction $f:\Omega\rightarrow\mathbb{R}$, we introduce the notation
\[
|f|_{C^{\alpha}(\Omega)}=\sup\left\{  \frac{|f(\xi)-f(\eta)|}{d(\xi
,\eta)^{\alpha}}:\,\text{$\xi,\eta\in\Omega$ and $\xi\neq\eta$}\right\}  .
\]
Accordingly, we define the space $C^{\alpha}(\Omega)$ as follows:
\[
C^{\alpha}(\Omega):=\{f\in C(\Omega)\cap L^{\infty}(\Omega):\,|f|_{C^{\alpha
}(\Omega)}<\infty\}.
\]
Finally, on this space $C^{\alpha}(\Omega)$ we introduce the norm
\[
\Vert f\Vert_{C^{\alpha}(\Omega)}:=\Vert f\Vert_{L^{\infty}(\Omega
)}+|f|_{C^{\alpha}(\Omega)}.
\]

\end{definition}

\begin{definition}
[Partially Dini and log-Dini continuity]\label{Def Dini} Let $\Omega$ be an
arbitrary open set in $\mathbb{R}^{N+1}$, and let $f\in L^{\infty}(\Omega)$.
For every $r>0$, we set
\[
\omega_{f,\Omega}(r)=\sup_{\begin{subarray}{c}
(x,t),(y,t)\in\Omega \\
\|x-y\| \leq r
\end{subarray}}|f(x,t)-f(y,t)|.
\]
We then say that \medskip

(i)\,\,$f$ is \emph{partially Dini-con\-ti\-nuo\-us in $\Omega$}, and we write
$f\in\mathcal{D}(\Omega)$, if
\begin{equation}
\label{Dini}\int_{0}^{1}\frac{\omega_{f,\Omega}\left(  r\right)  }{r}%
dr<\infty;
\end{equation}

(ii)\,\,$f$ is \emph{partially log-Dini continuous}, and we write
$f\in\mathcal{D}_{\log}(\Omega)$, if
\begin{equation}
\int_{0}^{1}\frac{\omega_{f,\Omega}(r)}{r}|\log r| dr<\infty. \label{log Dini}%
\end{equation}
If $f\in\mathcal{D}(\Omega)$, we define
\[
|f|_{\mathcal{D}(\Omega)} = \int_{0}^{1}\frac{\omega_{f,\Omega}(r)}%
{r}\,dr\quad\text{and}\quad\|f\|_{\mathcal{D}(\Omega)} = \|f\|_{L^{\infty
}(\Omega)}+|f|_{\mathcal{D}(\Omega)}.
\]

\end{definition}

\begin{remark}
\label{rem:funzioniOmega} Let $\Omega\subseteq\mathbb{R}^{N+1}$ be an open
set, and let $f\in\mathcal{D}_{\log}(\Omega)$. We will see in Section
\ref{sec:preliminaries} that the following functions are \emph{well-defined
moduli of continuity} {(}that is, non-decreasing functions on $(0,\infty)$
vanishing for $r\rightarrow0^{+}${)}:
\begin{align}
&  \mathcal{M}_{f,\Omega}(r)=\omega_{f,\Omega}(r)+\int_{0}^{r}\frac
{\omega_{f,\Omega}(s)}{s}\,ds+r\int_{r}^{\infty}\frac{\omega_{f,\Omega}%
(s)}{s^{2}}\,ds;\label{eq:defcontinM}\\
&  {\mathcal{N}}_{f,\Omega}(r)=\mathcal{M}_{f,\Omega}(r)+\int_{0}^{r}%
\frac{\mathcal{M}_{f,\Omega}(s)}{s}\,ds+r\int_{r}^{\infty}\frac{\mathcal{M}%
_{f,\Omega}(s)}{s^{2}}\,ds. \label{eq:defcontinN}%
\end{align}
Furthermore, given any $\mu>0$, we will see that also the functions
\begin{align}
\mathcal{U}_{f,\Omega}^{\mu}(r)  &  =\int_{\RN}e^{-\mu|z|^{2}}\Big(\int%
_{0}^{r\Vert z\Vert}\frac{\omega_{f,\Omega}(s)}{s}%
\,ds\Big)dz\label{eq:defUmuIntro}\\
\mathcal{V}_{f,\Omega}^{\mu}(r)  &  =\int_{\RN}e^{-\mu|z|^{2}}\Big(\int%
_{0}^{r\Vert z\Vert}\frac{\mathcal{M}_{f,\Omega}(s)}{s}\,ds\Big)dz
\label{eq:defVmuIntro}%
\end{align}
are well-defined on the interval $(0,\infty)$. The continuity estimates
appearing in our ma\-in results, namely Theorems \ref{Thm main space}%
-\ref{Thm main time}, will depend on these functions.
\end{remark}

\begin{definition}
\label{def:spacesS} Given any number $T > 0$, we define $\mathcal{S}^{0}%
(S_{T})$ as the space of all fun\-cti\-ons $u:\overline{S}_{T}\rightarrow
\mathbb{R}$ satisfying the following properties: \medskip

(i)\,\,$u\in C(\overline{S_{T}})\cap L^{\infty}(S_{T})$;

(ii)\,\,for every $1\leq i,j\leq m_{0}$, $\partial_{x_{i}}u,\partial
_{x_{i}x_{j}}^{2}u\in L^{\infty}(S_{T})$;

(iii)\,\,$Yu\in L^{\infty}(S_{T})$ \medskip

\noindent{(}in the above (ii)-(iii), the derivatives $\partial_{x_{i}%
}u,\,\partial_{x_{i}x_{j}}^{2}u$ and $Yu$ are intended in the sense of
di\-stri\-bu\-tions{)}. For every fixed $\tau<T$, we also define
\[
S^{0}(\tau,T)=\{f\in\mathcal{S}^{0}(S_{T}):\mathrm{supp}(f)\subset
\mathbb{R}^{N}\times(\tau,T)\}.
\]
Finally, we define $\mathcal{S}^{D}(S_{T})$ as the space of functions
$u\in\mathcal{S}^{0}(S_{T})$ such that
\[
\text{$\partial_{x_{i}}u,\,\partial_{x_{i}x_{j}}^{2}u\in\mathcal{D}(S_{T})$
(for $i,j=1,2,...,m_{0}$)\quad and \quad$Yu\in\mathcal{D}(S_{T})$}.
\]

\end{definition}

\begin{remark}
If $u\in\mathcal{S}^{0}(S_{T})$ then $u$ and $\partial_{x_{1}}u,...,\partial
_{x_{m_{0}}}u$ belong to $C^{\alpha}\left(  S_{T}\right)  $ for every
$\alpha\in\left(  0,1\right)  $. A quantitative estimate on these H\"{o}lder
norms is proved in Theorem \ref{Thm interpolaz} \cite{BB}, under the
assumption of H\"{o}lder continuity (w.r.t. $x$) of $a_{ij}$ and
$\mathcal{L}u$, while in our main result (Theorem \ref{Thm main space}) this
will be proved under the assumption of partial Dini continuity of $a_{ij}$ and
$\mathcal{L}u$.
\end{remark}

We are now ready to state the main results of this paper.

\begin{theorem}
[Global continuity estimates]\label{Thm main space} Let $\mathcal{L}$ be an
operator as in \eqref{L}, and assume that \emph{(H1), (H2)} are satisfied. In
addition, we assume that \medskip

\textbf{\emph{(H3)}} $a_{ij}\in\mathcal{D}(\mathbb{R}^{N+1})$ for every $1\leq
i,j\leq m_{0}$.

\medskip\noindent Then, for every $1\leq i,j\leq m_{0}$, every $T>0$ and
$\alpha\in(0,1)$ there exists a constant $c>0$, depending on $T$, $\alpha$,
the matrix $B$ in \eqref{B}, the number $\nu$ in \eqref{nu} and on the number
\begin{equation}
A=\textstyle\sum_{i,j=1}^{m_{0}} \Vert a_{ij}\Vert_{\mathcal{D}(\mathbb{R}%
^{N+1})} \label{a cost Dini}%
\end{equation}
such that the following estimates hold for every $u\in\mathcal{S}^{D}(S_{T}%
)$:
\begin{align*}
&  \mathrm{(i)}\,\,\sum_{h,k=1}^{m_{0}}\Vert\partial_{x_{h}x_{k}}^{2}%
u\Vert_{L^{\infty}(S_{T})}+\Vert Yu\Vert_{L^{\infty}(S_{T})}+\sum_{i=1}%
^{m_{0}}\Vert\partial_{x_{i}}u\Vert_{C^{\alpha}(S_{T})}+\Vert u\Vert
_{C^{\alpha}(S_{T})}\\[0.1cm]
&  \leq c\big\{\Vert\mathcal{L}u\Vert_{\mathcal{D}(S_{T})}+\Vert
u\Vert_{L^{\infty}(S_{T})}\big\}\\
&  \mathrm{(ii)}\,\,\sum_{h,k=1}^{m_{0}}\omega_{\partial_{x_{h}x_{k}}%
^{2}u,S_{T}}(r)+\omega_{Yu,S_{T}}(r)\\
&  \qquad\qquad\leq c\big\{\mathcal{M}_{\mathcal{L}u,S_{T}}(cr)+(\mathcal{M}%
_{a,S_{T}}(cr)+r^{\alpha})(\Vert\mathcal{L}u\Vert_{\mathcal{D}(S_{T})}+\Vert
u\Vert_{L^{\infty}(S_{T})})\big\}.
\end{align*}
Here, $\mathcal{M}_{a,S_{T}}=\sum_{i,j=1}^{m_{0}}\mathcal{M}_{a_{ij},S_{T}}$
and $\mathcal{M}_{\cdot,\,S_{T}}$ is as in \eqref{eq:defcontinM}.
\vspace*{0.1cm}

In particular, the functions $\partial_{x_{h}x_{k}}^{2}u, (h,k=1,..., m_{0}) ,
Yu$ are partially Dini continuous if both the coefficients $a_{ij}$ and the
function $\mathcal{L}u$ are partially log-Dini continuous.
\end{theorem}

It is worthwhile noting that the full H\"{o}lder norms of the lower order
terms $u,\partial_{x_{i}}u$ $\left(  i=1,2,...,m_{0}\right)  $ can be bounded
assuming the partial Dini continuity of $\mathcal{L}u$ and the coefficients
$a_{ij}$. In particular, any function in $\mathcal{S}^{D}(S_{T})$ has this
regularity property.

\begin{theorem}
[Continuity estimates in space-time for $\de^{2}_{x_{i}x_{j}}u$]%
\label{Thm main time} Let $\mathcal{L}$ be an operator as in \eqref{L}, and
assume that \emph{(H1), (H2)} are satisfied. In addition, we assume that
\medskip

\textbf{\emph{(H3)'}} $a_{ij}\in\mathcal{D}_{\mathrm{log}}(\mathbb{R}^{N+1})$
for every $1\leq i,j\leq m_{0}$.

\medskip\noindent Then, for every $1\leq i,j\leq m_{0}$, every $-\infty
<\tau<T$, every $\alpha\in(0,1)$ and every \emph{compact set} $K\subseteq
\mathbb{R}^{N}$ there exists a constant ${c>0}$, depending on $K,\tau
,T,\alpha$ and $\nu$, such that
\begin{equation}%
\begin{split}
&  |\partial_{x_{i}x_{j}}^{2}u(x_{1},t_{1})-\partial_{x_{i}x_{j}}^{2}%
u(x_{2},t_{2})|\\
&  \quad\leq c\left\{  \mathcal{N}_{\mathcal{L}u,S_{T}}(cr)+\mathcal{V}%
_{\mathcal{L}u,S_{T}}^{\mu}\left(  c\sqrt{|t_{1}-t_{2}|}\right)  \right. \\
&  \quad\quad\left.  +\left(  \mathcal{N}_{a,S_{T}}(cr)+\mathcal{V}_{a,S_{T}%
}^{\mu}\left(  c\sqrt{|t_{1}-t_{2}|}\right)  +r^{\alpha}\right)  \left(
\Vert\mathcal{L}u\Vert_{\mathcal{D}(S_{T})}+\Vert u\Vert_{L^{\infty}(S_{T}%
)}\right)  \right\}
\end{split}
\label{eq:maincontinuitydexixjtime}%
\end{equation}
for any $u\in\mathcal{S}^{D}(S_{T})$ with $\mathcal{L}u\in\mathcal{D}%
_{\mathrm{log}}(S_{T})$ and any $(x_{1},t_{1}),\,(x_{2},t_{2})\in
K\times\lbrack\tau,T]$. In the above estimate, we have used the notation
\[
r=d(x_{1},t_{1}),(x_{2},t_{2}))+|t_{1}-t_{2}|^{1/q_{N}}%
\]
where $q_{N}\geq3$ is the largest exponent in the dilations $D(\lambda)$, see
\eqref{dilations}; in addition,
\[
\textstyle\mathcal{N}_{a,S_{T}}=\sum_{i,j=1}^{m_{0}}\mathcal{N}_{a_{ij},S_{T}%
},\qquad\mathcal{V}_{a,S_{T}}^{\mu}=\sum_{i,j=1}^{m_{0}}\mathcal{V}%
_{a_{ij},S_{T}}^{\mu},
\]
and $\mathcal{N}_{\cdot,\,S_{T}},\,\mathcal{V}_{\cdot,\,S_{T}}^{\mu}$ are as
in \eqref{eq:defcontinN}-\eqref{eq:defVmuIntro}, respectively \emph{(}and
$\mu>0$ is a constant only depending on $\nu$\emph{)}.
\end{theorem}

More explicit bounds on the functions $\mathcal{U}_{...}^{\mu}\left(
r\right)  ,\mathcal{V}_{....}^{\mu}\left(  r\right)  $ appearing in
(\ref{eq:maincontinuitydexixjtime}) will be given in Proposition
\ref{prop:explicitboundUVf}.

\begin{remark}
[Dependence of the constants]Throughout the paper, we will call `structural
constant' any constant $c>0$ only depending on the matrix $B$ and the
\`ellipticity constant' $\nu$. Notice that the matrix $B$ encodes in
particular the numbers $N$, $Q$, $q_{i}$, $m_{i}$, $\boldsymbol{\kappa}$, and
the functions $d_{\boldsymbol{X}}$, $\rho_{\boldsymbol{X}}$, $d$, $\rho$. Any
other dependence will be specified.
\end{remark}

\bigskip

\noindent\emph{Structure of the paper.} Let us now briefly explain the
strategy we follow to prove our a priori estimates. As in the classical
Schauder theory, the operator with variable coefficients is seen as a small
local perturbation of the constant one obtained by freezing the coefficients
$a_{ij}$ at some point $\left(  \overline{x},\overline{t}\right)  $. In our
context, since the coefficients are not continuous in $t$, we can only see our
operator as a small local perturbation of the operator with coefficients only
depending on $t$, obtained by freezing the $a_{ij}\left(  \cdot,t\right)  $ at
some point $\overline{x}$. Therefore, our model operator is the one with
bounded measurable coefficients $a_{ij}\left(  t\right)  $:
\[
\mathcal{L}u=\sum_{i,j=1}^{m_{0}}a_{ij}(t)\partial_{x_{i}x_{j}}^{2}%
u+\sum_{k,j=1}^{N}b_{jk}x_{k}\partial_{x_{j}}u-\partial_{t}u.
\]
So, the startig point of our strategy is a careful study of the operator
$\mathcal{L}$ with bounded measurable coefficients $a_{ij}(t)$. For this
operator, an explicit fundamental solution has been computed and studied by
Bramanti and Polidoro in \cite{BP}; more properties and sharp estimates for
this fundamental solution have been established in \cite{BB}. In Section 2,
after recalling some known facts about the metrics (\S 2.1) and establishing
some preliminary results on the Dini-type function spaces (\S 2.2), in \S 2.3
and 2.4 we recall some results proved in \cite{BB} and \cite{BP} about \ the
fundamental solution of the model operator with coefficients $a_{ij}(t)$ and
some interpolation inequalities for H\"{o}lder norms.

In Section 3 we keep studying the model operator with coefficients only
depending on $t$. We first establish representation formulas for $u$ and
$u_{x_{i}x_{j}}$ in terms of $\mathcal{L}u$, exploiting this fundamental
solution, under the partial Dini-continuity assumption on $\mathcal{L}u$.
Then, by singular integral techniques, we prove the desired a priori estimates
for this model operator (see Theorems \ref{Thm Dini coeff t}%
-\ref{Thm local continuity time}). In Section \ref{sec:proofMainThm} we then
study the operator with coefficients $a_{ij}\left(  x,t\right)  $. Here we
apply the classical \textquotedblleft Korn's trick\textquotedblright\ of
freezing the coefficients of $\mathcal{L}$, in our case only w.r.t. $x$,
writing representation formulas for $u_{x_{i}x_{j}}$ and then regard the
original operator as a small local perturbation of the frozen one. This allows
us to prove the desired a priori estimates for functions with small compact
support (Section 4.1). Removing this restriction requires the use of cutoff
functions and interpolation inequalities for the derivatives of intermediate
order; this is accomplished in Section 4.2, completing the proof of our first
main result, Theorem \ref{Thm main space}. Finally, in Section 5 we prove our
second main result, Theorem \ref{Thm main time}, that is the bound of the
modulus of continuity of $u_{x_{i}x_{j}}$ in the joint variables $\left(
x,t\right)  $.

\bigskip

\noindent\emph{Acknowledgements.} We wish to thank the anonymous Referees who
carefully read our paper and pointed out several minor mistakes or
imprecisions. They have significantly helped us to improve the final form of
our paper.

\section{Preliminaries}

\label{sec:preliminaries} We collect in this section several preliminary
results which will be used in the rest of the paper. For basic facts and more
details about H\"{o}rmander vector fields, the metric they induce, and
homogeneous groups, we refer to \cite{BBbook}.

\subsection{Some metric properties}

As already discussed in the Introduction, Lanconelli and Polidoro \cite{LP}
proved that there is an `intrinsic subelliptic geometry' associated with any
KFP o\-pe\-ra\-tor. More precisely, if $\mathcal{L}$ is as in \eqref{L} and
if
\[
\mathbf{X}=\{\partial_{x_{1}},\ldots,\partial_{x_{m_{0}}},Y\},
\]
assumption (H2) ensures that the weighted distance $\rho_{\mathbf{X}}$ induced
by $\mathbf{X}$ is well-de\-fi\-ned, left-invariant w.r.t. group operation
\thinspace$\circ$ and $D(\lambda)$-homogeneous of degree $1$. Even if it
se\-ems natural to investigate the regularity properties of $\mathcal{L}$
using this distance, the lack of an explicit expression makes better suited
the \emph{quasi-distance}
\begin{equation}
d((x,t),(y,s))=\rho((y,s)^{-1}\circ(x,t))=\Vert x-E(t-s)y\Vert+\sqrt{|t-s|},
\label{eq:explicitd}%
\end{equation}
which is \emph{globally equivalent to $\rho_{\mathbf{X}}$ and has an explicit
form.} We now list here below some simple properties of $d$ which shall be
used in the sequel. \vspace*{0.1cm}

We begin by observing that, since $E(0)=\mathbb{I}$, from \eqref{eq:explicitd}
we infer that
\begin{equation}
d((x,t),(y,t))=\left\Vert x-y\right\Vert =\sum_{i=1}^{N}|x_{i}-y_{i}%
|^{1/q_{i}}\quad\forall\,\,x,y\in\RN,\,t\in\mathbb{R}. \label{d stesso t}%
\end{equation}
As a consequence, we derive that $d$ \emph{is sym\-me\-tric and independent of
$t$ when applied to points of $\mathbb{R}^{N+1}$ with the same $t$%
-coordinate}. Unfortunately, an a\-na\-lo\-gous property for points with the
same $x$-coordinate \emph{does not hold}. In fact, for every fixed
$x\in\mathbb{R}^{N}$ and every $t,s\in\mathbb{R}$, again by
\eqref{eq:explicitd}, we have
\[
d((x,t),(x,s))=\Vert x-E(t-s)x\Vert+\sqrt{|t-s|}.
\]
Now, since the geometry of a metric space is encoded in the `shape' of the
balls, in our context we are led to consider the \emph{$d$-balls} associated
with $d$. Recalling that $d$ is a \emph{quasi-distance} (in particular, $d$ is
not symmetric), we fix once and for all the following definition: given any
$\xi\in\mathbb{R}^{N+1}$ and any $r>0$, we define
\[
B_{r}(\xi):=\{\eta\in\mathbb{R}^{N+1}:\,d(\eta,\xi)<r\}.
\]
Using the translation-invariance and the homogeneity of $d$, it is not
difficult to recognize that the following properties are satisfied:
\begin{align}
\mathrm{(i)}  &  \,\,B_{r}(\xi)=\xi\circ B_{r}(0)=\xi\circ D_{r}\left(
B_{1}(0)\right)  \quad\forall\,\,\xi\in\mathbb{R}^{N+1}%
,\,r>0;\label{eq:balltraslD}\\
\mathrm{(ii)}  &  \,\,|B_{r}(\xi)|=|B_{r}(0)|=\omega_{Q}r^{Q+2},\quad
\text{where $\omega_{Q}:=|B_{1}(0)|>0$}. \label{measure ball}%
\end{align}
On the other hand, since $d$ satisfies
\eqref{eq:quasitriangled}-\eqref{eq:quasisymd} with a positive constant
$\bd\kappa$ \emph{pos\-si\-bly greater that $1$}, we readily derive that
\medskip

(iii)\,\,if $\eta\in B_{r}(\xi)$, then $\xi\in B_{\bd\kappa r}(\eta)$;

(iv)\,\,if $\eta_{1},\eta_{2}\in B_{r}(\xi)$, then $d(\eta_{1},\eta_{2}) <
2\bd\kappa r$. \medskip

\noindent Finally, we state for a future reference some elementary lemmas
concerning the quasi-distance $d$; for a proof of these results we refer to
\cite{BB}.

\begin{lemma}
\label{lem:equivalentd} There exists a structural constant $\bd{\vartheta}>0$
such that, if $\xi_{1},\xi_{2}$ and $\eta$ are points in $\mathbb{R}^{N+1}$
which satisfy $d(\xi_{1},\eta)\geq2\bd{\kappa}\,d(\xi_{1},\xi_{2})$, one has
\begin{equation}
\bd{\vartheta}^{-1}d(\xi_{2},\eta)\leq d(\xi_{1},\eta)\leq\bd{\vartheta}d(\xi
_{2},\eta), \label{constant M}%
\end{equation}
Here, $\bd{\kappa}>0$ is the constant appearing in \eqref{eq:quasitriangled}-\eqref{eq:quasisymd}.
\end{lemma}

\begin{lemma}
\label{Lemma E(s)x} There exists a structural constant $c>0$ such that
\begin{equation}
\Vert E(t)x\Vert\leq c\rho(x,t)=c\big(\Vert x\Vert+\sqrt{|t|}\big)\qquad
\forall\,\,x\in\mathbb{R}^{N},\,t\in\mathbb{R}. \label{eq:Etxhom}%
\end{equation}

\end{lemma}

\begin{lemma}
\label{lem:stimaEx} Let $K\subseteq\mathbb{R}^{N}$ be a fixed compact set, and
let $T>\tau>-\infty$. There exists a constant $c=c(K,T,\tau)>0$ such that, for
every $x\in K$ and $t,s\in\lbrack\tau,T]$,%
\begin{align}
&  \Vert x-E(t-s)x\Vert\leq c\,|t-s|^{1/q_{N}}\label{eq:stimaExHolder}\\
&  \Vert(E(t)-E(s))x\Vert\leq c\,|t-s|^{1/q_{N}}. \label{eq:stimaEtEsHolder}%
\end{align}
Here $q_{N}\geq3$ is the maximum exponent appearing in (\ref{dilations}).
\end{lemma}

\subsection{Function spaces}

\label{sec function spaces}Let us now turn our attention to the notion of
\emph{partial Dini and log-Dini continuity}. In what follows, $\Omega
\subseteq\mathbb{R}^{N+1}$ is an arbitrary open set. \vspace*{0.1cm}

We begin by recalling that, according to Definition \ref{Def Dini}, a function
$f:\Omega\rightarrow\mathbb{R}$ belongs to the space $\mathcal{D}(\Omega)$
(resp.\thinspace$\mathcal{D}_{\mathrm{log}}(\Omega)$) if $f\in L^{\infty
}(\Omega)$ and
\begin{equation}
\begin{gathered} \int_0^1\frac{\omega_{f,\Omega}(r)}{r}\,dr < \infty \quad \Big(\text{resp.\,$\int_0^1\frac{\omega_{f,\Omega}(r)}{r}|\log r|\,dr < \infty$}\Big), \\[0.1cm] \text{where $\omega_{f,\Omega}(r) =\sup_{\begin{subarray}{c} (x,t),(y,t)\in\Omega \\ \|x-y\| \leq r \end{subarray}}|f(x,t)-f(y,t)|$}. \end{gathered} \label{eq:defDinirecall}%
\end{equation}
We obviously have $\mathcal{D}_{\mathrm{log}}(\Omega)\subseteq\mathcal{D}%
(\Omega)$.

We also notice that, given any $f\in L^{\infty}(\Omega)$, by
\eqref{d stesso t} we can write
\[
\omega_{f,\Omega}(r)=\sup_{\begin{subarray}{c}
(x,t),(y,t)\in\Omega \\
d((x,t),(y,t)) \leq r
\end{subarray}}|f(x,t)-f(y,t)|.
\]
Moreover, $\omega_{f,\Omega}$ is non-negative, non-decreasing and
\emph{globally bounded} on $(0,\infty)$; more precisely, we have the obvious
estimate
\begin{equation}
0\leq\omega_{f,\Omega}(r)\leq2\Vert f\Vert_{L^{\infty}(\Omega)}\quad
\forall\,\,r>0. \label{eq:boundwf}%
\end{equation}

\begin{remark}
\label{rem:propDinispaces} Here we list some remarks on partially Dini and
log\--Dini con\-ti\-nuo\-us functions which easily follow from the definition.

\begin{enumerate}
\item If $f\in L^{\infty}(\Omega)$ and $0<a<b$, by \eqref{eq:boundwf} we have
\begin{equation}
\int_{a}^{b}\frac{\omega_{f,\Omega}(r)}{r}\,dr\leq2\Vert f\Vert_{L^{\infty
}(\Omega)}\log(b/a)<\infty. \label{eq:integralwfawayzero}%
\end{equation}
Thus, condition \eqref{eq:defDinirecall} is actually an integrability
condition \emph{near $0$}.

\item If $f\in\mathcal{D}(\Omega)$ and $r>0$, from
\eqref{eq:integralwfawayzero} we infer that
\begin{equation}%
\begin{split}
\int_{0}^{r}\frac{\omega_{f,\Omega}(s)}{s}\,ds  &  =\int_{0}^{1}\frac
{\omega_{f,\Omega}(s)}{s}\,ds+\int_{1}^{r}\frac{\omega_{f,\Omega}(s)}{s}\,ds\\
&  \leq|f|_{\mathcal{D}(\Omega)}+2\Vert f\Vert_{L^{\infty}(\Omega)}%
\log(r)\,\mathbf{1}_{(1,\infty)}(r)\\[0.1cm]
&  \leq\big(1+2\log(r)\mathbf{1}_{(1,\infty)}(r)\big)\Vert f\Vert
_{\mathcal{D}(\Omega)},
\end{split}
\label{omega grande r}%
\end{equation}
where $|\cdot|_{\mathcal{D}(\Omega)}$ and $\Vert\cdot\Vert_{\mathcal{D}%
(\Omega)}$ are as in Definition \ref{Def Dini}.

\item If $f\in\mathcal{D}(\Omega)$ (so that $f\in L^{\infty}(\Omega)$ and
condition \eqref{eq:defDinirecall} is satisfied), it is re\-a\-di\-ly seen
that $\omega_{f,\Omega}(r)\to0$ as $r\to0^{+}$; thus, $\omega_{f,\Omega}$ is a
\emph{continuity modulus} (i.e., it is non-negative, non-decreasing and it
vanishes as $r\to0^{+}$).
\end{enumerate}
\end{remark}

Next, we can now turn to the functions $\mathcal{M}_{f,\Omega},\,\mathcal{N}%
_{f,\Omega}$ introduced in Remark \ref{rem:funzioniOmega} (and appearing in
Theorems \ref{Thm main space}-\ref{Thm main time}).

In the following, we want to prove that when $f$ is a function satisfying
suitable continuity properties (reflecting in properties of $\omega_{f,\Omega
}$), then the moduli $\mathcal{M}_{f,\Omega},\,\mathcal{N}_{f,\Omega}$ have
suitable properties. By the definition of $\mathcal{N}_{f,\Omega}$, this will
involve some iterative argument. Now, while the function $\omega_{f,\Omega
}\left(  r\right)  $ is globally bounded as soon as $f$ is bounded, the same
is not true for $\mathcal{M}_{f,\Omega}\left(  r\right)  $ (see
(\ref{omega grande r})). In view of this fact, it is useful to introduce the
following definition.

\begin{definition}
We will say that a function $\omega:\mathbb{R}^{+}\equiv(0,\infty
)\rightarrow\mathbb{R}$ is a \emph{continuity modulus of exponent} $\alpha
\in(0,1)$ if

{(P1)}\thinspace\thinspace$\omega$ is non-decreasing on $\mathbb{R}^{+}$, and
$\omega(r)\rightarrow0$ as $r\rightarrow0^{+}$;

{(P2)}\thinspace\thinspace there exists $\omega_{0}>0$ such that
\[
\omega(r)\leq\omega_{0}\,r^{\alpha}\quad\forall\,\,r\geq1;
\]

If, in addition, we have
\begin{equation}
\lbrack\omega]:=\int_{0}^{1}\frac{\omega(r)}{r}\,dr<\infty,
\label{eq:propP3Dini}%
\end{equation}
we will say that $\omega$ is a \emph{Dini continuity modulus} (of exponent
$\alpha$).
\end{definition}

\begin{lemma}
\label{lem:generalMN} Let $\alpha\in(0,1)$, and let $\omega:\mathbb{R}%
^{+}\rightarrow\mathbb{R}$ be a Dini con\-ti\-nu\-ity mo\-du\-lus of exponent
$\alpha$. Then, the function $M(\omega)$ defined by
\begin{equation}
{M}(\omega)(r)=\omega(r)+\int_{0}^{r}\frac{\omega(s)}{s}\,ds+r\int_{r}%
^{\infty}\frac{\omega(s)}{s^{2}}\,ds. \label{eq:defMGeneral}%
\end{equation}
is a continuity modulus with exponent $\alpha$. In particular, there exists a
constant $c>0$, only depending on $\alpha$, such that
\begin{equation}
\text{$M(\omega)(r)\leq\omega_{0}^{\prime}r^{\alpha}$ for all $r\geq1$, \quad
where $\omega_{0}^{\prime}=c([\omega]+\omega_{0})$},
\label{eq:MomegaPropiiGeneral}%
\end{equation}
If, in addition, $\omega$ satisfies the \emph{stronger integrability
property}
\begin{equation}
\int_{0}^{1}\frac{\omega(r)}{r}|\log(r)|\,dr<\infty,
\label{eq:strongeromegaDiniLemma}%
\end{equation}
then ${M}(\omega)$ is a \emph{Dini continuity modulus}. In particular, we
have
\begin{equation}
\lbrack M(\omega)]=\int_{0}^{1}\frac{M(\omega)(s)}{s}\,ds\leq c\Big(\int%
_{0}^{1}\frac{\omega(s)}{s}(1+|\log(s)|)ds+\omega_{0}\Big),
\label{eq:DiniModMOmegaGen}%
\end{equation}
for a constant $c>0$ only depending on $\alpha$.
\end{lemma}

\begin{proof}
To ease the readability, we split the proof into three steps.

\textsc{Step I:} In this first step we prove that $M(\omega)$ is well-defined
on $\mathbb{R}^{+}$. To this end, we observe that, by (P1) and
\eqref{eq:propP3Dini}, we have
\begin{align*}
\int_{0}^{r}\frac{\omega(s)}{s}\,ds  &  \leq\int_{0}^{1}\frac{\omega(s)}%
{s}\,ds+\int_{1}^{\max\{1,r\}}\frac{\omega(s)}{s}\,ds\\
&  \leq\lbrack\omega]+\omega(\max\{r,1\})r<\infty\quad\forall\,\,r>0.
\end{align*}
Moreover, by exploiting property (P2) (and since $\alpha<1$), we also have
\begin{align*}
\int_{r}^{\infty}\frac{\omega(s)}{s^{2}}\,ds  &  \leq\int_{\min\{r,1\}}%
^{1}\frac{\omega(s)}{s^{2}}\,ds+\int_{1}^{\infty}\frac{\omega(s)}{s^{2}}\,ds\\
&  \leq\frac{\omega(1)}{r}+\omega_{0}\int_{1}^{\infty}\frac{1}{s^{2-\alpha}%
}\,ds<\infty.
\end{align*}
Gathering these facts, we then conclude that $M(\omega)(r)<\infty$ for all
$r>0$.

\textsc{Step II:} Now we have shown that $M(\omega)$ is well-defined, we then
turn to pro\-ve that such a function is a continuity modulus of exponent
$\alpha$, further satisfying estimate (\ref{eq:DiniModMOmegaGen}). To this
end, we first observe that, owing to the properties of $\omega$, the
(well-defined) function
\[
F(r):=\omega(r)+\int_{0}^{r}\frac{\omega(s)}{s}\,ds\quad(r>0)
\]
is clearly non-negative, non-decreasing and it vanishes as $r\rightarrow0^{+}%
$. Moreover, by using property (P2) of $\omega$ we see that, for every
$r\geq1$,
\begin{align*}
F(r)  &  \leq\omega_{0}r^{\alpha}+\int_{0}^{1}\frac{\omega(s)}{s}\,ds+\int%
_{1}^{r}\frac{\omega(s)}{s}\,ds\\
&  \leq\omega_{0}r^{\alpha}+[\omega]+\omega_{0}\int_{1}^{r}s^{\alpha
-1}\,ds\leq\frac{1}{\alpha}([\omega]+2\omega_{0})r^{\alpha},
\end{align*}
and thus, $F$ is a continuity modulus of exponent $\alpha$ satisfying
\eqref{eq:MomegaPropiiGeneral}. In view of these facts, and taking into
account \eqref{eq:defMGeneral}, to prove that ${M}(\omega)$ is a continuity
modulus, we consider the (well-defined) map
\[
G(r):=r\int_{r}^{\infty}\frac{\omega(s)}{s^{2}}\,ds\qquad(r>0),
\]
and we show that also $G$ satisfies the following properties: \vspace*{0.1cm}

(a)\thinspace\thinspace$G$ is non-negative, non-decreasing and it vanishes as
$r\rightarrow0^{+}$;

(b)\thinspace\thinspace there exists a constant $\hat{c}>0$, only depending on
$\alpha$ such that
\[
G(r)\leq\hat{c}\,\omega_{0}\,r^{\alpha}\quad\forall\,\,r\geq1.
\]
\noindent\emph{Proof of} (a). Clearly, $G(r)\geq0$ for every $r>0$ (as
$\omega$ is non-negative); moreover, by Lebesgue's Differentiation Theorem
(and recalling that the function $\omega$ is non-decreasing on $(0,\infty)$),
for a.e.\thinspace$r>0$ we have
\[
G^{\prime}(r)=\int_{r}^{\infty}\frac{\omega(s)}{s^{2}}\,ds-\frac{\omega(r)}%
{r}\geq\omega(r)\int_{r}^{\infty}\frac{1}{s^{2}}\,ds-\frac{\omega(r)}{r}=0,
\]
and this proves that $G$ is non-decreasing. Finally, we turn to prove that
$G(r)$ vanishes as $r\rightarrow0^{+}$. To this end, it is useful to
distinguish two cases.

$\bullet$\thinspace\thinspace If $\int_{0}^{\infty}\frac{\omega(s)}{s^{2}%
}\,ds<\infty$, we immediately get
\[
\lim_{r\rightarrow0^{+}}G(r)=\lim_{r\rightarrow0^{+}}r\int_{r}^{\infty}%
\frac{\omega(s)}{s^{2}}\,ds=0.
\]

$\bullet$\thinspace\thinspace If, instead, $\int_{0}^{\infty}\frac{\omega
(s)}{s^{2}}\,ds=\infty$, we observe that
\[
\frac{\left(  \int_{r}^{\infty}\frac{\omega(s)}{s^{2}}\,ds\right)  ^{\prime}%
}{(1/r)^{\prime}}=\omega(r)\quad\text{for every $r>0$};
\]
thus, since $\omega(r)$ vanishes as $r\rightarrow0^{+}$ (see assumption (ii)),
an immediate application of De L'H\^{o}pital's Theorem gives
\[
\lim_{r\rightarrow0^{+}}G(r)=\lim_{r\rightarrow0^{+}}\frac{\int_{r}^{\infty
}\frac{\omega(s)}{s^{2}}\,ds}{1/r}=0.
\]
\emph{Proof of} (b). By exploiting property (P2) of $\omega$, we immediately
get
\[
G(r)\leq\omega_{0}r\int_{r}^{\infty}s^{\alpha-2}\,ds=\omega_{0}r\Big[\frac
{s^{\alpha-1}}{\alpha-1}\Big]_{r}^{\infty}=\frac{\omega_{0}}{1-\alpha
}r^{\alpha}\quad\forall\,\,r\geq1,
\]
and this proves that $G$ satisfies (b). Summing up, the function $G$ is a
continuity modulus of exponent $\alpha$ satisfying
\eqref{eq:MomegaPropiiGeneral}, and thus the same is true for $M(\omega)$.

\textsc{Step III:} In this last step, we prove that $M(\omega)$ satisfies
\eqref{eq:DiniModMOmegaGen} (hence, $M(\omega)$ is a Dini continuity modulus
of exponent $\alpha$), provided $\omega$ satisfies the \emph{stronger
property} \eqref{eq:strongeromegaDiniLemma}. To prove this fact, and since
$\omega$ satisfies \eqref{eq:strongeromegaDiniLemma}, we set
\begin{equation}
M_{1}(r)=\int_{0}^{r}\frac{\omega(s)}{s}\,ds,\qquad M_{2}(r)=r\int_{r}%
^{\infty}\frac{\omega(s)}{s^{2}}=G(r), \label{eq:defM1M2proof}%
\end{equation}
and we show that both $M_{1},\,M_{2}$ satisfy property (iii), that is,
\[
\int_{0}^{1}\frac{M_{i}(r)}{r}\,dr<\infty\quad\forall\,\,i=1,2.
\]
As regards $M_{1}$, by Fubini-Tonelli's Theorem we have
\begin{equation}%
\begin{split}
\int_{0}^{1}\frac{M_{1}(r)}{r}dr  &  =\int_{0}^{1}\frac{1}{r}\left(  \int%
_{0}^{r}\frac{\omega(s)}{s}ds\right)  dr=\int_{0}^{1}\frac{\omega(s)}%
{s}\left(  \int_{s}^{1}\frac{dr}{r}\right)  ds\\
&  =\int_{0}^{1}\frac{\omega(s)}{s}|\log s|\,ds<\infty,
\end{split}
\label{eq:DiniM1LemmaGen}%
\end{equation}
where we have used the fact that $\omega$ satisfies
\eqref{eq:strongeromegaDiniLemma}. As regards $M_{2}$, again by using
Fubini-Tonelli's Theorem (and since $\omega$ satisfies
\eqref{eq:strongeromegaDiniLemma}), we obtain
\begin{equation}%
\begin{split}
\int_{0}^{1}\frac{M_{2}(r)}{r}\,dr  &  =\int_{0}^{1}\Big(\int_{r}^{\infty
}\frac{\omega(s)}{s^{2}}\,ds\Big)dr\\
&  =\int_{0}^{\infty}\frac{\omega(s)}{s^{2}}\Big(\int_{0}^{\min\{s,1\}}%
\,dr\Big)ds\\
&  =\int_{0}^{1}\frac{\omega(s)}{s}\,ds+\int_{1}^{\infty}\frac{\omega
(s)}{s^{2}}\,ds\\
&  \leq\int_{0}^{1}\frac{\omega(s)}{s}\,ds+\frac{\omega_{0}}{1-\alpha}<\infty,
\end{split}
\label{eq:DiniM2LemmaGen}%
\end{equation}
where we have also used the fact that $\omega$ satisfies property (P2) (with
$\alpha<1$). Finally, by combining
\eqref{eq:DiniM1LemmaGen}-\eqref{eq:DiniM2LemmaGen}, we conclude that
\begin{align*}
\lbrack M(\omega)]  &  =\int_{0}^{1}\frac{M(\omega)(s)}{s}\,ds\leq
\lbrack\omega]+\int_{0}^{1}\frac{\omega(s)}{s}|\log(s)|\,ds+[\omega
]+\frac{\omega_{0}}{1-\alpha}\\
&  \leq c\Big(\int_{0}^{1}\frac{\omega(s)}{s}(1+|\log(s)|)ds+\omega_{0}\Big).
\end{align*}
This ends the proof.
\end{proof}

\begin{remark}
\label{rem:stimedaUsareGen} Let $\alpha\in(0,1)$, and let $\omega
:\mathbb{R}^{+}\rightarrow\mathbb{R}$ be a Dini continuity modulus of exponent
$\alpha$. It is contained in the proof of Lemma \ref{lem:generalMN} the
following useful (thought not sharp) bound, which will be repeatedly used in
the sequel:
\begin{equation}
\int_{0}^{r}\frac{\omega(s)}{s}\,ds\leq c_{\alpha}([\omega]+\omega
_{0})(1+r^{\alpha})\quad\forall\,\,r>0 \label{eq:estimIntOmegafinoarDaUsare}%
\end{equation}
(where $c>0$ is a constant only depending on $\alpha$).
\end{remark}

Thanks to Lemma \ref{lem:generalMN}, we readily obtain the following

\begin{proposition}
\label{prop:funMN} Assume that $f\in\mathcal{D}(\Omega)$. Then, the function
$\mathcal{M}_{f,\Omega}(r)$ defined in \eqref{eq:defcontinM} is a
\emph{modulus of continuity of exponent $\alpha$}, for every $\alpha\in(0,1)$.
In particular, given any $\alpha\in(0,1)$ there exists a constant $c>0$ only
depending on $\alpha$ such that%
\begin{equation}
\mathcal{M}_{f,\Omega}(r)\leq{c}\Vert f\Vert_{\mathcal{D}(\Omega)}r^{\alpha
}\quad\forall\,\,r\geq1, \label{eq:omegazeroMfOmega}%
\end{equation}
If, in addition, $f\in\mathcal{D}_{\mathrm{log}}(\Omega)$, then the function
$\mathcal{M}_{f,\Omega}$ is a \emph{Dini continuity modulus}; in particular,
given any $\alpha\in(0,1)$ there exists $c=c_{\alpha}>0$ such that
\begin{equation}
\int_{0}^{1}\frac{\mathcal{M}_{f,\Omega}(r)}{r}\,dr\leq c\Big(\int_{0}%
^{1}\frac{\omega_{f,\Omega}(s)}{s}(1+|\log(s)|)ds+\Vert f\Vert_{L^{\infty
}(\Omega)}\Big)<\infty. \label{eq:MDinimod}%
\end{equation}
Finally, the function ${\mathcal{N}}_{f,\Omega}(r)$ defined in
\eqref{eq:defcontinN} is a \emph{modulus of continuity of exponent $\alpha$},
for every $\alpha\in(0,1)$.
\end{proposition}

\begin{proof}
First of all, we observe that, if $f\in\mathcal{D}(\Omega)$, then
$\omega_{f,\Omega}$ is a Dini continuity modulus of exponent $\alpha$, for
every $\alpha\in(0,1)$. To be more precise, if $\alpha\in(0,1)$ is
ar\-bi\-tra\-ri\-ly chosen, by exploiting \eqref{eq:boundwf} we get
\[
\text{$\omega_{f,\Omega}(r)\leq\omega_{0}\,r^{\alpha}$ for all $r\geq1$,\quad
with $\omega_{0} = 2\|f\|_{\mathcal{D}(\Omega)}$}.
\]
Thus, since $\mathcal{M}_{f,\Omega} = M(\omega_{f,\Omega})$ (where $M$ is as
in \eqref{eq:defMGeneral}), from Lemma \ref{lem:generalMN} we infer that
$\mathcal{M}_{f,\Omega}$ is a modulus of continuity of exponent $\alpha$. In
particular,
\[
\mathcal{M}_{f,\Omega}(r)\leq\, \omega_{0}^{\prime}\, r^{\alpha}\quad
\forall\,\,r\geq1,
\]
where $\omega_{0}^{\prime}> 0$ is a constant which, by
\eqref{eq:MomegaPropiiGeneral}, is of the form
\[
\omega_{0}^{\prime}= {c}([\omega_{f,\Omega}]+\omega_{0}) = {c}\Big(\int%
_{0}^{1}\frac{\omega_{f,\Omega}(s)}{s}\,ds+\omega_{0}\Big)
\leq{c}\|f\|_{\mathcal{D}(\Omega)},
\]
and ${c} > 0$ is a constant only depending on $\alpha$. This gives
\eqref{eq:omegazeroMfOmega}. \vspace*{0.1cm} If, in addition, $f\in
\mathcal{D}_{\log}(\Omega)$, the function $\omega_{f,\Omega}$ also satisfies
assumption \eqref{eq:strongeromegaDiniLemma} in the statement of Lemma
\ref{lem:generalMN}; we then infer from this lemma that
\[
\mathcal{M}_{f,\Omega} = M(\omega_{f,\Omega})
\]
is a \emph{Dini continuity modulus}. Moreover, by \eqref{eq:DiniModMOmegaGen},
we have
\begin{align*}
\int_{0}^{1}\frac{\mathcal{M}_{f,\Omega}(r)}{r}\,dr  &  = [\mathcal{M}%
_{f,\Omega}] \leq c\Big( \int_{0}^{1}\frac{\omega_{f,\Omega}(s)}{s}%
(1+|\log(s)|)ds + \omega_{0}\Big)\\
&  \leq c\Big(\int_{0}^{1}\frac{\omega_{f,\Omega}(s)}{s}(1+|\log(s)|)ds+
\|f\|_{L^{\infty}(\Omega)}\Big),
\end{align*}
where $c > 0$ depends on the fixed $\alpha$. Finally, we also have that
\[
\mathcal{N}_{f,\Omega} = M(\mathcal{M}_{f,\Omega})
\]
is a well-defined modulus of continuity, and the proof is complete.
\end{proof}

\begin{remark}
\label{rem:NnotDini} It should be noticed that, even if $f\in\mathcal{D}%
_{\mathrm{log}}(\Omega)$, the function $\mathcal{N}_{f,\Omega}$ may not be a
Dini con\-ti\-nui\-ty mo\-du\-lus; namely, we cannot ensure that
\begin{equation}
\int_{0}^{1}\frac{\mathcal{N}_{f,\Omega}(r)}{r}\,dr<\infty.
\label{eq:DinicontN}%
\end{equation}
In fact, by arguing as in the proof of Proposition \ref{prop:funMN}, we see
that a sufficient condition for \eqref{eq:DinicontN} to hold is the
\emph{log-Dini continuity of $\mathcal{M}_{f,\Omega}$}, i.e.,
\[
\int_{0}^{1}\frac{\mathcal{M}_{f,\Omega}(r)}{r}|\log(r)|\,dr<\infty.
\]
This, in turn, is readily seen to be satisfied as soon as%
\[
\int_{0}^{1}\frac{\omega_{f,\Omega}(r)}{r}\log^{2}(r)\,dr<\infty,
\]
that is when $f$ is log$^{2}$-Dini continuous.
\end{remark}

Now that we have fully established Proposition \ref{prop:funMN}, we proceed by
studying the two functions $\mathcal{U}^{\mu}_{f,\Omega},\,\mathcal{V}^{\mu
}_{f,\Omega}$ introduced in Remark \ref{rem:funzioniOmega}.

\begin{lemma}
\label{lem:Ufgeneral} Let $\alpha\in(0,1)$, and $\omega:\mathbb{R}%
^{+}\rightarrow\mathbb{R}$ be a Dini continuity modulus of exponent $\alpha$.
For a given $\mu>0$, we consider the function
\begin{equation}
U^{\mu}(\omega)(r):=\int_{\RN}e^{-\mu|z|^{2}}\Big(\int_{0}^{r\Vert z\Vert
}\frac{\omega(s)}{s}\,ds\Big)dz\qquad(r>0). \label{eq:defUmuGeneral}%
\end{equation}
Then, the following facts hold:

\begin{itemize}
\item[\emph{(i)}] there exists a constant $c>0$, only depending on $\mu$ and
$\alpha$, such that
\[
0\leq U^{\mu}(\omega)(r)\leq c(1+r^{\alpha})(\omega_{0}+[\omega])\quad
\forall\,\,r>0;
\]

\item[\emph{(ii)}] $U^{\mu}(\omega)(r)\rightarrow0$ as $r\rightarrow0^{+}$.
\end{itemize}
\end{lemma}

\begin{proof}
(i)\,\,Since $\omega$ is a Dini continuity modulus of exponent $\alpha$, we
get
\begin{equation}
\label{eq:estimUfmuIntegrand}%
\begin{split}
&  e^{-\mu|z|^{2}}\Big(\int_{0}^{r\|z\|}\frac{\omega(s)}{s}\,ds\Big)\\
&  \quad\leq e^{-\mu|z|^{2}} \Big( [\omega]+\omega_{0}\int_{1}^{\max
\{r\|z\|,1\}}s^{\alpha-1}\,ds\Big)\\[0.1cm]
&  \quad\leq c([\omega]+\omega_{0})\cdot e^{-\mu|z|^{2}}(1+r^{\alpha
}\|z\|^{\alpha})\quad\forall\,\,z\in\RN,\,r > 0.
\end{split}
\end{equation}
From this, since $\mu> 0$ and $\|z\| = \sum_{j}|z_{j}|^{1/q_{j}}$, we obtain
\begin{align*}
&  0\leq U^{\mu}(\omega)(r) \leq c([\omega]+\omega_{0}) \int_{\RN}%
e^{-\mu|z|^{2}}(1+r^{\alpha}\|z\|^{\alpha})\,dz\\
&  \qquad\leq c([\omega]+\omega_{0})\Big(\int_{\RN}e^{-\mu|z|^{2}%
}\,dz+r^{\alpha}\int_{\RN}e^{-\mu|z|^{2}}\|z\|^{\alpha}\,dz\Big)\\
&  \qquad= c([\omega]+\omega_{0})(c_{1,\mu}+c_{2,\mu}r^{\alpha}) \leq
c([\omega]+\omega_{0})(1+r^{\alpha}),
\end{align*}
where $c > 0$ is a constant only depending on $\mu$ and $\alpha$.
\medskip(ii)\,\,We fist observe that, since $\int_{0}^{1}\frac{\omega(s)}%
{s}\,ds<\infty$, we have
\[
\lim_{r\to0^{+}}e^{-\mu|z|^{2}}\Big(\int_{0}^{r\|z\|}\frac{\omega(s)}%
{s}\,ds\Big) = 0\quad\forall\,\,z\in\RN.
\]
On the other hand, for every $z\in\RN$ and every $r\in(0,1)$, by
\eqref{eq:estimUfmuIntegrand} we have
\begin{align*}
&  0\leq e^{-\mu|z|^{2}}\Big(\int_{0}^{r\|z\|}\frac{\omega(s)}{s}\,ds\Big)\\
&  \qquad\leq c\,e^{-\mu|z|^{2}}(1+\|z\|)\in L^{1}(\RN).
\end{align*}
Then, by Lebesgue's Dominated Convergence Theorem,
\[
\lim_{r\to0^{+}}U^{\mu}(\omega)(r) = \lim_{r\to0^{+}} \int_{\RN}e^{-\mu
|z|^{2}}\Big(\int_{0}^{r\|z\|}\frac{\omega(s)}{s}\,ds\Big)dz = 0.
\]
This ends the proof.
\end{proof}

The next proposition collects some explicit bounds for $U^{\mu}(\omega)$.

\begin{proposition}
\label{prop:stimaUf} Let $\alpha\in(0,1)$, and let $\omega:\mathbb{R}%
^{+}\rightarrow\mathbb{R}$ be a Dini continuity modulus of exponent $\alpha$.
Then, the following facts hold.

\begin{itemize}
\item[\emph{(i)}] There exist constants $c,\kappa>0$, only depending on $\mu$
and $N$, such that
\[
U^{\mu}(\omega)(r)\leq%
\begin{cases}
c\big(\int_{0}^{\sqrt{r}}\frac{\omega(s)}{s}\,ds+([\omega]+\omega
_{0})e^{-\frac{\kappa}{r}}\big) & \text{if $0<r<1$},\\[0.05cm]%
cr^{\alpha}([\omega]+\omega_{0}) & \text{if $r\geq1$}.
\end{cases}
\]

\item[\emph{(ii)}] Assume that there exists $\omega_{0} > 0$ such that
\begin{equation}
\label{eq:propertyP2strong}\text{$\omega(r)\leq\omega_{0}r^{\alpha}$ \emph{for
every $r > 0$}}%
\end{equation}
\emph{(}that is, $\omega$ satisfies the estimate in property \emph{(P2) for
every $r > 0$}, and not only for $r\geq1$\emph{)}; then, we have the following
estimate
\begin{align*}
U^{\mu}(\omega)(r)\leq c_{\mu,\alpha}\cdot\omega_{0}\,r^{\alpha}\quad\forall r
> 0.
\end{align*}

\end{itemize}
\end{proposition}

\begin{proof}
(i)\thinspace\thinspace For a fixed $R>1$ (to be chosen later on), we write%
\begin{equation}
U^{\mu}(\omega)(r)=\int_{\{\Vert z\Vert\leq R\}}\{\cdots\}\,dz+\int_{\{\Vert
z\Vert>R\}}\{\cdots\}\,dz\equiv A_{R}+B_{R}. \label{Ur2}%
\end{equation}
Then, we proceed by estimating the two integrals $A_{R},\,B_{R}$ separately,
distinguishing two cases.

\textsc{Case I: $0<r<1$}. We have:
\begin{equation}%
\begin{split}
A_{R}  &  \leq\int_{\{\Vert z\Vert\leq R\}}e^{-\mu|z|^{2}}\Big(\int_{0}%
^{rR}\frac{\omega(s)}{s}\,ds\Big)dz\\
&  \leq\Big(\int_{0}^{rR}\frac{\omega(s)}{s}\,ds\Big)\cdot\int_{\mathbb{R}%
^{N}}e^{-\mu|z|^{2}}\,dz\\
&  =c\int_{0}^{rR}\frac{\omega(s)}{s}\,ds,
\end{split}
\label{A_R}%
\end{equation}
where $c>0$ is a constant only depending on the fixed $\mu$. Next, since in
$B_{R}$ we have $\Vert z\Vert>R>1$ and we are assuming $0<r<1$, from
\eqref{eq:estimIntOmegafinoarDaUsare} we get
\[%
\begin{split}
B_{R}  &  \leq\int_{\{\Vert z\Vert>R\}}e^{-\mu|z|^{2}}\Big(\int_{0}^{\Vert
z\Vert}\frac{\omega(s)}{s}\,ds\Big)dz\\
&  \leq c([\omega]+\omega_{0})\int_{\{\Vert z\Vert>R\}}(1+\Vert z\Vert
^{\alpha})e^{-\mu|z|^{2}}\,dz\\
&  \leq c([\omega]+\omega_{0})\int_{\{\Vert z\Vert>R\}}\Vert z\Vert^{\alpha
}e^{-\mu|z|^{2}}\,dz.
\end{split}
\]
Then, by performing the change of variables $z=D_{0}(R)u$, we obtain
\begin{equation}
B_{R}\leq c([\omega]+\omega_{0})\,R^{Q+\alpha}\int_{\{\Vert u\Vert>1\}}%
e^{-\mu|D_{0}(R)u|^{2}}\Vert u\Vert^{\alpha}\,du. \label{eq:BR}%
\end{equation}
Now, since we are assuming $R>1$, we have
\[
|D_{0}(R)u|^{2}=\sum_{j=1}^{N}R^{2q_{j}}u_{j}^{2}\geq R^{2}|u|^{2}\quad
\forall\,\,u\in\RN.
\]
As a consequence, we obtain the following estimate
\begin{equation}%
\begin{split}
e^{-\mu|D_{0}(R)u|^{2}}\Vert u\Vert^{\alpha}  &  \leq e^{-\mu R^{2}|u|^{2}%
}\Vert u\Vert^{\alpha}\leq c\sum_{j=1}^{N}e^{-\mu R^{2}|u|^{2}}|u_{j}%
|^{\alpha/q_{j}}\\
&  \leq c\sum_{j=1}^{N}e^{-\mu R^{2}|u|^{2}}|u|^{\alpha/q_{j}}.
\end{split}
\label{eq:estimemumuhalf}%
\end{equation}
In view of \eqref{eq:estimemumuhalf}, and since $\{\Vert u\Vert>1\}\subseteq
\{|u|>\delta\}$ for some constant $\delta>0$ only depending on the dimension
$N$, from \eqref{eq:BR} we finally get
\begin{equation}%
\begin{split}
B_{R}  &  \leq c([\omega]+\omega_{0})\,R^{Q}\sum_{j=1}^{N}\int_{\{|u|>\delta
\}}e^{-\frac{\mu R^{2}}{2}|u|^{2}}|u|^{\alpha/q_{j}}\,du\\
&  =c([\omega]+\omega_{0})\,R^{Q}\sum_{j=1}^{N}\int_{\delta}^{\infty}%
e^{-\frac{\mu R^{2}}{2}\rho^{2}}\rho^{N+\frac{\alpha}{q_{j}}-1}\,d\rho\\
&  (\text{by the change of variables $\rho=s/R$, and since $R>1$})\\
&  \leq c([\omega]+\omega_{0})\,R^{Q}\sum_{j=1}^{N}\int_{\delta R}^{\infty
}e^{-\frac{\mu s^{2}}{2}}s^{N+\frac{\alpha}{q_{j}}-1}\,ds\\
&  \leq c([\omega]+\omega_{0})\,R^{Q}\int_{\delta R}^{\infty}e^{-\frac{\mu
s^{2}}{4}}\,ds\\
&  \leq c([\omega]+\omega_{0})\,R^{Q}e^{-\frac{\mu\delta^{2}R^{2}}{4}}\leq
c([\omega]+\omega_{0})\,e^{-{\kappa R^{2}}},
\end{split}
\label{eq:BRfinal}%
\end{equation}
where $c>0$ is a suitable constant, possibly different from line to line but
only depending on $\mu,N$ and $\alpha$, and $\kappa=\mu\,\delta^{2}/8$.
\vspace*{0.1cm} Gathering \eqref{A_R}-\eqref{eq:BRfinal}, and choosing
$R=1/\sqrt{r}>1$, we then conclude that
\[
U^{\mu}(\omega)(r)\leq A_{R}+B_{R}\leq c\Big(\int_{0}^{\sqrt{r}}\frac
{\omega(s)}{s}\,ds+([\omega]+\omega_{0})e^{-\frac{\kappa}{r}}\Big),
\]
where $c,\kappa>0$ are constants only depending on $\mu$ and $N$.
\vspace*{0.1cm}

\textsc{Case II:} $r\geq1$. Since $rR\geq R>1$, by combining
\eqref{eq:estimIntOmegafinoarDaUsare} with estimate \eqref{A_R} (which
actually holds for every $r>0$) we obtain
\begin{equation}%
\begin{split}
A_{R}  &  \leq c\int_{0}^{rR}\frac{\omega(s)}{s}\,ds\leq c([\omega]+\omega
_{0})(1+(rR)^{\alpha})\\
&  \leq c(rR)^{\alpha}([\omega]+\omega_{0}).
\end{split}
\label{A_RCaseII}%
\end{equation}
Next, since in $B_{R}$ we have $r\Vert z\Vert>rR>1$, again by
\eqref{eq:estimIntOmegafinoarDaUsare} we get
\[
\int_{0}^{r\Vert z\Vert}\frac{\omega(s)}{s}\,ds\leq c([\omega]+\omega
_{0})\big(1+(r\Vert z\Vert)^{\alpha}\big)\leq c(r\Vert z\Vert)^{\alpha
}([\omega]+\omega_{0});
\]
from this, since $\mu>0$ and $\Vert z\Vert=\sum_{j}|z_{j}|^{1/q_{j}}$, we
obtain
\begin{equation}
B_{R}\leq cr^{\alpha}([\omega]+\omega_{0})\int_{\RN}e^{-\mu|z|^{2}}\Vert
z\Vert^{\alpha}\,dz=cr^{\alpha}([\omega]+\omega_{0}), \label{B_RCaseII}%
\end{equation}
where $c>0$ is a constant depending on $\mu$ and $\alpha$. \vspace*{0.1cm}
Gathering \eqref{A_RCaseII}-\eqref{eq:BRfinal}, and choosing $R=2$, we then
conclude that
\[
U^{\mu}(\omega)(r)\leq A_{R}+B_{R}\leq cr^{\alpha}([\omega]+\omega_{0}),
\]
where $c>0$ are constants only depending on $\mu$ and $\alpha$. \medskip

(ii)\thinspace\thinspace If \eqref{eq:propertyP2strong} holds, by definition
of $U^{\mu}(\omega)$ we have
\begin{align*}
U^{\mu}(\omega)(r)  &  \leq\omega_{0}\int_{\RN}e^{-\mu|z|^{2}}\Big(\int%
_{0}^{r\Vert z\Vert}s^{\alpha-1}\,ds\Big)dz\\
&  =\frac{\omega_{0}\,r^{\alpha}}{\alpha}\int_{\RN}e^{-\mu|z|^{2}}\Vert
z\Vert^{\alpha}\,dz\equiv c_{\mu,\alpha}\cdot\omega_{0}\,r^{\alpha}%
\quad\forall r>0.
\end{align*}
This ends the proof.
\end{proof}

\medskip

Thanks to Lemmas \ref{lem:Ufgeneral}-\ref{prop:stimaUf}, we readily obtain the
following results.

\begin{lemma}
\label{lem:Ufwelldef} Let $\Omega\subseteq\mathbb{R}^{N+1}$ be an arbitrary
open set, and let $f\in\mathcal{D}(\Omega)$. For a given $\mu>0$, let
$\mathcal{U}_{f,\Omega}^{\mu}$ be as in \eqref{eq:defUmuIntro}. Then, the
following facts hold.

\begin{itemize}
\item[\emph{(i)}] for every $\alpha\in(0,1)$ there exists a constant $c>0$
only depending on $\mu$ and $\alpha$, such that
\[
0\leq\mathcal{U}_{f,\Omega}^{\mu}(r)\leq c(1+r^{\alpha})\Vert f\Vert
_{\mathcal{D}(\Omega)}\quad\forall\,\,r>0;
\]

\item[\emph{(ii)}] $\mathcal{U}^{\mu}_{f,\Omega}(r)\to0$ as $r\to0^{+}$.
\end{itemize}

If, in addition, $f\in\mathcal{D}_{\log}(\Omega)$, then the function
$\mathcal{V}_{f,\Omega}^{\mu}$ defined in \eqref{eq:defVmuIntro} satisfies the
following properties, analogous to \emph{(i)-(ii)} above:

\begin{itemize}
\item[\emph{(i)'}] for every $\alpha\in(0,1)$ there exists a constant $c>0$
only depending on $\mu$ and $\alpha$, such that
\[
0\leq\mathcal{V}_{f,\Omega}^{\mu}(r)\leq c(1+r^{\alpha})\Big(\int_{0}^{1}%
\frac{\omega_{f,\Omega}(s)}{s}(1+|\log(s)|)ds+\Vert f\Vert_{L^{\infty}%
(\Omega)}\Big),
\]

\item[\emph{(ii)'}] $\mathcal{V}_{f,\Omega}^{\mu}(r)\rightarrow0$ as
$r\rightarrow0^{+}$.
\end{itemize}
\end{lemma}

\begin{proof}
If $f\in\mathcal{D}(\Omega)$, both the properties (i) and (ii) of
$\mathcal{U}^{\mu}_{f,\Omega}$ immediately follow from Lemma
\ref{lem:Ufgeneral}, taking into account that
\[
\mathcal{U}^{\mu}_{f,\Omega} = U^{\mu}(\omega_{f,\Omega})\quad\text{and}%
\quad[\omega_{f,\Omega}]+\omega_{0} = \int_{0}^{1}\frac{\omega_{f,\Omega}%
(s)}{s}\,ds+\omega_{0} \leq3\|f\|_{\mathcal{D}(\Omega)},
\]
see the the \emph{incipit} of the proof of Proposition \ref{prop:funMN}. If,
in addition, $f\in\mathcal{D}_{\log}(\Omega)$, from Proposition
\ref{prop:funMN} we know that $\mathcal{M}_{f,\Omega}$ is a \emph{Dini
con\-ti\-nuity modulus} of exponent $\alpha$, for every $\alpha\in(0,1)$; more
precisely,
\begin{align*}
&  \text{(1)\,\,$\mathcal{M}_{f,\Omega}\leq\omega_{0}^{\prime\alpha}$ for all
$r\geq1$, where $\omega_{0}^{\prime}= c\|f\|_{\mathcal{D}(\Omega)}$};\\
&  \text{$(2)$}\,\,[\mathcal{M}_{f,\Omega}] = \int_{0}^{1}\frac{\mathcal{M}%
_{f,\Omega}(s)}{s}\,ds \leq c\Big(\int_{0}^{1}\frac{\omega_{f,\Omega}(s)}%
{s}(1+|\log(s)|)ds+\|f\|_{L^{\infty}(\Omega)}\Big),
\end{align*}
where $c > 0$ is a constant only depending on the fixed $\alpha$. As a
consequence, pro\-per\-ties (i)'-(ii)' of $\mathcal{V}^{\mu}_{f,\Omega}$
follow again from Lemma \ref{lem:Ufgeneral}, since
\[
\mathcal{V}^{\mu}_{f,\Omega} = U^{\mu}(\mathcal{M}_{f,\Omega})
\]
and since, by \eqref{eq:omegazeroMfOmega}-\eqref{eq:MDinimod} in Proposition
\ref{prop:funMN}, we have
\begin{equation}
\label{eq:omegaomegaprimeMf}%
\begin{split}
[\mathcal{M}_{f,\Omega}]+\omega_{0}^{\prime}  &  \leq\int_{0}^{1}%
\frac{\mathcal{M}_{f,\Omega}(s)}{s}\,ds + {c}\|f\|_{\mathcal{D}(\Omega)}\\
&  \leq c\Big(\int_{0}^{1}\frac{\omega_{f,\Omega}(s)}{s}(1+|\log
(s)|)ds+\|f\|_{L^{\infty}(\Omega)}\Big),
\end{split}
\end{equation}
where $c > 0$ is a constant only depending on $\alpha$. This ends the proof.
\end{proof}

\begin{proposition}
\label{prop:explicitboundUVf} Let $\Omega\subseteq\mathbb{R}^{N+1}$ be an
arbitrary open set, and let $f\in\mathcal{D}(\Omega)$. Moreover, let $\mu>0$
and $\alpha\in(0,1)$ be fixed. Then, we have
\begin{align*}
&  \mathcal{U}_{f,\Omega}^{\mu}(r)\leq%
\begin{cases}
c\big(\int_{0}^{\sqrt{r}}\frac{\omega_{f,\Omega}(s)}{s}\,ds+\Vert
f\Vert_{\mathcal{D}(\Omega)}e^{-\frac{\kappa}{r}}\big) & \text{if $0<r<1$%
},\\[0.05cm]%
cr^{\alpha}\Vert f\Vert_{\mathcal{D}(\Omega)} & \text{if $r\geq1$}.
\end{cases}
\end{align*}
If in addition $f\in\mathcal{D}_{\log}(\Omega)$, we have%

\begin{align*}
&  \mathcal{V}_{f,\Omega}^{\mu}(r)\leq%
\begin{cases}
c\big(\int_{0}^{\sqrt{r}}\frac{\omega_{f,\Omega}(s)}{s}\,ds+\Vert
f\Vert_{\mathcal{D}_{\log}(\Omega)}e^{-\frac{\kappa}{r}}\big) & \text{if
$0<r<1$},\\[0.05cm]%
cr^{\alpha}\Vert f\Vert_{\mathcal{D}_{\log}(\Omega)} & \text{if $r\geq1$}.
\end{cases}
\end{align*}
where $c,\kappa>0$ only depends on $\mu,N$ and $\alpha$, and
\[
\Vert f\Vert_{\mathcal{D}_{\log}(\Omega)}:=\int_{0}^{1}\frac{\omega_{f,\Omega
}(s)}{s}(1+|\log(s)|)ds+\Vert f\Vert_{L^{\infty}(\Omega)}%
\]

\end{proposition}

\begin{proof}
This is an immediate consequence of Proposition \ref{prop:stimaUf}, taking
into account the following identities (see \eqref{eq:boundwf} and Proposition
\ref{prop:funMN})
\[
\lbrack\omega]+\omega_{0}\leq c\cdot%
\begin{cases}
\Vert f\Vert_{\mathcal{D}(\Omega)}, & \text{if $\omega=\omega_{f,\Omega}$};\\
\Vert f\Vert_{\mathcal{D}_{\log}(\Omega)}, & \text{if $\omega=\mathcal{M}%
_{f,\Omega}$}.
\end{cases}
\]
This ends the proof.
\end{proof}

We conclude this part of the section with a couple of technical lemmas, which
will be repeatedly used in the sequel.

\begin{lemma}
\label{lem:omegasuppcpt} Let $f\in\mathcal{D}(\Omega)$, and let $\xi\in
\Omega,\,r > 0$ be such that $B = B_{r}(\xi)\Subset\Omega$. We assume that
$f\equiv0$ in $\Omega\setminus B$ \emph{(}i.e., $\mathrm{supp}(f)\subseteq
\overline{B}$\emph{)}. Then,
\[
\omega_{f,\Omega}(r) = \omega_{f,\overline{B}}(r)\quad\forall\,\,r > 0.
\]

\end{lemma}

\begin{proof}
First of all, since $B\Subset\Omega$ we have $\omega_{f,\Omega} \geq
\omega_{f,B}$ on $(0,\infty)$. To prove the reverse inequality, we fix $r > 0$
and we let $(x,t),(y,t)\in\Omega$ be such that
\[
d((x,t),(y,t)) = \|x-y\|\leq r.
\]
We then distinguish three cases.

\begin{itemize}
\item[(a)] $(x,t),(y,t)\in B$. In this case, by definition of $\omega_{f,B}$
we have
\begin{equation}
\label{eq:omegafcasea}|f(x,t)-f(y,t)|\leq\sup
_{\begin{subarray}{c} (z_1,s),(z_2,s)\in B \\ \|z_1-z_2\| \leq r \end{subarray}}%
|f(z_{1},s)-f(z_{2},s)| = \omega_{f,\overline{B}}(r).
\end{equation}

\item[(b)] $(x,t),(y,t)\notin B$. In this case, since $f\equiv0$ out of $B$,
we have
\begin{equation}
\label{eq:omegafcaseb}|f(x,t)-f(y,t)| = 0 \leq\omega_{f,\overline{B}}(r).
\end{equation}

\item[(c)] $(x,t)\in B,\,(y,t)\notin B$. In this last case, we consider the
segment
\[
\gamma(\tau)=(x+\tau(y-x),t)\qquad(\tau\in\lbrack0,1])
\]
and we first observe that, for every $0\leq\tau\leq1$, we have:
\[
d((x,t),\gamma(\tau))=\Vert\tau(y-x)\Vert=\sum_{j=1}^{N}\tau^{1/q_{j}}%
|x_{j}-y_{j}|\leq\tau^{1/q_{N}}\Vert x-y\Vert\leq r.
\]
On the other hand, since $\gamma(0)\in B$ and $\gamma(1)\notin B$, there
exists $\tau^{\ast}\in(0,1]$ such that $\gamma(\tau^{\ast})\in\partial B$.
Thus, since $f\equiv0$ in $\Omega\setminus B\supset\partial B$, we have
\begin{equation}%
\begin{split}
|f(x,t)-f(y,t)|  &  =|f(x,t)|=|f(x,t)-f(\gamma(\tau^{\ast}))|\\
&  \leq\sup
_{\begin{subarray}{c} (z_1,s),(z_2,s)\in \overline{B} \\ \|z_1-z_2\| \leq r \end{subarray}}%
|f(z_{1},s)-f(z_{2},s)|=\omega_{f,\overline{B}}(r).
\end{split}
\label{eq:omegafcasec}%
\end{equation}

\end{itemize}

Gathering \eqref{eq:omegafcasea}-to-\eqref{eq:omegafcasec}, we then conclude
that
\begin{align*}
\omega_{f,\Omega}(r) =\sup
_{\begin{subarray}{c} (x,t),(y,t)\in\Omega \\ d((x,t),(y,t)) \leq r \end{subarray}}%
|f(x,t)-f(y,t)| \leq\omega_{f,\overline{B}}(r),
\end{align*}
and the proof is complete.
\end{proof}

\begin{lemma}
\label{lem:doublingomega} There exists a structural constant $c>0$ such that,
for every $f\in\mathcal{D}(\Omega)$ every $r>0$ and every $\gamma>0$, the
following estimates hold true:
\begin{align}
\int_{\{\eta\in\mathbb{R}^{N+1}:\,d(\xi,\eta)>r\}}\frac{\omega_{f,\Omega
}(\gamma\,d(\xi,\eta))}{d(\xi,\eta)^{Q+3}}\,d\eta &  \leq c\int_{2r}^{\infty
}\frac{\omega_{f,\Omega}(\gamma s)}{s^{2}}\,ds\label{Lemma A}\\
\int_{\{\eta\in\mathbb{R}^{N+1}:\,d(\xi,\eta)<r\}}\frac{\omega_{f,\Omega
}(\gamma\,d(\xi,\eta))}{d(\xi,\eta)^{Q+2}}\,d\eta &  \leq c\int_{0}^{2r}%
\frac{\omega_{f,\Omega}(\gamma s)}{s}\,ds \label{Lemma B}%
\end{align}
Here, $Q\geq1$ is as in \eqref{eq:defQhomdim}, and $c$ is independent of both
$f$ and $r$.
\end{lemma}

\begin{proof}
The proof of both \eqref{Lemma A}-\eqref{Lemma B} is based on the fact that,
since $f\in\mathcal{D}(\Omega)$, the function $\omega_{f,\Omega}$ is
non-negative, non-decreasing and satisfies \eqref{eq:defDinirecall}; moreover,
we exploit the fact that $|B_{r}(\xi)|=\omega_{Q}\,r^{Q+2}$, see
\eqref{measure ball}. \vspace*{0.1cm}

(I)\thinspace\thinspace Proof of \eqref{Lemma A}. Taking into account
\eqref{eq:quasisymd}, we have
\begin{equation}%
\begin{split}
&  \int_{\{\eta\in\mathbb{R}^{N+1}:\,d(\xi,\eta)>r\}}\frac{\omega_{f,\Omega
}(\gamma\,d(\xi,\eta))}{d(\xi,\eta)^{Q+3}}\,d\eta\\
&  =\sum_{k=0}^{\infty}\int_{\{\eta\in\mathbb{R}^{N+1}:\,2^{k}r\leq d(\xi
,\eta)<2^{k+1}r\}}\frac{\omega_{f,\Omega}(\gamma\,d(\xi,\eta))}{d(\xi
,\eta)^{Q+3}}\,d\eta\\
&  \leq\sum_{k=0}^{\infty}\frac{\omega_{f,\Omega}(2^{k+1}\gamma r)}%
{(2^{k}r)^{Q+3}}\cdot\big|\{\eta\in\mathbb{R}^{N+1}:\,2^{k}r\leq d(\xi
,\eta)<2^{k+1}r\}\big|\\
&  \leq\sum_{k=0}^{\infty}\frac{\omega_{f,\Omega}(2^{k+1}\gamma r)}%
{(2^{k}r)^{Q+3}}\cdot|B_{2^{k+1}\bd\kappa r}(\xi)|\\
&  =\omega_{Q}(2\bd\kappa)^{Q+2}\sum_{k=0}^{\infty}\frac{\omega_{f,\Omega
}(2^{k+1}\gamma r)}{2^{k}r}.
\end{split}
\label{LemmaA1}%
\end{equation}
On the other hand, since $\omega_{f,\Omega}$ is non-decreasing we have
\begin{equation}%
\begin{split}
\int_{2r}^{\infty}\frac{\omega_{f,\Omega}(\gamma s)}{s^{2}}\,ds  &
=\sum_{k=0}^{\infty}\int_{2^{k+1}r}^{2^{k+2}r}\frac{\omega_{f,\Omega}(\gamma
s)}{s^{2}}\,ds\\
&  \geq\sum_{k=0}^{\infty}\frac{\omega_{f,\Omega}(2^{k+1}\gamma r)}%
{(2^{k+2}r)^{2}}\cdot2^{k+1}r=\frac{1}{8}\sum_{k=0}^{\infty}\frac
{\omega_{f,\Omega}(2^{k+1}\gamma r)}{2^{k}r}.
\end{split}
\label{LemmaA2}%
\end{equation}
By combining \eqref{LemmaA1}-\eqref{LemmaA2}, we immediately get
\eqref{Lemma A}. \vspace*{0.1cm}

(II)\thinspace\thinspace Proof of \eqref{Lemma B}. Using once again
\eqref{eq:quasisymd}, we have
\begin{equation}%
\begin{split}
&  \int_{\{\eta\in\mathbb{R}^{N+1}:\,d(\xi,\eta)<r\}}\frac{\omega_{f,\Omega
}(\gamma\,d(\xi,\eta))}{d(\xi,\eta)^{Q+2}}\,d\eta\\
&  \qquad=\sum_{k=0}^{\infty}\int_{\{\eta\in\mathbb{R}^{N+1}:\,r/2^{k+1}%
<d(\xi,\eta)\leq r/2^{k}\}}\frac{\omega_{f,\Omega}(\gamma\,d(\xi,\eta))}%
{d(\xi,\eta)^{Q+2}}\,d\eta\\
&  \qquad\leq\sum_{k=0}^{\infty}\frac{\omega_{f,\Omega}(\gamma r/2^{k}%
)}{(r/2^{k+1})^{Q+2}}\cdot\big|\{\eta\in\mathbb{R}^{N+1}:\,r/2^{k+1}%
<d(\xi,\eta)\leq r/2^{k}\}\big|\\
&  \qquad\leq\sum_{k=0}^{\infty}\frac{\omega_{f,\Omega}(\gamma r/2^{k}%
)}{(r/2^{k+1})^{Q+2}}\cdot|B_{\bd\kappa r/2^{k}}(\xi)|\\
&  \qquad=\omega_{Q}(2\bd\kappa)^{Q+2}\sum_{k=0}^{\infty}\omega_{f,\Omega
}(\gamma\,r/2^{k}).
\end{split}
\label{LemmaB1}%
\end{equation}
On the other hand, since $\omega_{f,\Omega}$ is non-decreasing we have
\begin{equation}%
\begin{split}
\int_{0}^{2r}\frac{\omega_{f,\Omega}(s)}{s}\,ds  &  =\sum_{k=0}^{\infty}%
\int_{r/2^{k}}^{r/2^{k-1}}\frac{\omega_{f,\Omega}(s)}{s}\,ds\\
&  \geq\sum_{k=0}^{\infty}\frac{\omega_{f,\Omega}(\gamma r/2^{k})}{r/2^{k-1}%
}\cdot\frac{r}{2^{k}}=\frac{1}{2}\sum_{k=0}^{\infty}\omega_{f,\Omega}(\gamma
r/2^{k}).
\end{split}
\label{LemmaB2}%
\end{equation}
By combining \eqref{LemmaB1}-\eqref{LemmaB2}, we immediately get \eqref{Lemma B}.
\end{proof}

\subsection{Fundamental solution and representation formulas for the operator
with coefficients only depending on $t$}

In this section we collect some results established in \cite{BB, BP}
concerning the KFP o\-pe\-ra\-tors $\mathcal{L}$ with \emph{coefficients
$a_{ij}$ only depending on $t$}, that is,
\begin{equation}
\mathcal{L}u=\sum_{i,j=1}^{m_{0}}a_{ij}(t)\partial_{x_{i}x_{j}}^{2}%
u+\sum_{k,j=1}^{N}b_{jk}x_{k}\partial_{x_{j}}u-\partial_{t}u.
\label{eq:LLsolot}%
\end{equation}
Throughout what follows, we tacitly understand that $\LL$ satisfies the
structural assumptions (H1)-(H2) stated in the Introduction. \vspace*{0.1cm}

We begin by stating a result proved in \cite{BP}, which provides an explicit
expression for the global fundamental solution (heat kernel) of $\mathcal{L}$.

\begin{theorem}
[Fundamental solution for operators as in \eqref{eq:LLsolot}]%
\label{Thm fund sol coeff t dip} Let $C(t,s)$ be the $N\times N$ matrix
defined as follows:
\begin{equation}
C(t,s)=\int_{s}^{t}E(t-\sigma)\cdot%
\begin{pmatrix}
A_{0}(\sigma) & 0\\
0 & 0
\end{pmatrix}
\cdot E(t-\sigma)^{T}\,d\sigma\quad(\text{with $t>s$}) \label{eq-EC}%
\end{equation}
\emph{(}we recall that $E(\sigma)=\exp(-\sigma B)$, see \eqref{B}\emph{)}.
Then, $C(t,s)$ is \emph{sym\-me\-tric and positive definite} for every $t>s$.
Moreover, if we define
\begin{equation}%
\begin{split}
&  \Gamma(x,t;y,s)\\
&  =\frac{1}{(4\pi)^{N/2}\sqrt{\det C(t,s)}}e^{-\frac{1}{4}\langle
C(t,s)^{-1}(x-E(t-s)y),\,x-E(t-s)y\rangle}\cdot\mathbf{1}_{\{t>s\}}%
\end{split}
\label{eq.exprGammapernoi}%
\end{equation}
\emph{(}where $\mathbf{1}_{A}$ denotes the indicator function of a set
$A$\emph{)}, then $\Gamma$ enjoys the following properties, so that $\Gamma$
is the \emph{fundamental solution} for $\mathcal{L}$ with pole at $(y,s)$.

\begin{enumerate}
\item In the open set $\mathcal{O}:=\{(x,t;y,s)\in\mathbb{R}^{2N+2}%
:\,(x,t)\neq(y,s)\}$, the function $\Gamma$ is \emph{jointly continuous} in
$(x,t;y,s)$ and $C^{\infty}$ with respect to $x,y$. Mo\-re\-o\-ver, for every
multi-indices $\alpha,\beta$ the functions
\[
\partial_{x}^{\alpha}\partial_{y}^{\beta}\Gamma=\frac{\partial^{\alpha+\beta
}\Gamma}{\partial x^{\alpha}\partial y^{\beta}}%
\]
are \emph{jointly continuous} in $(x,t;y,s)\in\mathcal{O}$. Finally, $\Gamma$
and $\partial_{x}^{\alpha}\partial_{y}^{\beta}\Gamma$ are \emph{Lip\-schitz
continuous} with respect to $t,s$ in any region $\mathcal{R}$ of the form
\[
\mathcal{R}=\{(x,t;y,s)\in\mathbb{R}^{2N+2}:\,H\leq s+\delta\leq t\leq K\},
\]
where $H,K\in\mathbb{R}$ and $\delta>0$ are arbitrarily fixed.

\item For every fixed $y\in\mathbb{R}^{N}$ and $t>s$, we have
\[
\lim_{|x|\rightarrow+\infty}\Gamma(x,t;y,s)=0.
\]

\item For every fixed $(y,s)\in\mathbb{R}^{N+1}$, we have
\[
\mathcal{L}\Gamma(\cdot;y,s)(x,t)=0\qquad\text{for every $x\in\mathbb{R}^{N}$
and a.e.\thinspace$t$}.
\]

\item For every $x\in\mathbb{R}^{N}$ and $t>s$, we have
\begin{equation}
\int_{\mathbb{R}^{N}}\Gamma(x,t;y,s)\,dy=1. \label{eq:integralGamma1}%
\end{equation}

\item For every $f\in C(\mathbb{R}^{N})\cap L^{\infty}(\mathbb{R}^{N})$ and
$s\in\mathbb{R}$, the function
\[
u(x,t)=\int_{\mathbb{R}^{N}}\Gamma(x,t;y,s)f(y)\,dy
\]
is the unique solution to the Cauchy problem
\begin{equation}%
\begin{cases}
\mathcal{L}u=0 & \text{in $\mathbb{R}^{N}\times(s,\infty)$}\\
u(\cdot,s)=f &
\end{cases}
\label{eq:pbCauchyThmBP}%
\end{equation}
In particular, $u(\cdot,s)\rightarrow f$ uniformly in $\mathbb{R}^{N}$ as
$t\rightarrow s^{+}$.
\end{enumerate}

Finally, the function $\Gamma^{\ast}(x,t;y,s):=\Gamma(y,s;x,t)$ satisfies dual
properties of \emph{(2)-(4)} with respect to the \emph{formal adjoint} of
$\mathcal{L}$, that is,
\[
\textstyle\mathcal{L}^{\ast}=\sum_{i,j=1}^{m_{0}}a_{ij}(s)\partial_{y_{i}%
y_{j}}-\sum_{k,j=1}^{N}b_{jk}y_{k}\partial_{y_{i}}+\partial_{s},
\]
and thus $\Gamma^{\ast}$ is the fundamental solution of $\mathcal{L}^{\ast}$.
\end{theorem}

The precise definition of \emph{solution to the Cauchy problem}
\eqref{eq:pbCauchyThmBP} requires some care, see \cite[Definitions 1.2 and
1.3]{BP} for the details. \medskip

In the particular case when the coefficients $a_{ij}$ of $\mathcal{L}$ are
\emph{constant}, the results of the previous theorem apply in a simpler form
(see also \cite{LP}).

\begin{theorem}
[Fundamental solution for operators with constant coefficients]%
\label{Thm fund sol cost coeff} Let $\alpha>0$ be fixed, and let
$\mathcal{L}_{\alpha}$ be the constant coefficient KFP operator
\begin{equation}
\mathcal{L}_{\alpha}u=\alpha\sum_{i=1}^{m_{0}}\partial_{x_{i}x_{i}}^{2}%
u+\sum_{k,j=1}^{N}b_{jk}x_{k}\partial_{x_{j}}u-\partial_{t}u. \label{L-alpha}%
\end{equation}
Moreover, let $\Gamma_{\alpha}$ be the fundamental solution of $\mathcal{L}%
_{\alpha}$, whose existence is guaranteed by Theorem
\ref{Thm fund sol coeff t dip}. Then, the following facts hold true:

\begin{enumerate}
\item $\Gamma_{\alpha}$ is a \emph{kernel of convolution type}, that is,
\begin{equation}
\label{eq:Gammaalfaconvolution}%
\begin{split}
\Gamma_{\alpha}(x,t;y,s)  &  =\Gamma_{\alpha}\big(x-E(t-s)y,t-s;0,0\big)\\
&  = \Gamma_{\alpha}\big((y,s)^{-1}\circ(x,t);0,0\big);
\end{split}
\end{equation}

\item the matrix $C(t,s)$ in \eqref{eq-EC} takes the simpler form
\begin{equation}
C(t,s)=C_{0}(t-s), \label{C_0}%
\end{equation}
where $C_{0}(\tau)$ is the $N\times N$ matrix defined as
\[
C_{0}(\tau)=\alpha\int_{0}^{\tau}E(t-\sigma)\cdot%
\begin{pmatrix}
I_{m_{0}} & 0\\
0 & 0
\end{pmatrix}
\cdot E(t-\sigma)^{T}d\sigma\qquad(\tau>0).
\]
Furthermore, one has the `homogeneity property'
\begin{equation}
C_{0}(\tau)=D_{0}(\sqrt{\tau})C_{0}(1)D_{0}(\sqrt{\tau})\qquad\forall
\,\,\tau>0. \label{C omogenea}%
\end{equation}

\end{enumerate}

In particular, by combining \eqref{eq.exprGammapernoi} with
\eqref{C_0}-\eqref{C omogenea}, we can write%
\begin{equation}%
\begin{split}
&  \Gamma_{\alpha}(x,t;0,0)=\frac{1}{(4\pi\alpha)^{N/2}\sqrt{\det C_{0}(t)}%
}e^{-\frac{1}{4\alpha}\left\langle C_{0}(t)^{-1}x,x\right\rangle }\\
&  \qquad=\frac{1}{(4\pi\alpha)^{N/2}t^{Q/2}\sqrt{\det C_{0}(1)}}e^{-\frac
{1}{4\alpha}\langle C_{0}(1)^{-1}\big(D_{0}\big(\frac{1}{\sqrt{t}%
}\big)x\big),\,D_{0}\big(\frac{1}{\sqrt{t}}\big)x\rangle}.
\end{split}
\label{eq.exprGammaalfa}%
\end{equation}

\end{theorem}

Now we have recalled Theorems \ref{Thm fund sol coeff t dip}%
-\ref{Thm fund sol cost coeff}, we collect in the next theorem some \emph{fine
properties} of $\Gamma$ and of its derivatives which will be extensively used
in the sequel. In what follows, if $\bd\alpha= (\alpha_{1},\ldots,\alpha
_{N})\in(\mathbb{N}\cup\{0\})^{N}$, we set
\[
\textstyle|\bd\alpha|:=\sum_{i=1}^{N}\alpha_{i}\qquad\text{and}\qquad
\omega(\bd\alpha):=\textstyle\sum_{i=1}^{N}q_{i}\alpha_{i},
\]
where the $q_{i}$'s are the exponents appearing in the dilation $D_{0}%
(\lambda)$, see \eqref{dilations}.

\begin{theorem}
[{See \cite[Thm.s 3.5 and 3.9]{BB}}]\label{thm:finepropGamma} Let $\Gamma$ be
as in Theorem \ref{Thm fund sol coeff t dip}, and let $\nu>0$ be as in
\eqref{nu}. Then, the following assertions hold:

\begin{enumerate}
\item there exists a structural constant $c_{1}>0$ and, for every pair of
multi-in\-dices $\bd\alpha_{1},\bd\alpha_{2}\in(\N\cup\{0\})^{N}$, there
exists $c=c(\nu,\bd\alpha_{1},\bd\alpha_{2})>0$, such that
\begin{equation}%
\begin{split}
\left\vert D_{x}^{\bd\alpha_{1}}D_{y}^{\bd\alpha_{2}}\Gamma
(x,t;y,s)\right\vert  &  \leq\frac{{c}}{(t-s)^{\omega(\bd\alpha_{1}%
+\bd\alpha_{2})/2}}\,\Gamma_{c_{1}\nu^{-1}}(x,t;y,s)\\
&  \leq\frac{c}{d((x,t),(y,s))^{Q+\omega(\bd\alpha_{1}+\bd\alpha_{2})}},
\end{split}
\label{eq:mainestim}%
\end{equation}
for every $(x,t),(y,s)\in\mathbb{R}^{N}$ with $t\neq s$. In particular, we
have
\[
\left\vert D_{x}^{\bd\alpha_{1}}D_{y}^{\bd\alpha_{2}}\Gamma
(x,t;y,s)\right\vert \leq\frac{c}{d((x,t),(y,s))^{Q+\omega(\bd\alpha
_{1}+\bd\alpha_{2})}}\quad\forall\,\,(x,t)\neq(y,s).
\]

\item Let $\eta= (y,s)\in\mathbb{R}^{N+1}$ be fixed, and let $\bd\alpha
\in(\N\cup\{0\})^{N}$ be a multi-index. Then, there exists a constant
$c=c(\bd\alpha,\nu)>0$ such that
\begin{equation}
\label{eq:meanvalueGamma}|D_{x}^{\bd{\alpha}}\Gamma(\xi_{1},\eta
)-D_{x}^{\bd{\alpha}}\Gamma(\xi_{2},\eta)|\leq c\frac{d(\xi_{1},\xi_{2}%
)}{d(\xi_{1},\eta)^{Q+\omega(\bd\alpha)+1}}%
\end{equation}
for every $\xi_{1}=(x_{1},t_{1}),\xi_{2}=(x_{2},t_{2})\in\mathbb{R}^{N+1}$
such that
\[
d(\xi_{1},\eta)\geq4\bd{\kappa}d(\xi_{1},\xi_{2})>0.
\]

\end{enumerate}
\end{theorem}

We conclude this subsection by recalling a {representation for\-mula for
fun\-ctions $u\in\mathcal{S}^{0}(\tau,T)$. }

\begin{theorem}
[{See \cite[Thm.\,3.11 and Cor.\,3.12]{BB}}]\label{thm:repru} Let
$T\in\mathbb{R}$ be fixed, and let $\tau<T$. Moreover, let $\mathcal{L}$ be as
in \eqref{eq:LLsolot}, and let $u\in\mathcal{S}^{0}(\tau;T)$. Then,
\begin{equation}
u(x,t)=-\int_{\mathbb{R}^{N}\times(\tau,t)}\!\!\Gamma(x,t;y,s)\mathcal{L}%
u(y,s)\,dy\,ds\quad\forall\,\,(x,t)\in S_{T}. \label{repr formula u}%
\end{equation}
Furthermore, given any $1\leq k\leq m_{0}$, the function $\partial_{x_{k}}u$
exists \emph{pointwise} on $S_{T}$ in the classical sense, and for every
$(x,t)\in S_{T}$ we have
\begin{equation}
\partial_{x_{k}}u(x,t)=-\int_{\mathbb{R}^{N}\times(\tau,t)}\!\!\partial
_{x_{k}}\Gamma(x,t;y,s)\mathcal{L}u(y,s)\,dy\,ds. \label{repr formula ux}%
\end{equation}

\end{theorem}

Starting from the repre\-sen\-ta\-tion formula \eqref{repr formula u}, in
\cite{BB} the Authors proved a representation formula for $\partial
_{x_{i}x_{j}}^{2}u$ (when $u\in\mathcal{S}^{0}(\tau,T)$ and $1\leq i,j\leq
m_{0}$) under the assumption that $\mathcal{L}u$ is partially
H\"{o}lder-continuous w.r.t.\thinspace$x$, uniformly in $t$ (see, precisely,
\cite[Cor.\,3.12 and Thm.\,3.14]{BB}). In Section \ref{sec:operatort}, we will
extend such formulas to all functions $u\in\mathcal{S}^{0}(\tau,T)$ with
$\mathcal{L}u$ only belonging to $\mathcal{D}(S_{T})$.

\subsection{Interpolation inequalities}

We conclude this preliminary section by stating some interpolation
inequalities, established in \cite{BB}, which will be exploited in the proof
of Theorem \ref{Thm main space}. \vspace*{0.1cm}

\begin{theorem}
[{See \cite[Thm. 4.2]{BB}}]\label{Thm interpolaz} Let $T\in\mathbb{R}$ be
arbitrarily fixed, and let $u\in\mathcal{S}^{0}(S_{T})$. Then, for every
$\alpha\in(0,1)$, $\xi\in S_{T}$ and $r>0$ we have
\[
u,\,\partial_{x_{k}}u\in C^{\alpha}(B_{r}^{T}(\xi))\quad\forall\,\,1\leq k\leq
m_{0},
\]
where we set:
\[
B_{r}^{T}(\xi)=B_{r}(\xi)\cap S_{T}.
\]
Moreover, the following \emph{interpolation inequality} holds: \emph{for every
$\alpha\in(0,1)$ and every $r>0$ there exist constants $c>0$ and $\gamma>1$
such that}
\begin{equation}%
\begin{split}
&  \sum_{h=1}^{m_{0}}\Vert\partial_{x_{h}u}\Vert_{C^{\alpha}(B_{r}^{T}(\xi
))}+\Vert u\Vert_{C^{\alpha}(B_{r}^{T}(\xi))}\\
&  \qquad\leq\varepsilon\bigg\{\sum_{h,k=1}^{m_{0}}\Vert\partial_{x_{k}x_{h}%
}^{2}u\Vert_{L^{\infty}(B_{4r}^{T}(\xi))}+\Vert Yu\Vert_{L^{\infty}(B_{4r}%
^{T}(\xi))}\bigg\}\\
&  \qquad\qquad+\frac{c}{\varepsilon^{\gamma}}\Vert u\Vert_{L^{\infty}%
(B_{4r}^{T}(\xi))}.
\end{split}
\label{disug interpolaz}%
\end{equation}
\emph{and this estimate holds for every $\e\in(0,1),\,\xi\in S_{T}$ and
$u\in\mathcal{S}^{0}(S_{T})$}. We stress that the con\-stant $c$ depends on
$r$ and $\alpha$, but is independent of $\e,\xi$ and $u$.
\end{theorem}

\begin{remark}
\label{rem:HolderDini} We explicitly highlight, for a future reference, the
following easy yet important fact: \emph{if $\Omega\subseteq\mathbb{R}^{N+1}$
is an arbitrary open set and if $f\in C^{\alpha}(\Omega)$ for some $\alpha
\in(0,1)$, then $f\in\mathcal{D}(\Omega)$ and we have the estimate}
\begin{align*}
\omega_{f,\Omega}(r)  &  =\sup
_{\begin{subarray}{c} (x,t),(y,t)\in\Omega \\ \|x-y\| \leq r \end{subarray}}%
|f(x,t)-f(y,t)|\\
&  \leq|f|_{C^{\alpha}(\Omega)}\cdot\sup
_{\begin{subarray}{c} (x,t),(y,t)\in\Omega \\ \|x-y\| \leq r \end{subarray}}%
d\big((x,t),(y,t)\big)^{\alpha}\leq r^{\alpha}|f|_{C^{\alpha}(\Omega)},
\end{align*}
\emph{where we have also used \eqref{eq:explicitd}}. As a consequence, we
easily deduce%
\[
\mathcal{M}_{f,\Omega}(r)\leq cr^{\alpha}|f|_{C^{\alpha}(\Omega)}.
\]

\end{remark}

\section{Operators with coefficients only depending on $t$}

\label{sec:operatort} In this section we establish some `weaker' versions of
Theorems \ref{Thm main space}-\ref{Thm main time} for KFP operators with
coefficients $a_{ij}$ only depending on $t$; we will use these results as a
crucial tool to prove Theorems \ref{Thm main space}-\ref{Thm main time} in
Section \ref{sec:proofMainThm}. Throughout what follows, we tacitly understand
that $\LL$ is as in \eqref{eq:LLsolot}, and that the structural assumptions
(H1)-(H2) stated in the Introduction are satisfied (without the need of repeat
it); moreover, $\Gamma$ denotes the fundamental solution of $\LL$, as in
Theorem \ref{Thm fund sol coeff t dip}. \vspace*{0.1cm}

To begin with, we extend to all functions $u\in\mathcal{S}^{0}(\tau,T)$ with
$\mathcal{L}u\in\mathcal{D}(S_{T})$ the representation formula for
$\partial_{x_{i}x_{j}}^{2}u$ (where $1\leq i,j\leq m_{0}$) proved in
\cite[Thm.\,3.14]{BB} under the more restrictive assumption that
$\mathcal{L}u$ is \emph{partially H\"{o}lder continuous w.r.t.\thinspace$x$}.
In this direction, a first key tool is the following proposition.

\begin{proposition}
\label{prop:Analoga313} There exist structural constants $c,\,\mu>0$ such
that, for every fixed $T\in\mathbb{R}$, every $f\in\mathcal{D}(S_{T})$,
$x\in\mathbb{R}^{N}$ and $\tau<t<T$, one has
\begin{equation}%
\begin{split}
\int_{\mathbb{R}^{N}\times(\tau,t)}  &  |\partial_{x_{i}x_{j}}^{2}%
\Gamma(x,t;y,s)|\cdot\omega_{f,S_{T}}(\Vert E(s-t)x-y\Vert)\,dy\,ds\\
&  \leq{c}\,\mathcal{U}_{f,S_{T}}^{\mu}(\sqrt{t-\tau}),
\end{split}
\label{eq:integraldexixjconv}%
\end{equation}
where $\mathcal{U}_{f,S_{T}}^{\mu}$ is as in \eqref{eq:defUmuIntro}. In
particular, we have
\begin{equation}%
\begin{split}
\int_{\mathbb{R}^{N}\times(t-\e,t)}  &  |\partial_{x_{i}x_{j}}^{2}%
\Gamma(x,t;y,s)|\cdot\omega_{f,S_{T}}(\Vert E(s-t)x-y\Vert)\,dy\,ds\rightarrow
0\\
&  \text{\emph{uniformly w.r.t.\thinspace$(x,t)\in\mathbb{R}^{N+1}$} as
$\e\rightarrow0^{+}$}.
\end{split}
\label{eq:integraldexixjunifconv}%
\end{equation}

\end{proposition}

\begin{proof}
The proof of this proposition is similar to that of \cite[Prop.\,3.13]{BB}; we
sketch it here for the sake of completeness, but we omit the details.
\vspace*{0.05cm} First of all, by combining estimate \eqref{eq:mainestim} with
\eqref{eq:Gammaalfaconvolution}-\eqref{eq.exprGammaalfa}, we have
\begin{align*}
&  \int_{\mathbb{R}^{N}\times(\tau,t)}|\partial_{x_{i}x_{j}}^{2}%
\Gamma(x,t;y,s)|\cdot\omega_{f,S_{T}}(\Vert E(s-t)x-y\Vert)\,dy\,ds\\
&  \quad\leq c\int_{\mathbb{R}^{N}\times(\tau,t)}\frac{e^{-\mu|D_{0}%
\big(\frac{1}{\sqrt{t-s}}\big)(E(s-t)x-y)|^{2}}}{(t-s)^{Q/2+1}}\cdot
\omega_{f,S_{T}}(\Vert E(s-t)x-y\Vert)\,dy\,ds\\
&  \quad\leq c\int_{\tau}^{t}\frac{1}{(t-s)^{Q/2+1}}\cdot\mathcal{I}%
_{x,t}(s)\,ds=(\bigstar),
\end{align*}
where we have introduced the notation
\[
\mathcal{I}_{x,t}(s)=\int_{\RN}e^{-\mu|D_{0}\big(\frac{1}{\sqrt{t-s}%
}\big)(E(s-t)x-y)|^{2}}\,\omega_{f,S_{T}}(\Vert E(s-t)x-y\Vert)\,dy,
\]
and $c,\mu>0$ are structural constants. We explicitly mention that, in the
above estimate, we have also used \eqref{LP 2.20} and the non-singularity of
$E(1)$. Then, by performing in the integral $\mathcal{I}_{x,t}(s)$ the change
of variables
\[
y=E(s-t)x-D_{0}(\sqrt{t-s})z,
\]
and by exploiting the $1$-homogeneity of $\Vert\cdot\Vert$, we obtain
\begin{align*}
(\bigstar)  &  =c\int_{\tau}^{t}\frac{1}{t-s}\bigg(\int_{\mathbb{R}^{N}%
}e^{-\mu|z|^{2}}\omega_{f,S_{T}}(\sqrt{t-s}\Vert z\Vert)\,dz\bigg)ds\\
&  =c\int_{\RN}e^{-\mu|z|^{2}}\Big(\int_{\tau}^{t}\frac{\omega_{f,S_{T}}%
(\sqrt{t-s}\Vert z\Vert)}{t-s}\,ds\Big)dz\\
&  =c\int_{\RN}e^{-\mu|z|^{2}}\Big(\int_{0}^{\sqrt{t-\tau}\Vert z\Vert}%
\frac{\omega_{f,S_{T}}(\sigma)}{\sigma}\,d\sigma\Big)dz=c\,\mathcal{U}%
_{f,S_{T}}^{\mu}(\sqrt{t-\tau}).
\end{align*}
Since the constants $c,\mu>0$ only depend on $\nu$, this is exactly the
desired \eqref{eq:integraldexixjconv}. As regards
\eqref{eq:integraldexixjunifconv}, it is a direct consequence of
\eqref{eq:integraldexixjconv} and Lemma \ref{lem:Ufwelldef}.
\end{proof}

Thanks to Proposition \ref{prop:Analoga313}, we can now prove the following theorem.

\begin{theorem}
\label{Corollary repr formula uxx} For $T>t>\tau>-\infty$, let $u\in
\mathcal{S}^{0}(\tau;T)$ be such that $\mathcal{L}u\in\mathcal{D}(S_{T})$.
Then, for every $1\leq i,j\leq m_{0}$ the second-order derivatives
$\partial_{x_{i}x_{j}}^{2}u$ \emph{exist pointwinse} on $S_{T}$ in the
classical sense, and for every $(x,t)\in S_{T}$ we have
\begin{equation}
\partial_{x_{i}x_{j}}^{2}u(x,t)=\int_{\mathbb{R}^{N}\times(\tau,t)}%
\partial_{x_{i}x_{j}}^{2}\Gamma(x,t;y,s)\big[\mathcal{L}%
u(E(s-t)x,s)-\mathcal{L}u(y,s)\big]\,dy\,ds. \label{repr formula u_xx}%
\end{equation}

\end{theorem}

\begin{proof}
The proof of this result is essentially analogous to that of \cite[Thm.\,3.14]%
{BB}, but we use Proposition \ref{prop:Analoga313} in place of
\cite[Prop.\,3.13]{BB}. For the sake of completeness we sketch the argument,
but we refer to \cite{BB} for the details. \vspace*{0.05cm} First of all we
observe that, owing to Proposition \ref{prop:Analoga313} (and taking into
account the very definition of $\omega_{f,S_{T}}$ in Definition \ref{Def Dini}%
), we have
\begin{align*}
&  \bigg|\int_{\mathbb{R}^{N}\times(\tau,t)}|\partial_{x_{i}x_{j}}^{2}%
\Gamma(x,t;y,s)|\cdot|\mathcal{L}u(E(s-t)x,s)-\mathcal{L}%
u(y,s)|\,dy\,ds\bigg|\\
&  \quad\leq\bigg|\int_{\mathbb{R}^{N}\times(\tau,t)}|\partial_{x_{i}x_{j}%
}^{2}\Gamma(x,t;y,s)|\cdot\omega_{\mathcal{L}u,S_{T}}(\Vert E(s-t)x-y\Vert
)\,dy\,ds\bigg|\\
&  \quad\leq c\,\mathcal{U}_{\mathcal{L}u,S_{T}}^{\mu}(\sqrt{|t-\tau|}%
)\quad\forall\,\,(x,t)\in S_{T}%
\end{align*}
(where $c,\mu>0$ only depend on $\nu$); hence, the function
\[
g(x,t):=\int_{\mathbb{R}^{N}\times(\tau,t)}\partial_{x_{i}x_{j}}^{2}%
\Gamma(x,t;y,s)\big[\mathcal{L}u(E(s-t)x,s)-\mathcal{L}u(y,s)\big]\,dy\,ds
\]
is {well-defined} on $S_{T}$. We then turn to prove that $\partial_{x_{i}%
x_{j}}^{2}u=g$ {po\-int\-wi\-se in $S_{T}$} by an approximation argument. To
this end, we fix $0<\e\ll1$ and we define
\[
v_{\varepsilon}(x,t):=-\int_{\mathbb{R}^{N}\times(\tau,t-\e)}\partial_{x_{j}%
}\Gamma(x,t;\cdot)\,\mathcal{L}u\,dy\,ds.
\]
Now, using \eqref{repr formula ux} and taking into account the regularity of
$\Gamma$, we see that \vspace*{0.1cm}

(i)\thinspace\thinspace$v_{\varepsilon}\in C(S_{T})$ and $v_{\varepsilon
}\rightarrow\partial_{x_{j}}u$ pointwise in $S_{T}$ as $\e\rightarrow0^{+}$;

(ii)\thinspace\thinspace$v_{\varepsilon}$ is continuously differentiable
w.r.t.\thinspace$x_{i}$ on $S_{T}$, and
\begin{align*}
\partial_{x_{i}}v_{\varepsilon}(x,t)  &  =-\int_{\mathbb{R}^{N}\times
(\tau,t-\e)}\partial_{x_{i}x_{j}}^{2}\Gamma(x,t;y,s)\times\\
&  \qquad\quad\times\big[\mathcal{L}u(E(s-t)x,s)-\mathcal{L}%
u(y,s)\big]\,dy\,ds\quad\forall\,\,(x,t)\in S_{T}.
\end{align*}
We explicitly stress that, in computing $\partial_{x_{i}}v_{\e}$, we have also
used \eqref{eq:integralGamma1}. On the other hand, again by Proposition
\ref{prop:Analoga313}, we have
\begin{align*}
&  |\partial_{x_{i}}v_{\varepsilon}(x,t)-g(x,t)|\\
&  \qquad=\int_{\mathbb{R}^{N}\times(t-\e,t)}|\partial_{x_{i}x_{j}}^{2}%
\Gamma(x,t;\cdot)|\,|\mathcal{L}u(E(s-t)x,s)-\mathcal{L}u|\,dy\,ds\\
&  \qquad\leq\int_{\mathbb{R}^{N}\times(t-\e,t)}|\partial_{x_{i}x_{j}}%
^{2}\Gamma(x,t;\cdot)|\ \cdot\omega_{\mathcal{L}u,S_{T}}(\Vert E(s-t)x-y\Vert
)\,dy\,ds\\
&  \qquad\leq c\,\mathcal{U}_{\mathcal{L}u,S_{T}}^{\mu}(\sqrt{\e}%
)\qquad\text{for every $(x,t)\in S_{T}$},
\end{align*}
and this shows that $\partial_{x_{i}}v_{\varepsilon}\rightarrow g$ {uniformly
on $S_{T}$ as $\e\rightarrow0^{+}$}. Gathering these facts, by standard
results, we then conclude that there exists
\[
\partial_{x_{i}x_{j}}^{2}u=\partial_{x_{i}}(\partial_{x_{j}}u)=g\quad
\text{pointwise in $S_{T}$}.
\]
This ends the proof.
\end{proof}

Now that we have established the representation formula
\eqref{repr formula u_xx}, we can prove the announced weaker version of
Theorem \ref{Thm main space} for KFP operators with coefficients $a_{ij}$ only
depending on $t$. Actually, we will deduce this result from the following
ge\-neral theorem, which will be used as a key tool in the next section.

\begin{theorem}
[Singular integrals and Dini-continuous functions]\label{Thm Dini sing int}
For e\-very fixed $T\in\mathbb{R}$ and $-\infty<\tau<T$, let us introduce the
function space
\[
\mathcal{D}(\tau;T):=\{f\in\mathcal{D}(S_{T}):\,\text{$f(x,t)=0$ for every
$t\leq\tau$}\},
\]
and define, on this space $\mathcal{D}(\tau;T)$, the linear operator
\[
f\mapsto T_{ij}f(x,t):=\int_{\mathbb{R}^{N}\times(\tau,t)}\partial_{x_{i}%
x_{j}}^{2}\Gamma(x,t;y,s)\,\big[f(E(s-t)x,s)-f(y,s)\big]\,dy\,ds.
\]
Then, there exist structural constants $c,\mu>0$ such that
\begin{align}
\Vert T_{ij}f\Vert_{L^{\infty}(S_{T})}  &  \leq{c}\,\mathcal{U}_{f,S_{T}}%
^{\mu}(\sqrt{T-\tau})\label{sup bound}\\
\omega_{T_{ij}f}(r)  &  \leq c\,\mathcal{M}_{f,S_{T}}(cr)\quad\forall\,\,r>0.
\label{Dini bound}%
\end{align}
Here, $\mathcal{M}_{f,S_{T}}$ is the function defined in \eqref{eq:defcontinM}.
\end{theorem}

\begin{proof}
The proof of this theorem is similar to that of \cite[Thm.\,3.17]{BB}, where
the Authors deal with the particular case of functions $f\in\mathcal{D}%
(S_{T})$ such that
\[
\omega_{f,S_{T}}(r)\leq c\,r^{\alpha}\quad\text{for some $\alpha\in(0,1)$}%
\]
(that is, $f$ is partially H\"{o}lder continuous w.r.t.\thinspace$x$,
uniformly in $t$). We then limit ourselves to sketch the argument exploited in
\cite{BB}, but we highlight how the main estimates have to be modified in our
more general context. \vspace*{0.05cm} Let $f\in\mathcal{D}(\tau;T)$ be
arbitrarily fixed. Since $f(\cdot,t)\equiv0$ for all $t\leq\tau$, we clearly
have $T_{ij}f(x,t)=0$ for all $x\in\mathbb{R}^{N}$ and $t\leq\tau$. Thus, it
is readily seen that
\begin{equation}
\Vert T_{ij}f\Vert_{L^{\infty}(S_{T})}=\Vert T_{ij}f\Vert_{L^{\infty}(\Omega
)}\quad\text{and}\quad\omega_{T_{ij}f,S_{T}}\equiv\omega_{T_{ij}f,\Omega},
\label{eq:TijsoloOmega}%
\end{equation}
where we have set $\Omega=\RN\times(\tau,T)$. On account of
\eqref{eq:TijsoloOmega}, to prove \eqref{sup bound}-\eqref{Dini bound} it then
suffices to study the function $T_{ij}f(x,t)$ only for $(x,t)\in\Omega$.
\vspace*{0.05cm} As regards \eqref{sup bound} we observe that, by Proposition
\ref{prop:Analoga313}, we have
\begin{align*}
&  |T_{ij}f(x,t)|\leq\int_{\mathbb{R}^{N}\times(\tau,t)}|\partial_{x_{i}x_{j}%
}^{2}\Gamma(x,t;y,s)|\cdot|f(E(s-t)x,s)-f(y,s)|\,dy\,ds\\
&  \quad\leq\int_{\mathbb{R}^{N}\times(\tau,t)}|\partial_{x_{i}x_{j}}%
^{2}\Gamma(x,t;y,s)|\cdot\omega_{f,S_{T}}(\Vert E(s-t)x-y\Vert)\,dy\,ds\\
&  \quad\leq c\,\mathcal{U}_{f,S_{T}}^{\mu}(\sqrt{t-\tau})\leq c\,\mathcal{U}%
_{f,S_{T}}^{\mu}(\sqrt{T-\tau})\quad\forall\,\,(x,t)\in\Omega,
\end{align*}
where $c,\mu>0$ are structural constants. Hence,
\[
\Vert T_{ij}f\Vert_{L^{\infty}(S_{T})}\leq c\,\mathcal{U}_{f,S_{T}}^{\mu
}(\sqrt{T-\tau}).
\]
We then turn to prove \eqref{Dini bound}. To begin with, we arbitrarily fix
$r>0$ and we let $\xi_{1}=(x_{1},t),\,\xi_{2}=(x_{2},t)\in\Omega$ be such that
$d(\xi_{1},\xi_{2})=\Vert x_{1}-x_{2}\Vert\leq r$. Using the compact notation
$\eta=(y,s)$, we write
\begin{equation}%
\begin{split}
&  T_{ij}f(x_{1},t)-T_{ij}f(x_{2},t)\\
&  \qquad=\int_{\mathbb{R}^{N}\times(\tau,t)}\Big\{\partial_{x_{i}x_{j}}%
^{2}\Gamma(x_{1},t;y,s)\big[f(E(s-t)x_{1},s)-f(y,s)\big]\\
&  \qquad\qquad\quad-\partial_{x_{i}x_{j}}^{2}\Gamma(x_{2}%
,t;y,s)\big[f(E(s-t)x_{2},s)-f(y,s)\big]\Big\}\,dy\,ds\\
&  \qquad=\int_{\{\eta:\,d(\xi_{2},\eta)\geq4\bd{\kappa}d(\xi_{2},\xi_{1}%
)\}}\{\cdots\}\,dy\,ds\\
&  \qquad\qquad\quad+\int_{\{\eta:\,d(\xi_{2},\eta)<4\bd{\kappa}d(\xi_{2}%
,\xi_{1})\}}\{\cdots\}\,dy\,ds\\[0.1cm]
&  \qquad=:\mathrm{A}_{1}+\mathrm{A}_{2},
\end{split}
\label{eq:spliTijSTART}%
\end{equation}
where $\bd\kappa>0$ is as in \eqref{eq:quasitriangled}-\eqref{eq:quasisymd}.
We then turn to estimate $\mathrm{A}_{1}$ and $\mathrm{A}_{2}$.

\noindent-\thinspace\thinspace\textsc{Estimate of $\mathrm{A}_{1}$.} To begin
with, we write $\mathrm{A}_{1}$ as follows:
\begin{align*}
\mathrm{A}_{1}  &  =\int_{\{\eta:\,d(\xi_{2},\eta)\geq4\bd{\kappa}d(\xi
_{2},\xi_{1})\}}\Big\{\big[f(E(s-t)x_{1},s)-f(y,s)\big]\times\\[0.1cm]
&  \qquad\qquad\times\big[\partial_{x_{i}x_{j}}^{2}\Gamma(x_{1}%
,t;y,s)-\partial_{x_{i}x_{j}}^{2}\Gamma(x_{2}%
,t;y,s)\big]\Big\}\,dy\,ds\\[0.15cm]
&  \qquad+\int_{\{\eta:\,d(\xi_{2},\eta)\geq4\bd{\kappa}d(\xi_{2},\xi_{1}%
)\}}\Big\{\partial_{x_{i}x_{j}}^{2}\Gamma(x_{2},t;y,s)\times\\
&  \qquad\qquad\times\big[f(E(s-t)x_{1},s)-f(E(s-t)x_{2}%
,s)\big]\Big\}\,dy\,ds\\[0.1cm]
&  =:\mathrm{A}_{11}+\mathrm{A}_{12}.
\end{align*}

\noindent-\thinspace\thinspace\emph{Estimate of $\mathrm{A}_{11}$}. By using
the mean value i\-ne\-qua\-li\-ties \eqref{eq:meanvalueGamma} in Theorem
\ref{thm:finepropGamma}, together with Lemma \ref{lem:equivalentd},
\eqref{eq:quasisymd} and the expression of $d$ in \eqref{eq:explicitd}, we
have
\begin{align*}
|\mathrm{A}_{11}|  &  \leq c\int_{\{\eta:\,d(\xi_{2},\eta)\geq
4\bd{\kappa}d(\xi_{2},\xi_{1})\}}\frac{d(\xi_{2},\xi_{1})}{d(\xi_{2}%
,\eta)^{Q+3}}\cdot\omega_{f,S_{T}}(\Vert E(s-t)x_{1}-y\Vert)\,dyds\\
&  \leq c\,d(\xi_{2},\xi_{1})\int_{\{\eta:\,d(\xi_{2},\eta)\geq
4\bd{\kappa}d(\xi_{2},\xi_{1})\}}\frac{\omega_{f,S_{T}}(d(\eta,\xi_{1}%
))}{d(\xi_{2},\eta)^{Q+3}}\,dyds\\
&  \leq c\,d(\xi_{2},\xi_{1})\int_{\{\eta:\,d(\xi_{2},\eta)\geq
4\bd{\kappa}d(\xi_{2},\xi_{1})\}}\frac{\omega_{f,S_{T}}(c\,d(\xi_{2},\eta
))}{d(\xi_{2},\eta)^{Q+3}}\,dyds
\end{align*}
From this, by applying Lemma \ref{lem:doublingomega} (and since $f\in
\mathcal{D}(S_{T})$), we obtain
\begin{equation}
|\mathrm{A}_{11}|\leq c\,d(\xi_{2},\xi_{1})\int_{8\bd\kappa d(\xi_{2},\xi
_{1})}^{\infty}\frac{\omega_{f,S_{T}}(cs)}{s^{2}}\,ds\leq cr\int_{cr}^{\infty
}\frac{\omega_{f,S_{T}}(s)}{s^{2}}\,ds, \label{eq:estimA11finalSpace}%
\end{equation}
where $c>0$ is a structural constant. We explicitly mention that, in the above
estimate, we have also used the monotonicity of the map
\[
r\mapsto\mathcal{M}_{2,f,S_{T}}(r)=r\int_{r}^{\infty}\frac{\omega_{f,S_{T}%
}(s)}{s^{2}}\,ds
\]
proved in Lemma \ref{lem:generalMN}, jointly with the fact that
\[
d(\xi_{2},\xi_{1})=\Vert x_{1}-x_{2}\Vert\leq r.
\]

\noindent-\thinspace\thinspace\emph{Estimate of $\mathrm{A}_{12}$}. Arguing as
in \cite{BB}, by using Lemma \ref{lem:stimaEx} we have
\begin{equation}%
\begin{split}
|\mathrm{A}_{12}|  &  \leq\int_{\tau}^{t}\big|f(E(s-t)x_{1},s)-f(E(s-t)x_{2}%
,s)\big|\cdot\mathcal{J}(s)\,ds\\
&  \leq\int_{\tau}^{t}\omega_{f,S_{T}}(\Vert E(s-t)(x_{1}-x_{2})\Vert
)\cdot\mathcal{J}(s)\,ds\\
&  \leq\int_{\tau}^{t}\omega_{f,S_{T}}\big(c(\Vert x_{1}-x_{2}\Vert+\sqrt
{t-s})\big)\cdot\mathcal{J}(s)\,ds,
\end{split}
\label{eq:tostartperdistinguere}%
\end{equation}
where $c>0$ is a structural constant and
\[
\mathcal{J}(s):=\bigg|\int_{\{y\in\mathbb{R}^{N}:\,d(\xi_{2},(y,s))\geq
4\bd{\kappa}d(\xi_{2},\xi_{1})\}}\partial_{x_{i}x_{j}}^{2}\Gamma
(x_{2},t;y,s)\,dy\bigg|.
\]
From this, since $\omega_{f,S_{T}}$ shares with the map $r\mapsto r^{\alpha}$
the same monotonicity, we can proceed \emph{exactly} as in the proof of
\cite[Thm.\,3.17]{BB}, obtaining
\begin{equation}
|A_{12}|\leq c\,\omega_{f,S_{T}}(c\Vert x_{1}-x_{2}\Vert)\leq c\,\omega
_{f,S_{T}}(cr). \label{eq:estimA12Space}%
\end{equation}
for a suitable structural constant $c>0$.

Summing up, by combining \eqref{eq:estimA11finalSpace} with
\eqref{eq:estimA12Space}, we conclude that
\begin{equation}
|\mathrm{A}_{1}|\leq c\,\Big\{\omega_{f,S_{T}}(cr)+cr\int_{cr}^{\infty}%
\frac{\omega_{f,S_{T}}(s)}{s^{2}}\,ds\Big\}, \label{eq:estimA1FINALSpace}%
\end{equation}
for a suitable structural constant $c>0$. \medskip

\noindent-\,\,\textsc{Estimate of $A_{2}$.} First of all, we write
\begin{equation}
|\mathrm{A}_{2}|\leq\mathrm{A}_{21}+\mathrm{A}_{22},
\label{eq:estimA2splittedSpace}%
\end{equation}
where, for $k=1,2$, we have introduced the notation
\begin{align*}
\mathrm{A}_{2k}  &  :=\int_{\{\eta:\,d(\xi_{2},\eta)<4\bd{\kappa}d(\xi_{2}%
,\xi_{1})\}}|\partial_{x_{i}x_{j}}^{2}\Gamma(x_{k},t;y,s)|\times\\
&  \qquad\quad\times\omega_{f,S_{T}}(\Vert E(s-t)x_{k}-y\Vert)\,dy\,ds.
\end{align*}
We then proceed by estimating the two integrals $\mathrm{A}_{21}%
,\,\mathrm{A}_{22}$ separately. \vspace*{0.1cm}

\noindent-\thinspace\thinspace\emph{Estimate of $\mathrm{A}_{21}$}. By
exploiting the estimates for $\partial_{x_{i}x_{j}}^{2}\Gamma$ in Theo\-rem
\ref{thm:finepropGamma}-(1), together with
\eqref{eq:quasitriangled}-\eqref{eq:quasisymd} and the expression of $d$, we
have
\begin{align*}
\mathrm{A}_{21}  &  \leq c\,\int_{\{\eta:\,d(\xi_{2},\eta)<4\bd{\kappa}d(\xi
_{2},\xi_{1})\}}\frac{\omega_{f,S_{T}}(\Vert E(s-t)x_{1}-y\Vert)}{d(\xi
_{1},\eta)^{Q+2}}\,dy\,ds\\
&  \leq c\,\int_{\{\eta:\,d(\xi_{2},\eta)<4\bd{\kappa}d(\xi_{2},\xi_{1}%
)\}}\frac{\omega_{f,S_{T}}(d(\eta,\xi_{1}))}{d(\xi_{1},\eta)^{Q+2}}\,dy\,ds\\
&  \leq c\,\int_{\{\eta:\,d(\xi_{2},\eta)<4\bd{\kappa}d(\xi_{2},\xi_{1}%
)\}}\frac{\omega_{f,S_{T}}(\bd\kappa d(\xi_{1},\eta))}{d(\xi_{1},\eta)^{Q+2}%
}\,dy\,ds\\[0.05cm]
&  \big(\text{since $d(\xi_{2},\eta)<4\bd{\kappa}d(\xi_{2},\xi_{1}%
)\,\Rightarrow\,d(\xi_{1},\eta)<c\,d(\xi_{2},\xi_{1})$}\big)\\
&  \leq c\,\int_{\{\eta:\,d(\xi_{1},\eta)<cd(\xi_{2},\xi_{1})\}}\frac
{\omega_{f,S_{T}}(\bd\kappa d(\xi_{1},\eta))}{d(\xi_{1},\eta)^{Q+2}}\,dy\,ds,
\end{align*}
where $c>0$ is a suitable structural constant. From this, using once again
Lemma \ref{lem:doublingomega}, we obtain
\begin{equation}
\mathrm{A}_{21}\leq c\,\int_{0}^{2cd(\xi_{2},\xi_{1})}\frac{\omega_{f,S_{T}%
}(\bd\kappa s)}{s}\,ds\leq c\,\int_{0}^{cr}\frac{\omega_{f,S_{T}}(s)}{s}\,ds.
\label{eq:estimA21finalSpace}%
\end{equation}

\noindent-\thinspace\thinspace\emph{Estimate of $\mathrm{A}_{22}$}. By
proceeding exactly as in the estimate of $\mathrm{A}_{21}$ (but without the
need of enlarging the domain of integration), we obtain
\begin{equation}%
\begin{split}
\mathrm{A}_{22}  &  \leq c\,\int_{\{\eta:\,d(\xi_{2},\eta)<4\bd{\kappa}d(\xi
_{2},\xi_{1})\}}\frac{\omega_{f,S_{T}}(\Vert E(s-t)x_{2}-y\Vert)}{d(\xi
_{2},\eta)^{Q+2}}\,dy\,ds\\
&  \leq c\,\int_{\{\eta:\,d(\xi_{2},\eta)<4\bd{\kappa}d(\xi_{2},\xi_{1}%
)\}}\frac{\omega_{f,S_{T}}(\bd\kappa d(\xi_{2},\eta))}{d(\xi_{2},\eta)^{Q+2}%
}\,dy\,ds\\
&  \leq c\,\int_{0}^{8\bd\kappa d(\xi_{2},\xi_{1})}\frac{\omega_{f,S_{T}%
}(\bd\kappa s)}{s}\,ds\leq c\,\int_{0}^{cr}\frac{\omega_{f,S_{T}}(s)}{s}\,ds.
\end{split}
\label{eq:estimA22finalSpace}%
\end{equation}
Summing up, by combining
\eqref{eq:estimA21finalSpace}-\eqref{eq:estimA22finalSpace} with
\eqref{eq:estimA2splittedSpace}, we conclude that
\begin{equation}
|\mathrm{A}_{2}|\leq c\,\int_{0}^{cr}\frac{\omega_{f,S_{T}}(s)}{s}\,ds,
\label{eq:estimA2FINALSpace}%
\end{equation}
where $c>0$ is a suitable structural constant.

Now we have estimated $\mathrm{A}_{1}$ and $\mathrm{A}_{2}$, we are ready to
complete the proof: in fact, gathering
\eqref{eq:estimA1FINALSpace}-\eqref{eq:estimA2FINALSpace}, and recalling
\eqref{eq:spliTijSTART}, we conclude that
\begin{align*}
&  |T_{ij}f(x_{1},t)-T_{ij}f(x_{2},t)|\leq|\mathrm{A}_{1}|+|\mathrm{A}_{2}|\\
&  \qquad\leq c\,\Big\{\omega_{f,S_{T}}(cr)+\int_{0}^{cr}\frac{\omega
_{f,S_{T}}(s)}{s}\,ds+cr\int_{cr}^{\infty}\frac{\omega_{f,S_{T}}(s)}{s^{2}%
}\,ds\Big\}\\
&  \qquad=c\,\mathcal{M}_{f,S_{T}}(cr)\quad\text{$\forall\,\,(x_{1}%
,t),\,(x_{2},t)\in\Omega$ with $\Vert x_{1}-x_{2}\Vert\leq r$},
\end{align*}
from which we readily obtain the desired \eqref{Dini bound}.
\end{proof}

\medskip

Thanks to Theorem \ref{Thm Dini sing int}, we can finally prove the following theorem.

\begin{theorem}
[Moduli of continuity of derivatives]\label{Thm Dini coeff t} Let Let
$T>\tau>-\infty$. Then, there exists structural constants $c,\mu>0$, such
that
\begin{equation}
\begin{gathered} \sum_{i,j=1}^{m_0}\Vert\de_{x_{i}x_{j}}^{2}u\Vert_{L^\infty(S_T)}\leq c\,\mathcal{U}^\mu_{\LL u,S_T}(\sqrt{T-\tau}),\\[0.1cm] \omega_{\de^2_{x_ix_j}u,S_T}(r)\leq c\,\mathcal{M}_{\LL u,S_T}(cr)\quad\forall\,\,r > 0. \end{gathered} \label{eq:Schauderspaceaijt}%
\end{equation}
for every $u\in\mathcal{S}^{0}(\tau,T)$ with $\mathcal{L}u\in\mathcal{D}%
(S_{T})$. Moreover, we have
\begin{equation}
\begin{gathered} \|Yu\|_{L^\infty(S_T)} \leq c\,\big(\|\LL u\|_{L^\infty(S_T)}+\mathcal{U}^\mu_{\LL u,S_T}(\sqrt{T-\tau})\big) \\ \omega_{Y u,S_T}(r)\leq c\,\mathcal{M}_{\LL u,S_T}(cr)\quad\forall\,\,r > 0. \end{gathered} \label{eq:SchauderspaceaijtDRIFT}%
\end{equation}
In particular, if $\mathcal{L}u\in\mathcal{D}_{\log}(S_{T})$ we have
$\partial_{x_{i}x_{j}}^{2}u\,,Yu\in\mathcal{D}(S_{T})$.
\end{theorem}

\begin{proof}
Let $u\in\mathcal{S}^{0}(\tau;T)$ be such that $\mathcal{L}u\in\mathcal{D}%
(S_{T})$. By applying the representation formula \eqref{repr formula u_xx}, we
can write
\begin{align*}
\partial_{x_{i}x_{j}}^{2}u(x,t)  &  =\int_{\mathbb{R}^{N}\times(\tau
,t)}\partial_{x_{i}x_{j}}^{2}\Gamma(x,t;y,s)\cdot\big[\mathcal{L}%
u(E(s-t)x,s)-\mathcal{L}u(y,s)\big]\,dy\,ds\\[0.1cm]
&  =T_{ij}(\mathcal{L}u)(x,t)\qquad\text{for every $(x,t)\in S_{T}$ and $1\leq
i,j\leq m_{0}$},
\end{align*}
where $T_{ij}$ is as in Theorem \ref{Thm Dini sing int}. Then, from
\eqref{sup bound}-\eqref{Dini bound} we obtain \eqref{eq:Schauderspaceaijt}.
On the other hand, using the definition of $\mathcal{L}$, and recalling that
the coefficients $a_{ij}(\cdot)$ are bounded and \emph{independent of $x$},
from \eqref{eq:Schauderspaceaijt} we also get
\begin{align*}
\Vert Yu\Vert_{L^{\infty}(S_{T})}  &  =\Big\|\mathcal{L}u-\sum_{i,j=1}^{m_{0}%
}a_{ij}\,\partial_{x_{i}x_{j}}^{2}u\Big\|_{L^{\infty}(S_{T})}\\
&  \leq c\big(\Vert\mathcal{L}u\Vert_{L^{\infty}(S_{T})}+\mathcal{U}%
_{\mathcal{L}u,S_{T}}^{\mu}(\sqrt{T-\tau})\big);
\end{align*}
analogously, if $A$ is as in \eqref{a cost Dini}, we obtain
\[
\omega_{Yu,S_{T}}(r)\leq\omega_{\mathcal{L}u}(r)+A\sum_{i,j=1}^{m_{0}}%
\omega_{\partial_{x_{i}x_{j}}^{2}u}(r)\leq c\,\mathcal{M}_{\mathcal{L}u}(cr).
\]
This is precisely \eqref{eq:SchauderspaceaijtDRIFT}. Finally, the Dini
continuity of the functions $\partial_{x_{i}x_{j}}^{2}u,\,Yu$, under the
additional assumption $\mathcal{L}u\in\mathcal{D}_{\mathrm{log}}(S_{T})$,
immediately follows from Proposition \ref{prop:funMN}.
\end{proof}

\medskip

We end this section with a weaker version of Theorem \ref{Thm main time} for
operators with coefficients only depending on $t$; we will use this result {to
prove} Theorem \ref{Thm main time}.

\begin{theorem}
[Continuity estimates in space-time]\label{Thm local continuity time}

Let $\mathcal{L}$ be as in \eqref{eq:LLsolot}, and let $T>\tau>-\infty$.
Moreover, let $K\subseteq\mathbb{R}^{N}$ be a \emph{compact set}.

Then, there exist a structural constant $\mu>0$ and a constant ${c}%
(K,\tau,T)>0$ such that, for every $u\in\mathcal{S}^{0}(\tau;T)$ such that
$\mathcal{L}u\in\mathcal{D}(S_{T})$, one has
\begin{equation}%
\begin{split}
&  |\partial_{x_{i}x_{j}}^{2}u(x_{1},t_{1})-\partial_{x_{i}x_{j}}^{2}%
u(x_{2},t_{2})|\\
&  \qquad\leq c\big\{\mathcal{M}_{\mathcal{L}u,S_{T}}\big(c(d((x_{1}%
,t_{1}),(x_{2},t_{2}))+|t_{1}-t_{2}|^{1/q_{N}})\big)\\
&  \qquad\qquad+\mathcal{U}_{\mathcal{L}u,S_{T}}^{\mu}(\sqrt{|t_{2}-t_{1}%
|})\big\}
\end{split}
\label{eq:schauderspacetimeuxixj}%
\end{equation}
for every $1\leq i,j\leq m_{0}$ and $(x_{1},t_{1}),(x_{2},t_{2})\in
K\times\lbrack\tau,T]$.

Here, $\mathcal{U}_{\mathcal{L}u,S_{T}}^{\mu}$ is as in \eqref{eq:defUmuIntro}
and $\mathcal{M}_{\mathcal{L}u,S_{T}}$ is as in \eqref{eq:defcontinM};
moreover, $q_{N}\geq3$ is the largest e\-xpo\-nent in the dilations
$D_{0}(\lambda)$, see \eqref{dilations}.

In particular, from \eqref{eq:schauderspacetimeuxixj} we deduce that the
derivatives $\partial_{x_{i}x_{j}}^{2}u$ are locally uniformly continuous in
the joint variables $(x,t)$.
\end{theorem}

\begin{proof}
Let $u\in\mathcal{S}^{0}(\tau;T)$ be such that $\mathcal{L}u\in\mathcal{D}%
(S_{T})$. To prove \eqref{eq:schauderspacetimeuxixj} we first observe that,
owing to Theorem \ref{Thm Dini coeff t}, for every $(x_{1},t),(x_{2},t)\in
S_{T}$ we have
\begin{equation}%
\begin{split}
&  |\partial_{x_{i}x_{j}}^{2}u(x_{1},t)-\partial_{x_{i}x_{j}}^{2}%
u(x_{2},t)|\leq\omega_{\partial_{x_{i}x_{j}}^{2}u}(\Vert x_{1}-x_{2}\Vert)\\
&  \qquad\qquad\leq c\,\mathcal{M}_{\mathcal{L}u,S_{T}}(c\Vert x_{1}%
-x_{2}\Vert)
\end{split}
\label{eq:estimuxixjtfisso}%
\end{equation}
where $c>0$ is a structural constant. As a consequence of
\eqref{eq:estimuxixjtfisso}, and ta\-king into account Lemma \ref{lem:stimaEx}%
, to prove \eqref{eq:schauderspacetimeuxixj} it suffices to show that
\begin{equation}%
\begin{split}
&  |\partial_{x_{i}x_{j}}^{2}u(x,t_{1})-\partial_{x_{i}x_{j}}^{2}u(x,t_{2})|\\
&  \qquad\qquad\leq c\,\big\{\mathcal{M}_{\mathcal{L}u,S_{T}}(c|t_{1}%
-t_{2}|^{1/q_{N}})+\mathcal{U}_{\mathcal{L}u,S_{T}}^{\mu}(\sqrt{|t_{1}-t_{2}%
|})\big\},
\end{split}
\label{eq:toproveuxixjxfissoGENERALE}%
\end{equation}
for every $(x,t_{1}),(x,t_{2})\in K\times\lbrack\tau,T]$, where $c>0$ is a
constant independent of $u$ (but possibly depending on $K,\tau,T$), while
$\mu>0$ is a structural constant. In fact, once
\eqref{eq:toproveuxixjxfissoGENERALE} has been established, by
\eqref{eq:estimuxixjtfisso}-\eqref{eq:toproveuxixjxfissoGENERALE} we get
\begin{align*}
&  |\partial_{x_{i}x_{j}}^{2}u(x_{1},t_{1})-\partial_{x_{i}x_{j}}^{2}%
u(x_{2},t_{2})|\\
&  \qquad\leq|\partial_{x_{i}x_{j}}^{2}u(x_{1},t_{1})-\partial_{x_{i}x_{j}%
}^{2}u(x_{2},t_{1})|+|\partial_{x_{i}x_{j}}^{2}u(x_{2},t_{1})-\partial
_{x_{i}x_{j}}^{2}u(x_{2},t_{2})|\\
&  \qquad\leq c\,\big\{\mathcal{M}_{\mathcal{L}u,S_{T}}(c\Vert x_{1}%
-x_{2}\Vert)+\mathcal{M}_{\mathcal{L}u,S_{T}}(c|t_{1}-t_{2}|^{1/q_{N}})\\
&  \qquad\qquad\qquad+\mathcal{U}_{\mathcal{L}u,S_{T}}^{\mu}(\sqrt
{|t_{1}-t_{2}|})\big\}\\
&  \qquad(\text{by the explicit expression of $d$, see \eqref{eq:explicitd}}%
)\\
&  \qquad=c\,\big\{\mathcal{M}_{\mathcal{L}u,S_{T}}\big(cd((x_{1}%
,t_{1}),(x_{2},t_{1}))\big)+\mathcal{M}_{\mathcal{L}u,S_{T}}(c|t_{1}%
-t_{2}|^{1/q_{N}})\\
&  \qquad\qquad\qquad+\mathcal{U}_{\mathcal{L}u,S_{T}}^{\mu}(\sqrt
{|t_{1}-t_{2}|})\big\}=:(\bigstar);
\end{align*}
from this, using the quasi-triangle inequality \eqref{eq:quasitriangled}
jointly with Lemma \ref{lem:stimaEx}, and recalling that $\mathcal{M}%
_{\mathcal{L}u,S_{T}}$ is non-decreasing, see Proposition \ref{prop:funMN}, we
obtain
\begin{align*}
(\bigstar)  &  \leq c\,\big\{\mathcal{M}_{\mathcal{L}u,S_{T}}\big(c(d((x_{1}%
,t_{1}),(x_{2},t_{2}))+d((x_{2},t_{1}),(x_{2},t_{2})))\big)\\
&  \qquad\qquad+\mathcal{M}_{\mathcal{L}u,S_{T}}(c|t_{1}-t_{2}|^{1/q_{N}%
})+\mathcal{U}_{\mathcal{L}u,S_{T}}^{\mu}(\sqrt{|t_{1}-t_{2}|})\big\}\\
&  =c\,\big\{\mathcal{M}_{\mathcal{L}u,S_{T}}\big(c(d((x_{1},t_{1}%
),(x_{2},t_{2}))+\Vert x_{2}-E(t_{1}-t_{2})x_{2}\Vert+\sqrt{|t_{1}-t_{2}%
|})\big)\\
&  \qquad\qquad+\mathcal{M}_{\mathcal{L}u,S_{T}}(c|t_{1}-t_{2}|^{1/q_{N}%
})+\mathcal{U}_{\mathcal{L}u,S_{T}}^{\mu}(\sqrt{|t_{1}-t_{2}|})\big\}\\
&  (\text{since $|t_{1}-t_{2}|\leq T-\tau$ and $q_{N}\geq3$})\\
&  \leq c\,\big\{\mathcal{M}_{\mathcal{L}u,S_{T}}\big(c(d((x_{1},t_{1}%
),(x_{2},t_{2}))+|t_{1}-t_{2}|^{1/q_{N}})\big)\\
&  \qquad\qquad+\mathcal{M}_{\mathcal{L}u,S_{T}}(c|t_{1}-t_{2}|^{1/q_{N}%
})+\mathcal{U}_{\mathcal{L}u,S_{T}}^{\mu}(\sqrt{|t_{1}-t_{2}|})\big\}\\
&  \leq c\,\big\{\mathcal{M}_{\mathcal{L}u,S_{T}}\big(c(d((x_{1},t_{1}%
),(x_{2},t_{2}))+|t_{1}-t_{2}|^{1/q_{N}})\big)\\
&  \qquad\qquad+\mathcal{U}_{\mathcal{L}u,S_{T}}^{\mu}(\sqrt{|t_{1}-t_{2}%
|})\big\},
\end{align*}
which is exactly \eqref{eq:schauderspacetimeuxixj}. Hence, we turn to prove
\eqref{eq:toproveuxixjxfissoGENERALE}. This can be done adapting several
computations already exploited in the proof of Theorem \ref{Thm Dini coeff t}.
We will point out just the relevant differences.

To begin with, we fix $\xi_{1}=(x,t_{1}),\,\xi_{2}=(x,t_{2})\in K\times
\lbrack\tau,T]$ and we exploit the re\-pre\-sentation formula
\eqref{repr formula u_xx}: assuming, to fix ideas, that $t_{2}\geq t_{1}$ (and
using the compact notation $\eta=(y,s)$), we can write
\begin{equation}%
\begin{split}
&  \partial_{x_{i}x_{j}}^{2}u(x,t_{1})-\partial_{x_{i}x_{j}}^{2}u(x,t_{2})\\
&  \quad=\int_{\mathbb{R}^{N}\times(\tau,t_{1})}\Big\{\partial_{x_{i}x_{j}%
}^{2}\Gamma(x,t_{1};y,s)\big[\mathcal{L}u(E(s-t_{1})x,s)-\mathcal{L}%
u(y,s)\big]\\
&  \qquad\qquad\quad-\partial_{x_{i}x_{j}}^{2}\Gamma(x,t_{2}%
;y,s)\big[\mathcal{L}u(E(s-t_{2})x,s)-\mathcal{L}u(y,s)\big]\Big\}\,dy\,ds\\
&  \qquad-\int_{\mathbb{R}^{N}\times(t_{1},t_{2})}\partial_{x_{i}x_{j}}%
^{2}\Gamma(x,t_{2};y,s)\big[\mathcal{L}u(E(s-t_{2})x,s)-\mathcal{L}%
u(y,s)\big]\,dy\,ds\\
&  \quad=\int_{\{\eta:\,d(\xi_{2},\eta)\geq4\bd{\kappa}d(\xi_{2},\xi_{1}%
)\}}\{\cdots\}\,dy\,ds\\
&  \qquad\qquad+\int_{\{\eta:\,d(\xi_{2},\eta)<4\bd{\kappa}d(\xi_{2},\xi
_{1})\}}\{\cdots\}\,dy\,ds\\
&  \qquad\qquad-\int_{\mathbb{R}^{N}\times(t_{1},t_{2})}\{\cdots
\}\,dy\,ds\\[0.1cm]
&  \quad=:\mathrm{A}_{1}+\mathrm{A}_{2}-\mathrm{A}_{3},
\end{split}
\label{eq:spliuxixjSTART}%
\end{equation}
where $\bd\kappa>0$ is as in \eqref{eq:quasitriangled}-\eqref{eq:quasisymd} We
now turn to estimate $\mathrm{A}_{1},\mathrm{A_{2}}$ and $\mathrm{A}_{3}$.
\vspace*{0.1cm}

\noindent-\thinspace\thinspace\textsc{Estimate of $\mathrm{A}_{1}$.} To begin
with, we write $\mathrm{A}_{1}$ as follows:
\begin{align*}
\mathrm{A}_{1}  &  =\int_{\{\eta:\,d(\xi_{2},\eta)\geq4\bd{\kappa}d(\xi
_{2},\xi_{1})\}}\Big\{\big[\mathcal{L}u(E(s-t_{1})x,s)-\mathcal{L}%
u(y,s)\big]\times\\[0.1cm]
&  \qquad\qquad\times\big[\partial_{x_{i}x_{j}}^{2}\Gamma(x,t_{1}%
;y,s)-\partial_{x_{i}x_{j}}^{2}\Gamma(x,t_{2}%
;y,s)\big]\Big\}\,dy\,ds\\[0.15cm]
&  \qquad+\int_{\{\eta:\,d(\xi_{2},\eta)\geq4\bd{\kappa}d(\xi_{2},\xi_{1}%
)\}}\Big\{\partial_{x_{i}x_{j}}^{2}\Gamma(x,t_{2};y,s)\times\\
&  \qquad\qquad\times\big[\mathcal{L}u(E(s-t_{1})x,s)-\mathcal{L}%
u(E(s-t_{2})x,s)\big]\Big\}\,dy\,ds\\[0.1cm]
&  =:\mathrm{A}_{11}+\mathrm{A}_{12}.
\end{align*}
We then turn to estimate $\mathrm{A}_{11}$ and $\mathrm{A}_{12}$ separately.
\vspace*{0.1cm}

\noindent-\thinspace\thinspace\emph{Estimate of $\mathrm{A}_{11}$}. First of
all, by proceeding \emph{exactly} as in the estimate of $\mathrm{A}_{11}$ in
the proof of Theorem \ref{Thm Dini coeff t}, we get the following estimate
\begin{align*}
|\mathrm{A}_{11}|  &  \leq c\int_{\{\eta:\,d(\xi_{2},\eta)\geq
4\bd{\kappa}d(\xi_{2},\xi_{1})\}}\frac{d(\xi_{2},\xi_{1})}{d(\xi_{2}%
,\eta)^{Q+3}}\cdot\omega_{\mathcal{L}u,S_{T}}(\Vert E(s-t_{1})x-y\Vert
)\,dyds\\
&  \leq c\,d(\xi_{2},\xi_{1})\int_{\{\eta:\,d(\xi_{2},\eta)\geq
4\bd{\kappa}d(\xi_{2},\xi_{1})\}}\frac{\omega_{\mathcal{L}u,S_{T}}(cd(\xi
_{2},\eta))}{d(\xi_{2},\eta)^{Q+3}}\,dyds\\
&  \leq cd(\xi_{2},\xi_{1})\int_{cd(\xi_{2},\xi_{1})}^{\infty}\frac
{\omega_{\mathcal{L}u,S_{T}}(s)}{s^{2}}\,ds,
\end{align*}
where $c>0$ is a structural constant. On the other hand, by exploiting Lemma
\ref{lem:stimaEx} (and since $t_{1},t_{2}\in\lbrack\tau,T]$), we have
\begin{equation}
d(\xi_{2},\xi_{1})=\Vert x-E(t_{2}-t_{1})x\Vert+\sqrt{|t_{2}-t_{1}|}\leq
c\,|t_{2}-t_{1}|^{1/q_{N}}, \label{eq:usoLemma23oa}%
\end{equation}
where $c>0$ is a constant depending on $K,\tau,T$. Hence, we obtain
\begin{equation}
|\mathrm{A}_{11}|\leq c|t_{1}-t_{2}|^{1/q_{N}}\int_{c|t_{1}-t_{2}|^{1/q_{N}}%
}^{\infty}\frac{\omega_{f,S_{T}}(s)}{s^{2}}\,ds. \label{A11y}%
\end{equation}

\noindent-\thinspace\thinspace\emph{Estimate of $\mathrm{A}_{12}$}. First of
all, using once again Lemma \ref{lem:stimaEx} we get
\[%
\begin{split}
|\mathrm{A}_{12}|  &  \leq\int_{\tau}^{t_{1}}\big|\mathcal{L}u(E(s-t_{1}%
)x,s)-\mathcal{L}u(E(s-t_{2})x,s)\big|\cdot\mathcal{J}(s)\,ds\\
&  \leq\int_{\tau}^{t_{1}}\omega_{\mathcal{L}u,S_{T}}(\Vert(E(s-t_{1}%
)-E(s-t_{2}))x\Vert)\cdot\mathcal{J}(s)\,ds\\
&  (\text{since $|s-t_{1}|,|s-t_{2}|\leq T-\tau$ for all $\tau\leq s\leq
t_{1}$})\\
&  \leq\omega_{\mathcal{L}u,S_{T}}(c|t_{1}-t_{2}|^{1/q_{N}})\int_{\tau}%
^{t_{1}}\mathcal{J}(s)\,ds=:(\bigstar)
\end{split}
\]
where $c>0$ is a constant depending on $K,\tau,T$ and
\[
\mathcal{J}(s):=\bigg|\int_{\{y\in\mathbb{R}^{N}:\,d((x,t_{2}),(y,s))\geq
4\bd{\kappa}d(\xi_{2},\xi_{1})\}}\partial_{x_{i}x_{j}}^{2}\Gamma
(x,t_{2};y,s)\,dy\bigg|.
\]
From this, using the cancellation property of $\mathcal{J}$ in
\cite[Thm.\,3.16]{BB}, we obtain
\begin{equation}
(\bigstar)\leq c\,\omega_{\mathcal{L}u,S_{T}}(c|t_{1}-t_{2}|^{1/q_{N}}),
\label{eq:estimA12Scht}%
\end{equation}
for a suitable constant $c>0$ depending on $K,\tau,T$. \vspace*{0.05cm} By
combining \eqref{A11y} with \eqref{eq:estimA12Scht}, we conclude that
\begin{equation}%
\begin{split}
\left\vert \mathrm{A}_{1}\right\vert  &  \leq c\Big\{\omega_{\mathcal{L}%
u,S_{T}}(|t_{1}-t_{2}|^{1/q_{N}})\\
&  \qquad\quad+c|t_{1}-t_{2}|^{1/q_{N}}\int_{c|t_{1}-t_{2}|^{1/q_{N}}}%
^{\infty}\frac{\omega_{f,S_{T}}(s)}{s^{2}}\,ds\Big\},
\end{split}
\label{eq:estimA1FINALScht}%
\end{equation}
for a suitable constant $c>0$ possibly depending on $K,T,\tau$. \medskip

\noindent-\thinspace\thinspace\textsc{Estimate of $A_{2}$.} By proceeding
exactly as in the estimate of $\mathrm{A}_{2}$ in the proof of Theorem
\ref{Thm Dini coeff t}, and by taking into account \eqref{eq:usoLemma23oa}, we
obtain the estimate
\begin{equation}%
\begin{split}
\left\vert \mathrm{A}_{2}\right\vert  &  \leq c\,\Big\{\int_{\{\eta
:\,d(\xi_{2},\eta)<4\bd{\kappa}d(\xi_{2},\xi_{1})\}}\frac{\omega
_{\mathcal{L}u,S_{T}}\left(  \Vert E(s-t_{1})x-y\Vert\right)  }{d(\xi_{1}%
,\eta)^{Q+2}}\,dy\,ds\\
&  \qquad\quad+\int_{\{\eta:\,d(\xi_{2},\eta)<4\bd{\kappa}d(\xi_{2},\xi
_{1})\}}\frac{\omega_{\mathcal{L}u,S_{T}}(\Vert E(s-t_{2})x-y\Vert)}{d(\xi
_{2},\eta)^{Q+2}}\,dy\,ds\Big\}\\
&  \leq c\,\Big\{\int_{\{\eta:\,d(\xi_{1},\eta)<cd(\xi_{2},\xi_{1})\}}%
\frac{\omega_{\mathcal{L}u,S_{T}}(\bd\kappa d(\xi_{1},\eta))}{d(\xi_{1}%
,\eta)^{Q+2}}\,dy\,ds\\
&  \qquad\quad+\int_{\{\eta:\,d(\xi_{2},\eta)<4\bd{\kappa}d(\xi_{2},\xi
_{1})\}}\frac{\omega_{\mathcal{L}u,S_{T}}(\bd\kappa d(\xi_{2},\eta))}%
{d(\xi_{2},\eta)^{Q+2}}\,dy\,ds\Big\}\\
&  \leq c\,\int_{0}^{cd(\xi_{2},\xi_{1})}\frac{\omega_{\mathcal{L}u,S_{T}}%
(s)}{s}\,ds\leq c\,\int_{0}^{c|t_{1}-t_{2}|^{1/q_{N}}}\frac{\omega
_{\mathcal{L}u,S_{T}}(s)}{s}\,ds,
\end{split}
\label{eq:estimA2FINALScht}%
\end{equation}
where $c>0$ is a suitable constant depending on $K,T,\tau$. \medskip

\noindent-\thinspace\thinspace\textsc{Estimate of $\mathrm{A}_{3}$.} Using the
assumption $\mathcal{L}u\in\mathcal{D}(S_{T})$, together with estimate
\eqref{eq:integraldexixjconv} in Proposition \ref{prop:Analoga313}, we
immediately obtain
\begin{equation}%
\begin{split}
|\mathrm{A}_{3}|  &  \leq\int_{\mathbb{R}^{N}\times(t_{1},t_{2})}%
|\partial_{x_{i}x_{j}}^{2}\Gamma(x,t_{2};y,s)|\cdot\omega_{\mathcal{L}u,S_{T}%
}(\Vert E(s-t_{2})x-y\Vert)\,dy\,ds\\
&  \leq c\,\mathcal{U}_{\mathcal{L}u,S_{T}}^{\mu}(\sqrt{|t_{1}-t_{2}|}),
\end{split}
\label{eq:estimA3Scht}%
\end{equation}
where $c,\mu>0$ ore structural constants. \vspace*{0.1cm}

Now we have estimated $\mathrm{A}_{1},\mathrm{A}_{2}$ and $\mathrm{A}_{3}$, we
can complete the proof: in fact, gathering
\eqref{eq:estimA1FINALScht},\eqref{eq:estimA2FINALScht} and
\eqref{eq:estimA3Scht}, and recalling \eqref{eq:spliuxixjSTART}, we conclude
that
\begin{align*}
&  |\partial_{x_{i}x_{j}}^{2}u(x,t_{1})-\partial_{x_{i}x_{j}}^{2}%
u(x,t_{2})|\leq|\mathrm{A}_{1}|+|\mathrm{A}_{2}|+|\mathrm{A}_{3}|\\
&  \qquad\leq c\,\Big\{\omega_{\mathcal{L}u,S_{T}}(c|t_{1}-t_{2}|^{1/q_{N}%
})+c|t_{1}-t_{2}|^{1/q_{N}}\int_{c|t_{1}-t_{2}|^{1/q_{N}}}^{\infty}%
\frac{\omega_{\mathcal{L}u,S_{T}}(s)}{s^{2}}\,ds\\
&  \qquad\qquad\quad+\int_{0}^{c|t_{1}-t_{2}|^{1/q_{N}}}\frac{\omega
_{\mathcal{L}u,S_{T}}(s)}{s}\,ds+\mathcal{U}_{\mathcal{L}u,S_{T}}^{\mu}%
(|t_{1}-t_{2}|^{1/q_{N}})\Big\}\\
&  \qquad=c\big\{\mathcal{M}_{\mathcal{L}u,S_{T}}(c|t_{1}-t_{2}|^{1/q_{N}%
})+\mathcal{U}_{\mathcal{L}u,S_{T}}^{\mu}(\sqrt{|t_{1}-t_{2}|})\big\},
\end{align*}
which is exactly the desired \eqref{eq:toproveuxixjxfissoGENERALE}.
\end{proof}

\section{Operators with coefficients depending on $\left(  x,t\right)  $}

\label{sec:proofMainThm}

\subsection{The basic estimate for functions with small support}

We want to extend our result to operators with coefficients $a_{ij}(x,t)$ Dini
continuous in $x$ and bounded measurable in $t$. The first step is a local
estimate for functions with small compact support. \vspace*{0.1cm}

\noindent\textbf{Notation:} Since in this section we will make \emph{crucial
use} of the interpolation inequality contained in Theorem \ref{Thm interpolaz}%
, we will adopt the following notation: given any $T > 0$, any $\overline{\xi
}\in S_{T}$ and any $r > 0$, we set
\[
B_{r}^{T}(\overline{\xi}) = B_{r}(\overline{\xi})\cap S_{T}.
\]

\begin{theorem}
\label{Thm local Dini x} Let $\mathcal{L}$ be as in \eqref{L}, satisfying
assumptions \emph{(H1), (H2), (H3) }stated in Section 1. Then, there exist
constants $c,r_{0}>0$ depending on $T$, the matrix $B$ in \eqref{B}, the
number $\nu$ in \eqref{nu} and $\Vert a\Vert_{\mathcal{D}(S_{T})}$ in
\eqref{a cost Dini}, respectively, such that
\begin{align}
&  \Vert\partial_{x_{i}x_{j}}^{2}u\Vert_{L^{\infty}(B_{r}^{T}(\overline{\xi
}))}\leq c\,\mathcal{U}_{\mathcal{L}u,S_{T}}^{\mu}(1)\label{basic sup}\\
&  \omega_{\partial_{x_{i}x_{j}}^{2}u,B_{r}^{T}(\overline{\xi})}(\rho)\leq
c\big(\mathcal{M}_{\mathcal{L}u,S_{T}}(c\rho)+\mathcal{M}_{a,S_{T}}%
(c\rho)\cdot\mathcal{U}_{\mathcal{L}u,S_{T}}^{\mu}(1)\big)\quad\forall
\,\,\rho>0, \label{basic Dini}%
\end{align}
and these estimates hold for every $\overline{\xi}\in S_{T}$, $0<r\leq r_{0}$,
$1\leq i,j\leq m_{0}$ and $u\in\mathcal{S}^{D}(S_{T})$ with $\mathrm{supp}%
(u)\subseteq B_{r}(\overline{\xi})\cap\overline{S_{T}}$. Here,
\[
\text{$\textstyle\mathcal{M}_{a,S_{T}}=\sum_{i,j=1}^{m_{0}}\mathcal{M}%
_{a_{ij},S_{T}}$ and $\mathcal{M}_{\cdot,\,S_{T}}$ is as in
\eqref{eq:defcontinM}}.
\]
We stress that the constant $c$ in \eqref{basic sup}-\eqref{basic Dini} is
independent of the ball $B_{r}(\overline{\xi})$. In particular, in view of
Proposition \ref{prop:funMN}, estimate \eqref{basic Dini} expresses
Dini-continuity of $\partial_{x_{i}x_{j}}^{2}u$ provided that both
$\mathcal{L}u$ and $a_{ij}$ are \emph{log-Dini continuous}, while it just
expresses uniform continuity if $\mathcal{L}u$ and $a_{ij}$ are only Dini continuous.
\end{theorem}

\begin{proof}
We follow and revise the proof of \cite[Thm.\,4.1]{BB}. To begin with, we
ar\-bi\-tra\-rily fix $0<r_{0}\leq1/2$ (to be suitably chosen later on) and a
point $\overline{\xi}=(\overline{x},\overline{t})\in S_{T}$. We then consider
the o\-pe\-ra\-tor $\mathcal{L}_{\overline{x}}$ with coefficients
$a_{ij}(\overline{x},t)$ (frozen in space, variable in time), and we let
$\Gamma^{\overline{x}}$ be its fundamental solution. We now observe that,
given any $u\in\mathcal{S}^{D}(S_{T})$ with $\mathrm{supp}(u)\subseteq
B_{r}(\overline{\xi})\cap\overline{S_{T}}$ (for some $0<r\leq r_{0}$), we
clearly have $u\in\mathcal{S}^{0}(\overline{t}-r,T)$ and $\mathcal{L}%
_{\overline{x}}u\in\mathcal{D}(S_{T})$; thus, we can exploit the
representation formula in Corollary \ref{Corollary repr formula uxx}, giving
\begin{align*}
&  \partial_{x_{i}x_{j}}^{2}u(x,t)\\
&  \quad=\int_{\overline{t}-r}^{t}\left(  \int_{\mathbb{R}^{N}}\partial
_{x_{i}x_{j}}^{2}\Gamma^{\overline{x}}(x,t;y,s)\left[  \mathcal{L}%
_{\overline{x}}u(E(s-t)x,s)-\mathcal{L}_{\overline{x}}u(y,s)\right]
dy\right)  ds,
\end{align*}
for every $(x,t)\in B_{r}^{T}(\overline{\xi})$. From this, since we can write
\begin{align*}
\mathcal{L}_{\overline{x}}u  &  =\mathcal{L}u+(\mathcal{L}_{\overline{x}%
}-\mathcal{L})u\\
&  =\mathcal{L}u+\sum_{h,k=1}^{m_{0}}\big(a_{hk}(\overline{x},t)-a_{hk}%
(x,t)\big)\partial_{x_{h}x_{k}}^{2}u,
\end{align*}
we obtain the following identity
\begin{equation}%
\begin{split}
\partial_{x_{i}x_{j}}^{2}u(x,t)  &  =\int_{\overline{t}-r}^{t}\left(
\int_{\mathbb{R}^{N}}\partial_{x_{i}x_{j}}^{2}\Gamma^{\overline{x}%
}(x,t;y,s)\big\{\mathcal{L}u(E(s-t)x,s)-\mathcal{L}u(y,s)\big\}dy\right)  ds\\
&  \qquad+\sum_{h,k=1}^{m_{0}}\int_{\overline{t}-1}^{t}\int_{\mathbb{R}^{N}%
}\partial_{x_{i}x_{j}}^{2}\Gamma^{\overline{x}}(x,t;y,s)\times\\
&  \qquad\qquad\times\Big\{\big(a_{hk}(\overline{x},s)-a_{hk}%
(E(s-t)x,s)\big)\partial_{x_{h}x_{k}}^{2}u(E(s-t)x,s)\\
&  \qquad\qquad\qquad-\big(a_{hk}(\overline{x},s)-a_{hk}(y,s)\big)\partial
_{x_{h}x_{k}}^{2}u(y,s)\Big\}dyds\\
&  =T_{ij}(\mathcal{L}u)(x,t)+\sum_{h,k=1}^{m_{0}}T_{ij}(f_{hk})(x,t),
\end{split}
\label{eq:reprformulaTij}%
\end{equation}
where $T_{ij}(\cdot)$ is as in Theorem \ref{Thm Dini sing int}, and
\[
f_{hk}(y,s)=\big(a_{hk}(\overline{x},s)-a_{hk}(y,s)\big)\partial_{x_{h}x_{k}%
}^{2}u(y,s)\in\mathcal{D}(\overline{t}-r,T).
\]
To proceed further, we turn to estimate the $L^{\infty}$-norm and the
continuity mo\-du\-lus of $T_{ij}(\mathcal{L}u)$ and of $T_{ij}(f_{hk})$ (for
$1\leq h,k\leq m_{0}$) on $B_{r}^{T}(\overline{\xi})$.

\noindent(1)\thinspace\thinspace\textsc{Estimate of the $L^{\infty}$-norm}.
First of all we observe that, since we as\-su\-ming $0<r\leq r_{0}\leq1/2$, by
Proposition \ref{prop:Analoga313} we get the following estimate
\begin{equation}%
\begin{split}
&  \Vert T_{ij}(\mathcal{L}u)\Vert_{L^{\infty}(B_{r}^{T}(\overline{\xi}))}\\
&  \qquad\leq\sup_{(x,t)\in B_{r}(\overline{\xi})}\int_{\RN\times(\overline
{t}-r,t)}\partial_{x_{i}x_{j}}^{2}\Gamma^{\overline{x}}(x,t;y,s)\omega
_{\mathcal{L}u,S_{T}}(\Vert E(s-t)x-y\Vert)\,dy\,ds\\
&  \qquad\leq c\,\sup_{(x,t)\in B_{r}(\overline{\xi})}\mathcal{U}%
_{\mathcal{L}u,S_{T}}^{\mu}(\sqrt{t-\overline{t}+r})=c\,\mathcal{U}%
_{\mathcal{L}u,S_{T}}^{\mu}(\sqrt{2r})\leq c\,\mathcal{U}_{\mathcal{L}u,S_{T}%
}^{\mu}(1),
\end{split}
\label{eq:estimTijLu}%
\end{equation}
where we have used the fact that $\mathcal{U}_{\mathcal{L}u,S_{T}}^{\mu}$ is
non-decreasing, see \eqref{eq:defUmuIntro}, and $c,\mu>0$ are structural
constants. Analogously, since $0<r\leq r_{0}$, we have%
\begin{equation}%
\begin{split}
\Vert T_{ij}(f_{hk})\Vert_{L^{\infty}(B_{r}^{T}(\overline{\xi}))}  &  \leq
c\,\mathcal{U}_{f_{hk},S_{T}}^{\mu}(\sqrt{2r})\leq c\,\mathcal{U}%
_{f_{hk},S_{T}}^{\mu}(\sqrt{2r_{0}})\\
&  =c\int_{\RN}e^{-\mu|z|^{2}}\Big(\int_{0}^{\sqrt{2r_{0}}\Vert z\Vert}%
\frac{\omega_{f_{hk},S_{T}}(s)}{s}\,ds\Big)dz
\end{split}
\label{eq:estimTijsupIntermedia}%
\end{equation}
Now, by exploiting the \emph{product structure} of $f_{hk}$, together with
Lemma \ref{lem:omegasuppcpt} (note that $f_{hk}=\partial_{x_{h}x_{k}}^{2}u=0$
on $S_{T}\setminus B_{r}(\overline{\xi})$) and \eqref{eq:boundwf}, we can
write
\begin{equation}%
\begin{split}
&  \omega_{f_{hk},S_{T}}(\rho)=\omega_{f_{hk},\overline{B}_{r}(\overline{\xi
})\cap S_{T}}(\rho)\\
&  \qquad\leq\sup_{(y,s)\in\overline{B}_{r}(\overline{\xi})\cap S_{T}}%
|a_{hk}(y,s)-a_{hk}(\overline{x},s)|\cdot\omega_{\partial_{x_{h}x_{k}}%
^{2}u,S_{T}}(\rho)\\
&  \qquad\qquad\quad+\omega_{a_{hk},S_{T}}(\rho)\cdot\Vert\partial_{x_{h}%
x_{k}}^{2}u\Vert_{L^{\infty}(B_{r}^{T}(\overline{\xi}))}\\
&  \qquad\leq2\,\omega_{a_{hk},S_{T}}(\rho)\cdot\Vert\partial_{x_{h}x_{k}}%
^{2}u\Vert_{L^{\infty}(B_{r}^{T}(\overline{\xi}))}+\omega_{a_{hk},S_{T}}%
(\rho)\cdot\Vert\partial_{x_{h}x_{k}}^{2}u\Vert_{L^{\infty}(B_{r}%
^{T}(\overline{\xi}))}\\
&  \qquad\leq3\,\omega_{a,S_{T}}(\rho)\cdot\Vert\partial_{x_{h}x_{k}}%
^{2}u\Vert_{L^{\infty}(B_{r}^{T}(\overline{\xi}))}\qquad(\text{for all
$\rho>0$}),
\end{split}
\label{eq:estimomegafhk}%
\end{equation}
where $\omega_{a,S_{T}}(\cdot)=\sum_{h,k=1}^{m_{0}}\omega_{a_{hk},S_{T}}%
(\cdot)$. Thus, by combining
\eqref{eq:estimTijsupIntermedia}-\eqref{eq:estimomegafhk} we get
\begin{equation}%
\begin{split}
&  \Vert T_{ij}(f_{hk})\Vert_{L^{\infty}(B_{r}^{T}(\overline{\xi}))}\\
&  \qquad\leq c\,\Vert\partial_{x_{h}x_{k}}^{2}u\Vert_{L^{\infty}(B_{r}%
^{T}(\overline{\xi}))}\int_{\RN}e^{-\mu|z|^{2}}\Big(\int_{0}^{\sqrt{2r_{0}%
}\Vert z\Vert}\frac{\omega_{a,S_{T}}(s)}{s}\,ds\Big)dz\\
&  \qquad\equiv c\,\Vert\partial_{x_{h}x_{k}}^{2}u\Vert_{L^{\infty}(B_{r}%
^{T}(\overline{\xi}))}\,\mathcal{U}_{a,S_{T}}^{\mu}(\sqrt{2r_{0}%
})\phantom{\int_0^1}\qquad(\text{for all $1\leq i,j\leq m_{0}$}).
\end{split}
\label{eq:estimTijfhksup}%
\end{equation}
Gathering \eqref{eq:reprformulaTij}, \eqref{eq:estimTijLu} and
\eqref{eq:estimTijfhksup}, we finally obtain
\begin{align*}
&  \max_{1\leq i,j\leq m_{0}}\Vert\partial_{x_{i}x_{j}}^{2}u\Vert_{L^{\infty
}(B_{r}^{T}(\overline{\xi}))}\\
&  \qquad\leq c\Big(\mathcal{U}_{\mathcal{L}u,S_{T}}^{\mu}(1)+\mathcal{U}%
_{a,S_{T}}^{\mu}(\sqrt{2r_{0}})\sum_{h,k=1}^{m_{0}}\Vert\partial_{x_{h}x_{k}%
}^{2}u\Vert_{L^{\infty}(B_{r}^{T}(\overline{\xi}))}\Big)\\
&  \qquad\leq c\Big(\mathcal{U}_{\mathcal{L}u,S_{T}}^{\mu}(1)+\max_{1\leq
i,j\leq m_{0}}\Vert\partial_{x_{i}x_{j}}^{2}u\Vert_{L^{\infty}(B_{r}%
^{T}(\overline{\xi}))}\cdot\mathcal{U}_{a,S_{T}}^{\mu}(\sqrt{2r_{0}})\Big),
\end{align*}
where $c>0$ is a constant, possibly different from line to line. From this, if
we choose $0<r_{0}\leq1/2$ so small that
\begin{equation}
c\,\mathcal{U}_{a,S_{T}}^{\mu}(\sqrt{2r_{0}})\leq\frac{1}{2}
\label{eq:choicerzero}%
\end{equation}
(recall that $\mathcal{U}_{a,S_{T}}^{\mu}(r)$ vanishes as $r\rightarrow0^{+}$,
see Lemma \ref{lem:Ufwelldef}), we immediately derive the desired
\eqref{basic sup}. We explicitly point out that the choice of $r_{0}$ (in such
a way that \eqref{eq:choicerzero} is satisfied) \emph{only depends on the
coefficients $a_{hk}$}. \medskip

\noindent(2)\thinspace\thinspace\textsc{Estimate of the continuity modulus}.
First of all we observe that, by combining the representation formula
\eqref{eq:reprformulaTij} with Theorem \ref{Thm Dini sing int}, we get
\begin{equation}%
\begin{split}
\omega_{\partial_{x_{i}x_{j}}^{2}u,B_{r}^{T}(\overline{\xi})}(\rho)  &
\leq\omega_{T_{ij}(\mathcal{L}u),S_{T}}(\rho)+\sum_{h,k=1}^{m_{0}}%
\omega_{T_{ij}(f_{hk}),S_{T}}(\rho)\\
&  \leq c\Big(\mathcal{M}_{\mathcal{L}u,S_{T}}(c\rho)+\sum_{h,k=1}^{m_{0}%
}\mathcal{M}_{f_{hk},S_{T}}(c\rho)\Big),
\end{split}
\label{eq:estimomegade2StepI}%
\end{equation}
where $c>0$ is a structural constant. On the other hand, using
\eqref{eq:estimomegafhk} (and taking into account the very definition of
$\mathcal{M}_{f_{hk},S_{T}}$, see \eqref{eq:defcontinM}), we can write
\begin{equation}%
\begin{split}
\mathcal{M}_{f_{hk},S_{T}}(\rho)  &  =\omega_{f_{hk},S_{T}}(\rho)+\int%
_{0}^{\rho}\frac{\omega_{f_{hk},S_{T}}(s)}{s}\,ds+\rho\int_{\rho}^{\infty
}\frac{\omega_{f_{hk},S_{T}}(s)}{s^{2}}\,ds\\
&  \leq3\,\Vert\partial_{x_{h}x_{k}}^{2}u\Vert_{L^{\infty}(B_{r}^{T}%
(\overline{\xi}))}\mathcal{M}_{a,S_{T}}(\rho)\qquad(\text{for all $\rho>0$}).
\end{split}
\label{eq:estimMfhk}%
\end{equation}
By combining \eqref{eq:estimomegade2StepI}-\eqref{eq:estimMfhk} with
\eqref{basic sup} (which has been already proved), we then obtain the
following estimate, provided that $r_{0}$ is small enough:
\begin{align*}
&  \omega_{\partial_{x_{i}x_{j}}^{2}u,B_{r}^{T}(\overline{\xi})}(\rho)\leq
c\Big(\mathcal{M}_{\mathcal{L}u,S_{T}}(c\rho)+\mathcal{M}_{a,S_{T}}%
(c\rho)\cdot\sum_{h,k=1}^{m_{0}}\Vert\partial_{x_{h}x_{k}}^{2}u\Vert
_{L^{\infty}(B_{r}^{T}(\overline{\xi}))}\Big)\\
&  \qquad\leq c\big(\mathcal{M}_{\mathcal{L}u,S_{T}}(c\rho)+\mathcal{M}%
_{a,S_{T}}(c\rho)\cdot\mathcal{U}_{\mathcal{L}u,S_{T}}^{\mu}(1)\big)\qquad
(\text{for all $\rho>0$}),
\end{align*}
This is precisely the desired \eqref{basic Dini}, and the proof is complete.
\end{proof}

\subsection{The continuity estimate in the general case}

Given an arbitrary open set $\Omega\subseteq\mathbb{R}^{N+1}$ and a function
$f:\Omega\rightarrow\mathbb{R}$, we recall that the (partial) continuity
modulus $\omega_{f,\Omega}$ of $f$ is defined as follows:
\[
\omega_{f,\Omega}(r)=\sup_{\begin{subarray}{c}
(x,t),(y,t)\in\Omega \\
\|x-y\| \leq r
\end{subarray}}|f(x,t)-f(y,t)|\qquad(r>0).
\]
In the following, we will get a control on $\omega_{f,S_{T}}$ starting with a
uniform control on the moduli $\omega_{f,B_{r}(\overline{\xi}_{i})}$ where
$\{B_{r}(\overline{\xi}_{i})\}_{i=1}^{\infty}$ is a covering of $S_{T}$.
\medskip

This is possible in view of the following:

\begin{proposition}
\label{Prop covering} Let $r>0$ be fixed, and let $\{B_{r}(\overline{\xi}%
_{i})\}_{i=1}^{\infty}$ be a covering of $S_{T}$, that is, $S_{T}%
\subseteq\bigcup_{i}B_{r}(\overline{\xi_{i}})$. Then, we have
\[
\omega_{f,S_{T}}(\rho)\leq\sup_{i}\omega_{f,B_{\theta r}^{T}(\overline{\xi
}_{i})}(\rho)\quad\text{ for every $0<\rho\leq r$}.
\]
where $\theta\geq1$ is a structural constant.
\end{proposition}

\begin{proof}
Let $(x_{1},t),\,(x_{2},t)\in S_{T}$ be two points satisfying $\Vert
x_{1}-x_{2}\Vert=s\leq r$, and let $i_{1}\in\mathbb{N}$ be such that
$(x_{1},t)\in B_{r}^{T}(\overline{\xi}_{i_{1}})$. Using the quasi-triangle
inequality of $d$, see \eqref{eq:quasitriangled}, we infer that $(x_{2},t)\in
B_{\theta r}^{T}(\overline{\xi}_{i_{1}})$ for some structural constant
$\theta\geq1$ (actually, we have $\theta=\bd\kappa(1+\bd\kappa)$); as a
consequence, we get
\[
|f(x_{1},t)-f(x_{2},t)|\leq\omega_{f,B_{\theta r}^{T}(\overline{\xi}_{i_{1}}%
)}(s)\leq\sup_{i}\omega_{f,B_{\theta r}^{T}(\overline{\xi}_{i})}(s),
\]
and this implies the assertion.
\end{proof}

Thanks to all the results established so far, we can now give the

\bigskip

\begin{proof}
[Proof of Theorem \ref{Thm main space}]To begin with, we fix $r>0$ so small
that the \emph{local conti\-nu\-ity estimates} in Theorem
\ref{Thm local Dini x} hold on balls of radius $2\theta r$ (where $\theta
\geq1$ is as in Proposition \ref{Prop covering}), and we let $\{B_{r}%
(\overline{\xi}_{n})\}_{n\geq1}$ be a covering of $S_{T}$. We then choose a
function $\Phi\in C_{0}^{\infty}(B_{2\theta r}(0))$ satisfying $\text{$\Phi
\equiv1$ in $B_{\theta r}(0)$}$, and we define
\[
\phi_{n}(\xi)=\Phi(\overline{\xi}_{n}^{-1}\circ\xi)\qquad(n\geq1).
\]
Note that, by \eqref{eq:balltraslD}, $\phi_{n}\in C_{0}^{\infty}(B_{2\theta
r}(\overline{\xi}_{n}))$ and $\phi_{n}\equiv1$ in $B_{\theta r}(\overline{\xi
})$; mo\-re\-o\-ver, by the le\-ft\--in\-va\-riance of $\partial_{x_{1}%
},\ldots,\partial_{x_{m_{0}}},Y$ we see that the $C^{\alpha}$-norms of
$\phi_{n},\partial_{x_{k}}\phi_{n},\mathcal{L}(\phi_{n})$ are \emph{bounded
independently of $n$} (for all $\alpha\in(0,1)$). Throughout this proof, the
constants involved may depend on $r$, which however is by now fixed.
\vspace*{0.1cm} We now arbitrarily fix $n\geq1$ and we observe that, since
$u_{n}=u\phi_{n}\in\mathcal{S}^{D}(S_{T})$ and since $\mathrm{supp}%
(u_{n})\subseteq B_{2\theta r}^{T}(\overline{\xi}_{n})$, we can apply the
estimates \eqref{basic sup}-\eqref{basic Dini} in Theo\-rem
\ref{Thm local Dini x} to this function $u_{n}$: recalling that $\phi
_{n}\equiv1$ in $B_{\theta r}(\overline{\xi}_{n})$, this gives
\begin{align}
&  \Vert\partial_{x_{i}x_{j}}^{2}u\Vert_{L^{\infty}(B_{\theta r}^{T}%
(\overline{\xi}_{n}))}\leq\Vert\partial_{x_{i}x_{j}}^{2}u_{n}\Vert_{L^{\infty
}(B_{2\theta r}^{T}(\overline{\xi}_{n}))}\leq c\,\mathcal{U}_{\mathcal{L}%
u_{n},S_{T}}^{\mu}(1)\label{Dini locale 0}\\
&  \omega_{\partial_{x_{i}x_{j}}^{2}u,B_{\theta r}^{T}(\overline{\xi}_{n}%
)}(\rho)\leq\omega_{\partial_{x_{i}x_{j}}^{2}u_{n},B_{2\theta r}^{T}%
(\overline{\xi}_{n})}(\rho)\nonumber\\
&  \qquad\quad\leq c\big(\mathcal{M}_{\mathcal{L}u_{n},S_{T}}(c\rho
)+\mathcal{M}_{a,S_{T}}(c\rho)\cdot\mathcal{U}_{\mathcal{L}u_{n},S_{T}}^{\mu
}(1)\big)\quad\forall\,\,\rho>0, \label{Dini locale}%
\end{align}
where $c,\mu>0$ are structural constants (and $1\leq i,j\leq m_{0}$). On the
other hand, since a direct computation shows that
\[
\mathcal{L}u_{n}=\textstyle(\mathcal{L}u)\phi_{n}+u(\mathcal{L}\phi_{n}%
)+2\sum_{h,k=1}^{m_{0}}a_{hk}\partial_{x_{h}}u\cdot\partial_{x_{k}}\phi_{n},
\]
we clearly have the following estimate
\begin{equation}%
\begin{split}
\omega_{\mathcal{L}u_{n},S_{T}}(\rho)  &  \leq\omega_{(\mathcal{L}u)\phi
_{n},S_{T}}(\rho)+\omega_{u(\mathcal{L}\phi_{n}),S_{T}}(\rho)\\[0.1cm]
&  \qquad+\textstyle2\sum_{h,k=1}^{m_{0}}\omega_{a_{hk}\partial_{x_{h}}%
u\cdot\partial_{x_{k}}\phi_{n},S_{T}}(\rho).
\end{split}
\label{eq:omegaLLunsum}%
\end{equation}
In view of \eqref{Dini locale 0}-\eqref{Dini locale}, and taking into account
the above \eqref{eq:omegaLLunsum}, to prove the theorem we then turn to
estimate the three \emph{continuity moduli}
\[
\mathrm{(1)}\,\,\omega_{(\mathcal{L}u)\phi_{n},S_{T}},\qquad\mathrm{(2)}%
\,\,\omega_{u(\mathcal{L}\phi_{n}),S_{T}},\qquad\mathrm{(3)}\,\,\omega
_{a_{hk}\partial_{x_{h}}u\cdot\partial_{x_{k}}\phi_{n},S_{T}}.
\]
To this end we will repeatedly use the following straightforward estimate,
holding true for every open set $\Omega\subseteq\mathbb{R}^{N+1}$ and every
$f,g\in\mathcal{D}(\Omega)$:
\begin{equation}
\omega_{fg,\Omega}(\rho)\leq\Vert f\Vert_{L^{\infty}(\Omega)}\omega_{g,\Omega
}(\rho)+\Vert g\Vert_{L^{\infty}(\Omega)}\omega_{f,\Omega}(\rho)\qquad
\forall\,\,\rho>0. \label{eq:estimomegafg}%
\end{equation}

\noindent-\thinspace\thinspace\textsc{Estimate of} (1). On account of
\eqref{eq:estimomegafg}, and since $\phi_{n}\in C_{0}^{\infty}(\mathbb{R}%
^{N+1})$ (hence, in par\-ti\-cu\-lar, $\phi_{n}\in C^{\alpha}(\mathbb{R}%
^{N+1})$ for every $\alpha\in(0,1)$), we immediately get
\begin{equation}%
\begin{split}
\omega_{(\mathcal{L}u)\phi_{n},S_{T}}(\rho)  &  \leq\Vert\mathcal{L}%
u\Vert_{L^{\infty}(S_{T})}\omega_{\phi_{n},S_{T}}(\rho)+\Vert\phi_{n}%
\Vert_{L^{\infty}(S_{T})}\omega_{\mathcal{L}u,S_{T}}(\rho)\\
&  \leq c\big(\rho^{\alpha}\Vert\mathcal{L}u\Vert_{L^{\infty}(S_{T})}%
+\omega_{\mathcal{L}u,S_{T}}(\rho)\big)\qquad\forall\,\,\rho>0,
\end{split}
\label{eq:estimomegaLLuphinFINAL}%
\end{equation}
where $c>0$ is a constant only depending on $\Phi$. \medskip

\noindent-\thinspace\thinspace\textsc{Estimate of} (2). Using once again
\eqref{eq:estimomegafg}, and taking into account that $u\mathcal{L}\phi_{n}$
is compactly supported in $B_{2\theta r}(\overline{\xi}_{n})$, by Lemma
\ref{lem:omegasuppcpt} we can write
\begin{align*}
\omega_{u(\mathcal{L}\phi_{n}),S_{T}}(\rho)  &  =\omega_{u(\mathcal{L}\phi
_{n}),\overline{B}_{2\theta r}(\overline{\xi}_{n})\cap S_{T}}(\rho)\\
&  \leq\Vert u\Vert_{L^{\infty}(B_{2\theta r}^{T}(\overline{\xi}_{n}))}%
\omega_{\mathcal{L}\phi_{n},S_{T}}(\rho)\\
&  \qquad+\Vert\mathcal{L}\phi_{n}\Vert_{L^{\infty}(B_{2\theta r}%
^{T}(\overline{\xi}_{n}))}\omega_{u,\overline{B}_{2\theta r}(\overline{\xi
}_{n})\cap S_{T}}(\rho)\\
&  (\text{since $\mathcal{L}\phi_{n}\in C^{\alpha}(\mathbb{R}^{N+1})$ for
every $0<\alpha<1$})\\
&  \leq c\big(\rho^{\alpha}\Vert u\Vert_{L^{\infty}(B_{2\theta r}%
^{T}(\overline{\xi}_{n}))}+\omega_{u,\overline{B}_{2\theta r}(\overline{\xi
}_{n})\cap S_{T}}(\rho)\big)=(\bigstar),
\end{align*}
where $c>0$ is a constant only depending on $\Phi$. On the other hand, since
we know from Theorem \ref{Thm interpolaz} that $u\in C^{\alpha}(B\cap S_{T})$
\emph{for every ball $B=B_{R}(\overline{\eta})$} (with $\overline{\eta}\in
S_{T}$) and every $\alpha\in(0,1)$, we obtain
\begin{equation}
(\bigstar)\leq c\,\rho^{\alpha}\Vert u\Vert_{C^{\alpha}(B_{2\theta r}%
^{T}(\overline{\xi}_{n}))}\qquad\forall\,\,\rho>0.
\label{eq:estimomegauLLphinFINAL}%
\end{equation}

\noindent-\thinspace\thinspace\textsc{Estimate of} (3). By repeatedly
exploiting \eqref{eq:estimomegafg}, and by taking into account the smoothness
and support of $\phi_{n}$, we derive the following estimate
\begin{align*}
&  \omega_{a_{hk}\partial_{x_{h}}u\cdot\partial_{x_{k}}\phi_{n},S_{T}}%
(\rho)=\omega_{a_{hk}\partial_{x_{h}}u\cdot\partial_{x_{k}}\phi_{n}%
,\overline{B}_{2\theta r}(\overline{\xi}_{n})\cap S_{T}}(\rho)\\
&  \qquad\leq\omega_{a_{hk},S_{T}}(\rho)\,\Vert\partial_{x_{h}}u\Vert
_{L^{\infty}(B_{2\theta r}^{T}(\overline{\xi}_{n}))}\Vert\partial_{x_{k}}%
\phi_{n}\Vert_{L^{\infty}(B_{2\theta r}(\overline{\xi}_{n}))}\\
&  \qquad\qquad+\omega_{\partial_{x_{h}}u,\overline{B}_{2\theta r}%
(\overline{\xi}_{n})\cap S_{T}}(\rho)\,\Vert a_{hk}\Vert_{L^{\infty
}(\mathbb{R}^{N+1})}\Vert\partial_{x_{k}}\phi_{n}\Vert_{L^{\infty}%
(\mathbb{R}^{N+1}))}\\
&  \qquad\qquad\qquad+\omega_{\partial_{x_{k}}\phi_{n},S_{T}}(\rho)\,\Vert
a_{hk}\Vert_{L^{\infty}(\mathbb{R}^{N+1})}\Vert\partial_{x_{h}}u\Vert
_{L^{\infty}({B}_{2\theta r}^{T}(\overline{\xi}_{n}))}\\
&  \qquad\leq c\big(\omega_{a_{hk},S_{T}}(\rho)\,\Vert\partial_{x_{h}}%
u\Vert_{L^{\infty}(B_{2\theta r}^{T}(\overline{\xi}_{n}))}+\omega
_{\partial_{x_{h}}u,\overline{B}_{2\theta r}(\overline{\xi}_{n})\cap S_{T}%
}(\rho)\\
&  \qquad\qquad+\rho^{\alpha}\Vert\partial_{x_{h}}u\Vert_{L^{\infty}%
({B}_{2\theta r}^{T}(\overline{\xi}_{n}))}\big)=:(\bigstar),
\end{align*}
where $c>0$ is a constant only depending on $\Phi$ and on $\nu$ in \eqref{nu}.
On the other hand, since we know from Theorem \ref{Thm interpolaz} that
$\partial_{x_{h}}u\in C^{\alpha}(B\cap S_{T})$ \emph{for every ball
$B=B_{R}(\overline{\eta})$} (with $\overline{\eta}\in S_{T}$) and every
$\alpha\in(0,1)$, we obtain
\begin{equation}%
\begin{split}
(\bigstar)  &  \leq c\big(\omega_{a_{hk},S_{T}}(\rho)\,\Vert\partial_{x_{h}%
}u\Vert_{L^{\infty}(B_{2\theta r}(\overline{\xi}_{n}))}\\
&  \qquad+\rho^{\alpha}\Vert\partial_{x_{h}}u\Vert_{C^{\alpha}(B_{2\theta
r}^{T}(\overline{\xi}_{n}))}\big)\qquad\forall\,\,\rho>0.
\end{split}
\label{eq:estimomegadeudephinFINAL}%
\end{equation}
Gathering
\eqref{eq:estimomegaLLuphinFINAL}-to-\eqref{eq:estimomegadeudephinFINAL}, from
\eqref{eq:omegaLLunsum} we then get
\begin{equation}%
\begin{split}
\omega_{\mathcal{L}u_{n},S_{T}}(\rho)  &  \leq c\Big(\rho^{\alpha}%
\Vert\mathcal{L}u\Vert_{L^{\infty}(S_{T})}+\omega_{\mathcal{L}u,S_{T}}%
(\rho)+\rho^{\alpha}\Vert u\Vert_{C^{\alpha}(B_{2\theta r}^{T}(\overline{\xi
}_{n}))}\\
&  \qquad+\sum_{h,k=1}^{m_{0}}\big(\omega_{a_{hk},S_{T}}(\rho)\,\Vert
\partial_{x_{h}}u\Vert_{L^{\infty}(B_{2\theta r}(\overline{\xi}_{n}))}\\
&  \qquad\qquad+\rho^{\alpha}\Vert\partial_{x_{h}}u\Vert_{C^{\alpha
}(B_{2\theta r}^{T}(\overline{\xi}_{n}))}\big)\Big)\\
&  (\text{setting, as usual, $\textstyle\omega_{a,S_{T}}=\sum_{h,k=1}^{m_{0}%
}\omega_{a_{hk},S_{T}}$})\\
&  \leq c\Big(\rho^{\alpha}\Vert\mathcal{L}u\Vert_{L^{\infty}(S_{T})}%
+\omega_{\mathcal{L}u,S_{T}}(\rho)+\rho^{\alpha}\Vert u\Vert_{C^{\alpha
}(B_{2\theta r}^{T}(\overline{\xi}_{n}))}\\
&  \qquad+\big(\omega_{a,S_{T}}(\rho)+\rho^{\alpha}\big)\sum_{h=1}^{m_{0}%
}\Vert\partial_{x_{h}}u\Vert_{C^{\alpha}(B_{2\theta r}^{T}(\overline{\xi}%
_{n}))}\Big),
\end{split}
\label{eq:estimomegaLLunFINAL}%
\end{equation}
and this estimate holds \emph{for every $\rho>0$}. With
\eqref{eq:estimomegaLLunFINAL} at hand, we are now ready to establish
assertions (i)-(ii) in the statement of the theorem. \medskip\noindent

-\thinspace\thinspace\emph{Proof of} (i). First of all, by combining estimates
\eqref{Dini locale 0}-\eqref{eq:estimomegaLLunFINAL} and by e\-xplo\-i\-ting
Lemma \ref{lem:Ufwelldef}-(i), we derive the bound
\begin{equation}%
\begin{split}
\Vert\partial_{x_{i}x_{j}}^{2}  &  u\Vert_{L^{\infty}(B_{r}^{T}(\overline{\xi
}_{n}))}\leq c\,\mathcal{U}_{\mathcal{L}u_{n},S_{T}}^{\mu}(1)\\
&  =c\int_{\RN}e^{-\mu|z|^{2}}\Big(\int_{0}^{\Vert z\Vert}\frac{\omega
_{\mathcal{L}u_{n},S_{T}}(s)}{s}\,ds\Big)dz\\
&  \leq c\Big(\Vert\mathcal{L}u\Vert_{L^{\infty}(S_{T})}+\mathcal{U}%
_{\mathcal{L}u,S_{T}}^{\mu}(1)\\
&  \qquad+\sum_{h=1}^{m_{0}}\Vert\partial_{x_{h}}u\Vert_{C^{\alpha}(B_{2\theta
r}^{T}(\overline{\xi}_{n}))}+\Vert u\Vert_{C^{\alpha}(B_{2\theta r}%
^{T}(\overline{\xi}_{n}))}\Big)\\
&  \leq c\,\Big(\Vert\mathcal{L}u\Vert_{\mathcal{D}(S_{T})}+\sum_{h=1}^{m_{0}%
}\Vert\partial_{x_{h}}u\Vert_{C^{\alpha}(B_{2\theta r}^{T}(\overline{\xi}%
_{n}))}+\Vert u\Vert_{C^{\alpha}(B_{2\theta r}^{T}(\overline{\xi}_{n}))}\Big)
\end{split}
\label{eq:dausareassertionii}%
\end{equation}
where $c>0$ now depends on the chosen $\alpha$ and on the number $A$ in
\eqref{a cost Dini}. From this, by using the interpolation inequality
\eqref{disug interpolaz} in Th\-e\-o\-rem \ref{Thm interpolaz}, we obtain
\begin{equation}%
\begin{split}
&  \Vert\partial_{x_{i}x_{j}}^{2}u\Vert_{L^{\infty}(B_{r}^{T}(\overline{\xi
}_{n}))}\\
&  \qquad\leq c\Big\{\Vert\mathcal{L}u\Vert_{\mathcal{D}(S_{T})}%
+\varepsilon\Big(\sum_{h,k=1}^{m_{0}}\Vert\partial_{x_{k}x_{h}}^{2}%
u\Vert_{L^{\infty}(S_{T})}+\Vert Yu\Vert_{L^{\infty}(S_{T})}\Big)\\
&  \qquad\qquad+\frac{1}{\varepsilon^{\gamma}}\Vert u\Vert_{L^{\infty}(S_{T}%
)}\Big\},
\end{split}
\label{eq:dovefaresupLinf}%
\end{equation}
and this estimate holds \emph{for every $\e\in(0,1)$}. We then observe that,
since $n\geq1$ is ar\-bi\-tra\-rily fixed, by taking the $\sup$ over
$\mathbb{N}$ in the previous inequality we get
\begin{align*}
&  \Vert\partial_{x_{i}x_{j}}^{2}u\Vert_{L^{\infty}(S_{T})}\leq c\Big\{\Vert
\mathcal{L}u\Vert_{\mathcal{D}(S_{T})}+\varepsilon\Big(\sum_{h,k=1}^{m_{0}%
}\Vert\partial_{x_{k}x_{h}}^{2}u\Vert_{L^{\infty}(S_{T})}+\Vert Yu\Vert
_{L^{\infty}(S_{T})}\Big)\\
&  \qquad\qquad\quad+\frac{1}{\varepsilon^{\gamma}}\Vert u\Vert_{L^{\infty
}(S_{T})}\Big\};
\end{align*}
moreover, since $Yu=\mathcal{L}u-\sum_{h,k=1}^{m_{0}}a_{hk}\partial
_{x_{h}x_{k}}^{2}u$, by exploiting assumption (H3) we can write (up to
possibly change the constant $c$)
\[
\Vert\partial_{x_{i}x_{j}}^{2}u\Vert_{L^{\infty}(S_{T})}\leq c\Big\{\Vert
\mathcal{L}u\Vert_{\mathcal{D}(S_{T})}+\varepsilon\sum_{h,k=1}^{m_{0}}%
\Vert\partial_{x_{k}x_{h}}^{2}u\Vert_{L^{\infty}(S_{T})}+\frac{1}%
{\varepsilon^{\gamma}}\Vert u\Vert_{L^{\infty}(S_{T})}\Big\}.
\]
Thus, if we choose $\e>0$ so small that $c\,\e<1/2$, we conclude that
\begin{equation}
\Vert\partial_{x_{i}x_{j}}^{2}u\Vert_{L^{\infty}(S_{T})}\leq c\big(\Vert
\mathcal{L}u\Vert_{\mathcal{D}(S_{T})}+\Vert u\Vert_{L^{\infty}(S_{T})}\big),
\label{eq:assertioniStepI}%
\end{equation}
and this implies, again by the identity $Yu=\mathcal{L}u-\sum_{h,k=1}^{m_{0}%
}a_{hk}\partial_{x_{h}x_{k}}^{2}u$,
\begin{equation}%
\begin{split}
\Vert Yu\Vert_{L^{\infty}(S_{T})}  &  \leq\Vert\mathcal{L}u\Vert_{L^{\infty
}(S_{T})}+\sum_{h,k=1}^{m_{0}}\Vert a_{hk}\Vert_{L^{\infty}(\mathbb{R}^{N+1}%
)}\Vert\partial_{x_{h}x_{k}}^{2}u\Vert_{L^{\infty}(S_{T})}\\
&  \leq c\big(\Vert\mathcal{L}u\Vert_{\mathcal{D}(S_{T})}+\Vert u\Vert
_{L^{\infty}(S_{T})}\big).
\end{split}
\label{eq:assertioniStepII}%
\end{equation}
In view of \eqref{eq:assertioniStepI}-\eqref{eq:assertioniStepII} and Theorem
\ref{Thm interpolaz}, assertion (i) is now established. \medskip\noindent

-\thinspace\thinspace\emph{Proof of} (ii). First of all, by combining
\eqref{Dini locale} with \eqref{eq:estimomegaLLunFINAL} (and by taking into
account the very definition of $\mathcal{M}_{\cdot,S_{T}}$, see
\eqref{eq:defcontinM}), we get
\begin{align*}
&  \omega_{\partial_{x_{i}x_{j}u}^{2},B_{\theta r}^{T}(\overline{\xi}_{n}%
)}(\rho)\leq c\big(\mathcal{M}_{\mathcal{L}u_{n},S_{T}}(c\rho)+\mathcal{M}%
_{a,S_{T}}(c\rho)\cdot\mathcal{U}_{\mathcal{L}u_{n},S_{T}}^{\mu}(1)\big)\\
&  \qquad\leq c\Big\{\rho^{\alpha}\Vert\mathcal{L}u\Vert_{L^{\infty}(S_{T}%
)}+\mathcal{M}_{\mathcal{L}u,S_{T}}(c\rho)\\
&  \qquad\qquad+\big(\mathcal{M}_{a,S_{T}}(c\rho)+\rho^{\alpha}\big)\Big(\sum
_{h=1}^{m_{0}}\Vert\partial_{x_{h}}u\Vert_{C^{\alpha}(B_{2\theta r}%
^{T}(\overline{\xi}_{n}))}+\Vert u\Vert_{C^{\alpha}(B_{2\theta r}%
^{T}(\overline{\xi}_{n}))}\Big)\\
&  \qquad\qquad\qquad+\mathcal{M}_{a,S_{T}}(c\rho)\cdot\mathcal{U}%
_{\mathcal{L}u_{n},S_{T}}^{\mu}(1)\Big\}\\
&  \qquad(\text{by the same computation in \eqref{eq:dausareassertionii}})\\
&  \qquad\leq c\Big\{\mathcal{M}_{\mathcal{L}u,S_{T}}(c\rho)+\big(\mathcal{M}%
_{a,S_{T}}(c\rho)+\rho^{\alpha}\big)\Vert\mathcal{L}u\Vert_{\mathcal{D}%
(S_{T})}\\
&  \qquad\qquad+\big(\mathcal{M}_{a,S_{T}}(c\rho)+\rho^{\alpha}\big)\Big(\sum
_{h=1}^{m_{0}}\Vert\partial_{x_{h}}u\Vert_{C^{\alpha}(B_{2\theta r}%
^{T}(\overline{\xi}_{n}))}+\Vert u\Vert_{C^{\alpha}(B_{2\theta r}%
^{T}(\overline{\xi}_{n}))}\Big)\Big\},
\end{align*}
where $c>0$ depends on $\alpha\in(0,1)$ and on the number $A$. From this, by
exploiting the interpolation inequality \eqref{disug interpolaz} \emph{with
$\e=1$}, jointly with the estimate in assertion (i) (which has been already
established), we obtain
\[%
\begin{split}
&  \omega_{\partial_{x_{i}x_{j}u}^{2},B_{\theta r}^{T}(\overline{\xi}_{n}%
)}(\rho)\\
&  \qquad\leq c\Big\{\mathcal{M}_{\mathcal{L}u,S_{T}}(c\rho)+\big(\mathcal{M}%
_{a,S_{T}}(c\rho)+\rho^{\alpha}\big)\big(\Vert\mathcal{L}u\Vert_{\mathcal{D}%
(S_{T})}+\Vert u\Vert_{L^{\infty}(S_{T})}\big)\\
&  \qquad\qquad+\big(\mathcal{M}_{a,S_{T}}(c\rho)+\rho^{\alpha}\big)\Big(\sum
_{h,k=1}^{m_{0}}\Vert\partial_{x_{k}x_{h}}^{2}u\Vert_{L^{\infty}(S_{T})}+\Vert
Yu\Vert_{L^{\infty}(S_{T})}\Big)\Big\}\\
&  \qquad\leq c\Big\{\mathcal{M}_{\mathcal{L}u,S_{T}}(c\rho)+\big(\mathcal{M}%
_{a,S_{T}}(c\rho)+\rho^{\alpha}\big)\big(\Vert\mathcal{L}u\Vert_{\mathcal{D}%
(S_{T})}+\Vert u\Vert_{L^{\infty}(S_{T})}\big)\Big\}.
\end{split}
\]
We then observe that, since $n\geq1$ is ar\-bi\-tra\-rily fixed, by taking the
supremum over $\mathbb{N}$ in the above estimate and by using Proposition
\ref{Prop covering}, we obtain
\[%
\begin{split}
&  \omega_{\partial_{x_{i}x_{j}u}^{2},S_{T}}(\rho)\leq\sup_{n\in\mathbb{N}%
}\omega_{\partial_{x_{i}x_{j}u}^{2},B_{\theta r}^{T}(\overline{\xi}_{n})}%
(\rho)\\
&  \qquad\leq c\Big\{\mathcal{M}_{\mathcal{L}u,S_{T}}(c\rho)+\big(\mathcal{M}%
_{a,S_{T}}(c\rho)+\rho^{\alpha}\big)\big(\Vert\mathcal{L}u\Vert_{\mathcal{D}%
(S_{T})}+\Vert u\Vert_{L^{\infty}(S_{T})}\big)\Big\}
\end{split}
\]
and this estimate holds for every $0<\rho\leq r$ (note that $r>0$ is fixed
\emph{once and for all}); this, together with the identity $Yu=\mathcal{L}%
u-\sum_{h,k}a_{hk}\partial_{x_{h}x_{k}}^{2}u$ and \eqref{eq:estimomegafg},
immediately implies an analogous bound for the modulus
\[
\omega_{Yu,S_{T}}(\rho)\quad(\text{for $0<\rho\leq r$}).
\]
Finally, when $\rho\geq r$ estimate (ii) is an immediate consequence of (i).
\end{proof}

\section{Time continuity of $\partial_{x_{i}x_{j}}^{2}u$}

Now we have established Theorem \ref{Thm main space}, we are finally ready to
give the

\bigskip

\begin{proof}
[Proof of Theorem \ref{Thm main time}]Let $K,T,\tau,\alpha$ be as in the
statement of the theorem, and let $\psi(t)\in C_{0}^{\infty}(\mathbb{R})$ be a
cut-off function such that
\[
\text{(i)\thinspace\thinspace$0\leq\psi\leq1$ on $\mathbb{R}$},\quad
\text{(ii)\thinspace\thinspace$\psi\equiv1$ on $[\tau,T]$},\quad
\text{(iii)\thinspace\thinspace$\psi(t)=0$ for $t\leq\tau-1$}.
\]
We then fix a point $\overline{\xi}=(\overline{x},\overline{t})\in S_{T}$ and,
for a given function $u\in\mathcal{S}^{D}(S_{T})$ with $\mathcal{L}%
u\in\mathcal{D}_{\log}(S_{T})$, we apply the continuity estimate
\eqref{eq:schauderspacetimeuxixj} in Theorem \ref{Thm local continuity time}
to the function $v:=u\psi\in\mathcal{S}^{0}(\tau-1,T)$ (see property (iii) of
$\psi$): this gives
\begin{equation}%
\begin{split}
&  |\partial_{x_{i}x_{j}}^{2}u(x_{1},t_{1})-\partial_{x_{i}x_{j}}^{2}%
u(x_{2},t_{2})|=|\partial_{x_{i}x_{j}}^{2}v(x_{1},t_{1})-\partial_{x_{i}x_{j}%
}^{2}v(x_{2},t_{2})|\\
&  \qquad\leq c\big\{\mathcal{M}_{\mathcal{L}_{\overline{x}}v,S_{T}%
}\big(c(d((x_{1},t_{1}),(x_{2},t_{2}))+|t_{1}-t_{2}|^{1/q_{N}})\big)\\
&  \qquad\qquad+\mathcal{U}_{\mathcal{L}_{\overline{x}}v,S_{T}}^{\mu}%
(\sqrt{|t_{2}-t_{1}|})\big\},
\end{split}
\label{eq:estimStartFrozenCT}%
\end{equation}
for every couple of points $(x_{1},t_{1}),(x_{2},t_{2})\in K\times\lbrack
\tau,T]$. Owing to \eqref{eq:estimStartFrozenCT}, and ta\-king into account
the definitions of $\mathcal{M}_{\cdot,\,S_{T}}$ and of $\mathcal{U}%
_{\cdot,\,S_{T}}^{\mu}$, to complete the proof we then turn to estimate the
continuity modulus
\[
\omega_{\mathcal{L}_{\overline{x}}v,S_{T}}(r)\quad(\text{for $r>0$}).
\]
First of all, since $\psi$ is independent of $x$, we get $\mathcal{L}%
_{\overline{x}}v=\psi(\mathcal{L}_{\overline{x}}u)-u\,\partial_{t}\psi;$
moreover,
\begin{align*}
\mathcal{L}_{\overline{x}}u &  =\mathcal{L}u+(\mathcal{L}_{\overline{x}%
}-\mathcal{L})u\\
&  =\mathcal{L}u+\sum_{h,k=1}^{m_{0}}\big(a_{hk}(\overline{x},t)-a_{hk}%
(x,t)\big)\partial_{x_{h}x_{k}}^{2}u.
\end{align*}
In view of these facts, and since $\psi$ is \emph{constant w.r.t.\thinspace
$x$}, by repeatedly exploiting \eqref{eq:estimomegafg}, together with
estimates (i)-(ii) in Theorem \ref{Thm main space}, we then obtain%
\begin{align}
\omega_{\mathcal{L}_{\overline{x}}v,S_{T}}\left(  \rho\right)   &  \leq
\omega_{\psi(\mathcal{L}_{\overline{x}}u),S_{T}}\left(  \rho\right)
+\omega_{u\partial_{t}\psi,S_{T}}\left(  \rho\right)  \leq c\left(
\omega_{\mathcal{L}_{\overline{x}}u,S_{T}}\left(  \rho\right)  +\omega
_{u,S_{T}}\left(  \rho\right)  \right)  \nonumber\\
&  \leq c\left\{  \omega_{\mathcal{L}u,S_{T}}\left(  \rho\right)  +2A%
{\displaystyle\sum\limits_{h,k=1}^{m_{0}}}
\omega_{\partial_{x_{h}x_{k}}^{2}u.S_{T}}\left(  \rho\right)  \right.
\nonumber\\
&  +\left.  \omega_{a,S_{T}}\left(  \rho\right)
{\displaystyle\sum\limits_{h,k=1}^{m_{0}}}
\Vert\partial_{x_{h}x_{k}}^{2}u\Vert_{L^{\infty}(S_{T})}+\omega_{u,S_{T}%
}\left(  \rho\right)  \right\}  \nonumber\\
&  \leq c\left\{  \omega_{\mathcal{L}u,S_{T}}\left(  \rho\right)  \right.
\nonumber\\
&  +\left(  \mathcal{M}_{\mathcal{L}u,S_{T}}\left(  c\rho\right)  +\left(
\mathcal{M}_{a,S_{T}}\left(  c\rho\right)  +\rho^{\alpha}\right)  \left(
\left\Vert \mathcal{L}u\right\Vert _{\mathcal{D}(S_{T})}+\left\Vert
u\right\Vert _{L^{\infty}(S_{T})}\right)  \right)  \nonumber\\
&  +\left.  \omega_{a,S_{T}}\left(  \rho\right)  \left(  \left\Vert
\mathcal{L}u\right\Vert _{\mathcal{D}(S_{T})}+\left\Vert u\right\Vert
_{L^{\infty}(S_{T})}\right)  +\omega_{u,S_{T}}\left(  \rho\right)  \right\}
\nonumber\\
&  \text{(since, by definition, $\omega_{a\,S_{T}}\leq\mathcal{M}_{a,S_{T}}$%
)}\nonumber\\
&  \leq c\left\{  \omega_{\mathcal{L}u,S_{T}}\left(  \rho\right)
+\omega_{u,S_{T}}\left(  \rho\right)  +\mathcal{M}_{\mathcal{L}u,S_{T}}\left(
c\rho\right)  \right.  \nonumber\\
&  \left.  +\left(  \left(  \mathcal{M}_{a,S_{T}}\left(  c\rho\right)
+\rho^{\alpha}\right)  \left(  \left\Vert \mathcal{L}u\right\Vert
_{\mathcal{D}(S_{T})}+\left\Vert u\right\Vert _{L^{\infty}(S_{T})}\right)
\right)  \right\}  \label{eq:mainestimomegadexixjLxbar}%
\end{align}
where, as usual, $\omega_{a,S_{T}}=\sum_{h,k=1}^{m_{0}}\omega_{a_{hk},S_{T}}$.
With estimate \eqref{eq:mainestimomegadexixjLxbar} at hand, we can easily
complete the proof of the theorem: indeed, by combining
\eqref{eq:estimStartFrozenCT}-\eqref{eq:mainestimomegadexixjLxbar} and by
taking into account the definitions of the functions involved, we get
\begin{align}
&  |\partial_{x_{i}x_{j}}^{2}u(x_{1},t_{1})-\partial_{x_{i}x_{j}}^{2}%
u(x_{2},t_{2})|\nonumber\\
&  \quad\leq c\big\{\mathcal{M}_{\mathcal{L}u,S_{T}}(cr)+\mathcal{M}_{u,S_{T}%
}(cr)+\mathcal{N}_{\mathcal{L}u,S_{T}}(cr)\nonumber\\
&  \quad\quad+(\mathcal{N}_{a,S_{T}}(cr)+r^{\alpha})(\Vert\mathcal{L}%
u\Vert_{\mathcal{D}(S_{T})}+\Vert u\Vert_{L^{\infty}(S_{T})}))\nonumber\\
&  \qquad\quad+\mathcal{U}_{\mathcal{L}u,S_{T}}^{\mu}(c\sqrt{|t_{1}-t_{2}%
|})+\mathcal{U}_{u,S_{T}}^{\mu}(c\sqrt{|t_{1}-t_{2}|})+\mathcal{V}%
_{\mathcal{L}u,S_{T}}^{\mu}(c\sqrt{|t_{1}-t_{2}|})\nonumber\\
&  \qquad\qquad+(\mathcal{V}_{a,S_{T}}^{\mu}(c\sqrt{|t_{1}-t_{2}|}%
)+|t_{1}-t_{2}|^{\alpha/2})(\Vert\mathcal{L}u\Vert_{\mathcal{D}(S_{T})}+\Vert
u\Vert_{L^{\infty}(S_{T})})\big\},\label{final ineq}%
\end{align}
where we have set
\[
r:=d\big((x_{1},t_{1}),(x_{2},t_{2})\big)+|t_{1}-t_{2}|^{1/q_{N}}.
\]
Next we note that, on the one hand we have $\omega_{\mathcal{L}u,S_{T}}%
\leq\mathcal{M}_{\mathcal{L}u,S_{T}}\leq\mathcal{N}_{\mathcal{L}u,S_{T}}$,
which implies that $\mathcal{U}_{\mathcal{L}u,S_{T}}^{\mu}\leq\mathcal{V}%
_{\mathcal{L}u,S_{T}}^{\mu}$. On the other hand, since by Theorem
\ref{Thm main space} we know that%
\[
\Vert u\Vert_{C^{\alpha}(S_{T})}\leq c\left\{  \Vert\mathcal{L}u\Vert
_{\mathcal{D}(S_{T})}+\Vert u\Vert_{L^{\infty}(S_{T})}\right\}  ,
\]
by Remark \ref{rem:HolderDini} and Proposition \ref{prop:stimaUf}, we can
write%
\[
\mathcal{M}_{u,S_{T}}(cr)+\mathcal{U}_{u,S_{T}}^{\mu}(c\sqrt{|t_{1}-t_{2}%
|})\leq cr^{\alpha}\left\{  \Vert\mathcal{L}u\Vert_{\mathcal{D}(S_{T})}+\Vert
u\Vert_{L^{\infty}(S_{T})}\right\}  .
\]
Using these facts in (\ref{final ineq}) we obtain the desired \eqref{eq:maincontinuitydexixjtime}.
\end{proof}

\bigskip

\noindent\textbf{Data availability statement.} Data sharing is not applicable
to this article as no datasets were generated or analyzed during the current study.

\noindent\textbf{Conflict of interest statement.} The authors have no
conflicts of interest to declare that are relevant to the content of this article.

\bigskip

\noindent\textbf{Addresses}

\noindent Stefano\thinspace Biagi and Marco\thinspace Bramanti.

\noindent Dipartimento di Matematica, Politecnico di Milano.

\noindent Via Bonardi 9, 20133 Milano, Italy.

\noindent stefano.biagi@polimi.it; marco.bramanti@polimi.it

\bigskip

\noindent Bianca Stroffolini.

\noindent Dipartimento di Matematica e Applicazioni \textquotedblleft Renato
Caccioppoli\textquotedblright.

\noindent Universit\`{a} degli Studi di Napoli \textquotedblleft Federico
II\textquotedblright

\noindent Via Cintia, 80126 Napoli, Italy.

\noindent bstroffo@unina.it
\end{document}